\documentclass[10pt,oneside,reqno]{amsart}
\usepackage{geometry}
\usepackage[utf8]{inputenc}
\usepackage{amstext,latexsym,amsbsy,amsmath,amssymb,amsthm,mathtools,relsize,geometry,enumerate}
\usepackage{hyperref}
\usepackage[shortlabels]{enumitem}

\hypersetup{
  colorlinks   = true, 
  urlcolor     = blue, 
  linkcolor    = red, 
  citecolor   = red 
}
\usepackage{mathtools,enumerate,enumitem}

\usepackage[capitalize,nameinlink,noabbrev]{cleveref}

\usepackage{graphicx}
\usepackage{epstopdf}
\setkeys{Gin}{width=\linewidth,totalheight=\textheight,keepaspectratio}
\graphicspath{{./images/}}

\usepackage{multicol}
\usepackage{booktabs}
\usepackage{mathrsfs}
\usepackage{color}

\newcommand{\nbb}{\mathbb{N}}
\newcommand{\rbb}{\mathbb{R}}

\renewcommand{\L}{\mathcal{L}}

\newcommand{\W}{\mathcal{W}}
\renewcommand{\H}{\mathcal{H}}
\newcommand{\Pcal}{\mathcal{P}}
\newcommand{\B}{\mathcal{B}}
\newcommand{\la}{\langle}
\newcommand{\ra}{\rangle}

\newcommand{\grad}{\nabla}

\newcommand{\rx}{U_0}

\newcommand{\ru}{u_0}

\newcommand{\Hzeroeps}{\mathcal{H}^0_{\mueps}}
\newcommand{\Honeeps}{\mathcal{H}^1_{\mueps}}

\newcommand{\Hieps}{\mathcal{H}^i_{\mueps}}

\newcommand{\Mzeroeps}{M^0_{\mueps}}
\newcommand{\Moneeps}{M^1_{\mueps}}

\newcommand{\Mieps}{M^i_{\mueps}}

\newcommand{\TV}{\text{TV}}

\newcommand{\nt}{\nonumber}

\newcommand{\Phieps}{U^\varepsilon}

\newcommand{\ueps}{u^\varepsilon}
\newcommand{\etaeps}{\eta^\varepsilon}
\newcommand{\Ueps}{U^\varepsilon}

\newcommand{\mueps}{\mu_\varepsilon}
\newcommand{\nueps}{\nu^{\mu_\varepsilon}}
\newcommand{\numu}{\nu^\mu}

\newcommand{\etazero}{\eta^{0,\varepsilon}}

\newcommand{\Law}{\text{Law}}

\newcommand{\Wtilde}{\tilde{W}}
\newcommand{\wtilde}{\tilde{w}}

\newcommand{\Peps}{P^{\mueps}}

\newcommand{\Utilde}{\tilde{U}}
\newcommand{\etatilde}{\tilde{\eta}}
\newcommand{\utilde}{\tilde{u}}
\newcommand{\Psitilde}{\tilde{\Psi}}

\newcommand{\Xtil}{\widetilde{X}}

\newcommand{\ubar}{\bar{z}}
\newcommand{\zeps}{z^{\varepsilon}}

\newcommand{\zetaeps}{\zeta^{\varepsilon}}
\newcommand{\Pzero}{P^{0,\varepsilon}}
\newcommand{\nuzero}{\nu^{0,\varepsilon}}

\newcommand{\Phizero}{U^{0,\varepsilon}}
\newcommand{\De}{\Delta}

\newcommand{\x}{U_0}
\newcommand{\y}{\tilde{U}_0}

\newcommand{\mi}{\wedge}
\newcommand{\D}{\mathcal{D}}
\newcommand{\Tr}{\text{Tr}}

\renewcommand{\d}{\text{d}}

\newcommand{\domain}{\mathcal{O}}

\newcommand{\f}{\varphi}

\newcommand{\E}{\mathbb{E}}
\newcommand{\Ecal}{E}
\renewcommand{\P}{\mathbb{P}}

\newcommand{\textred}[1]{\textcolor{red}{#1}}

\newcommand{\A}{\mathcal{A}}
\newcommand{\M}{\mathcal{M}}
\newcommand{\Z}{\mathcal{Z}}
\newcommand{\T}{\mathbb{T}}
\newcommand{\Tcal}{\mathcal{T}}
\newcommand{\dom}{\text{Dom}}
\newcommand{\nbar}{\bar{n}}
\newcommand{\nhat}{\hat{n}}

\newcommand{\close}{\!\!\!}

\theoremstyle{plain}
\newtheorem{theorem}{Theorem}[section]
\newtheorem{corollary}[theorem]{Corollary}

\newtheorem{lemma}[theorem]{Lemma}

\newtheorem{proposition}[theorem]{Proposition}
\newtheorem{definition}[theorem]{Definition}

\theoremstyle{definition}

\newtheorem{remark}[theorem]{Remark}

\numberwithin{equation}{section}

\title[short memory limit]{The short memory limit for long time statistics in a stochastic Coleman-Gurtin model of heat conduction}

\author{Nathan E.~ Glatt-Holtz$^1$, Vincent R. Martinez$^{2}$ and Hung D.~Nguyen$^3$}

\address{$^1$ Department of Statistics, Indiana University, Bloomington, IN, USA}
\address{$^{2}$ Department of Mathematics \& Statistics, CUNY Hunter College, Department of Mathematics, CUNY Graduate Center, New York, NY, USA}
\address{$^3$ Department of Mathematics, University of Tennessee, Knoxville, TN, USA}

\begin{document}

\begin{abstract}
We consider a class of semi-linear differential Volterra equations with polynomial-type potentials that incorporates the effects of memory {{while being}} subjected to random perturbations via an additive Gaussian noise. Our main study is the long time statistics of the system in the singular regime as the memory kernel collapses to a Dirac function. Specifically, we show that provided that sufficiently many directions in the phase space are stochastically forced, there is a unique invariant probability measure to which the system converges, with respect to a suitable Wasserstein-type topology, and at an exponential rate which is independent of the decay rate of the memory kernel. We then prove the convergence of this unique statistically steady state to the unique invariant probability measure of the classical stochastic reaction-diffusion equation in the zero-memory limit. Consequently, we establish the global-in-time validity of the short memory approximation. 
\end{abstract}

\maketitle

%
\setcounter{tocdepth}{1}
\tableofcontents

\section{Introduction} \label{sec:intro}
Let $\domain\subset \rbb^d$ be a bounded open domain with smooth boundary, {where $d\geq1$}. We are interested in the following semilinear stochastic Volterra equation, which is written in non-dimensional variables:
\begin{equation} \label{eqn:react-diff:K}
\begin{aligned}
&\d u(t)=\kappa\Delta u(t)\d t +(1-\kappa)\int_{0}^{\infty}\close K(s)\Delta u(t-s)\d s\d t+\f(u(t))\d t+Q\d w(t),\\
&u(t)\big{|}_{\partial\domain}=0,\qquad u|_{(-\infty,0]}=u_0.
\end{aligned}
\end{equation}
This equation was introduced to describe the {evolution} of a {scalar} field $u(t)=u(t,x):[0,\infty)\times \domain\to\rbb$, such as heat, in a viscoelastic medium. On the right--hand side of~\eqref{eqn:react-diff:K}, $\f:\rbb\to\rbb$ {represents the potential, typically given as a polynomial nonlinearity that satisfies} certain dissipative conditions, $Q\d w(t)$ is a Gaussian process which is delta correlated (white) in time and {whose spatial correlation is characterized by the operator $Q$},
$\kappa\in(0,1)$ defines the relative contribution of memory terms, and $K:\rbb^+\to\rbb^+$, where $\rbb^+=[0,\infty)$, denotes a memory kernel, which regulates the extent to which the past can affect the present; it is assumed to be a smooth function that satisfies
\begin{equation} \label{cond:K:exponential}
K'+\delta K\leq 0,
\end{equation}
for some constant $\delta>0$.

In the absence of memory effects, that is, when $K\equiv \delta_0$, where $\delta_0$ denotes the Dirac mass at $s=0$, the system \eqref{eqn:react-diff:K} reduces to the classical stochastic reaction-diffusion equation 
\begin{equation} \label{eqn:react-diff}
\begin{aligned}
\d u(t)&=\Delta u(t)\d t  +\f(u(t)){\d t}+Q{\d}w(t),\quad &u(t)\big{|}_{\partial\domain}=0.
\end{aligned}
\end{equation}
This equation is commonly used to model a {scalar field} of {a heat flow by conduction} that is governed by Fick's second law, {in which} the instantaneous rate-of-change of {the internal energy} in a small volume element is proportional to the flux of the internal energy across the volume element's boundary. However, as pointed out in \cite{coleman1967equipresence,coleman1964material,gurtin1968general}, there are cases where it is physically {relevant to account for} the history of the velocity {field as well}. For example, the presence of elasticity in a viscoelastic medium induces a ``memory effect" between the {conduction} and the surrounding molecular bombardment \cite{clement1996some,clement1997white,clement1998white}. In this setting, it is more appropriate to consider~\eqref{eqn:react-diff:K}. There is a vast literature investigating the classical system~\eqref{eqn:react-diff} and its time-asymptotic behavior. In particular, under certain assumptions on the potential and noise, it is known that \eqref{eqn:react-diff} admits a unique invariant probability and exhibits geometric ergodicity (see \cite{cerrai2001second,cerrai2020convergence,
da1996ergodicity,da2014stochastic}). In contrast, much less is known about the statistically stationary states of \eqref{eqn:react-diff:K} and whether or not they resemble those of~\eqref{eqn:react-diff}.

In a recent work of \cite{glatt2024paperII}, it can be shown that all invariant probability measures of \eqref{eqn:react-diff:K} possess regularity properties that are dictated by the nonlinear structures. However, both the uniqueness and the issue of relating the statistically steady states of the system with memory effect to those the system without memory were left open in that work. The main goal of the present article is to address the issue of unique ergodicity of the system \eqref{eqn:react-diff:K} and to establish a strong relationship between \eqref{eqn:react-diff:K} and \eqref{eqn:react-diff} in the regime of small memory parameter, namely, that \eqref{eqn:react-diff:K} can uniformly approximate \eqref{eqn:react-diff} (with respect to a Wasserstein-type distance) on an infinite-time horizon. We do so by studying the short memory limit, in which $K$ converges to $\delta_0$, in both finite-time and infinite-time horizons. {We will now provide a detailed overview of the main mathematical results of the article}.

\subsection{Overview of the main results} \label{sec:intro:results}
As discussed elsewhere in \cite{conti2006singular,dafermos1970asymptotic,
glatt2024paperII}, due to the presence of memory, the process $u(t)$ governed by \eqref{eqn:react-diff:K} is generally not Markovian. This issue is dealt with by lifting the original process into an extended phase space to recover Markovianity. This approach was originally developed in \cite{conti2006singular,dafermos1970asymptotic} for deterministic systems by introducing a ``history variable," $\eta(t,s)$ that is subsequently appended to the original process, $u(t)$, to form a joint process $(u(t),\eta(t,s))$ that is Markovian in an extended phase space. More precisely, we introduce the \emph{integrated past history} of $u(t)$ given by 
\begin{equation} \label{form:eta}
        \eta(t,s)=\begin{cases} \int_0^su(t-r)\d r,& 0<s\leq t\\ \eta_0(s-t)+\int_0^tu(t-r)\d r,&s>t.\end{cases}
\end{equation}
Observe that $\eta$ satisfies the following inhomogeneous transport equation
\begin{align*}
\partial_t\eta(t,s)=-\partial_s\eta(t,s)+u(t),\quad \eta(0,\cdotp)=\eta_0(\cdotp),\quad \eta_0\big{|}_{\partial\mathcal{O}}=0.
\end{align*}
Notice, moreover, that under the assumption that $u(t)\big{|}_{\partial\mathcal{O}}=0$, one has $\eta(t,s)\big{|}_{\partial\mathcal{O}}=0$, for all $s,t\geq0$. To see the role of $\eta$ in \eqref{eqn:react-diff:K}, set 
\begin{align}\label{def:mu:K}
\mu(s):=-K'(s).
\end{align}
Then integrate by parts with respect to $s$ in the integral appearing in \eqref{eqn:react-diff:K} and invoke \eqref{form:eta} so that
\begin{equation} \label{eqn:integration-by-part}
\int_{0}^{\infty}\close K (s)\Delta u(t-s)\d s=\int_{0}^{\infty}\close \mu(s)\Delta\eta(t,s)\d s.
\end{equation}
This results in the following extended system, whose dynamics are Markovian:
\begin{equation} \label{eqn:react-diff:mu:original}
\begin{aligned}
&{\d} u(t)=\kappa\Delta u(t){\d t}+(1-\kappa)\int_0^\infty\close \mu(s)\Delta\eta(t,s)\d s {\d t} +\f(u(t))\d t+Q{\d} w(t),\\
\quad
&\partial_t \eta(t,s)=-\partial_s\eta(t,s)+u(t),\quad  \\
&u(t)\big{|}_{\partial\domain}=0,\quad\eta(t,s)\big{|}_{\partial \domain}=0,\\
u(0)&=u_0 \text{ in }\domain,\quad \eta(0;s)=\eta_0(s).
\end{aligned}
\end{equation}
The precise details of the phase space for the extended variable $(u,\eta)$ are briefly provided in \cref{sec:notation}.

Now recall that when $K$ is (formally) replaced by the Dirac delta function, $\delta_0$, in \eqref{eqn:react-diff:K}, the system \eqref{eqn:react-diff:K} reduces to \eqref{eqn:react-diff}. Intuitively speaking, this corresponds to the situation where past information has become negligible \cite{conti2006singular}. More precisely, given $\varepsilon>0$, define
\begin{equation} \label{form:K_epsilon}
K_\varepsilon(s): = \frac{1}{\varepsilon}K\big(\frac{s}{\varepsilon}\big).
\end{equation}
Assuming that
\begin{equation} \label{eqn:int.K=1}
\int_0^\infty\close K(s)\d s =1,
\end{equation}
then $K_\varepsilon$ also satisfies \eqref{eqn:int.K=1} and \eqref{form:K_epsilon} converges to $\delta_0$ (in the sense of distributions) as $\varepsilon\to0$, namely
\begin{align*}
\int_0^\infty\close K_\varepsilon(s)\f(s)\d s \to \f(0),\quad\text{as }\, \varepsilon\to 0,
\end{align*}
for all test functions $\f\in C^\infty_c([0,\infty))$, where $C^\infty_c([0,\infty))$ denotes the class of smooth functions over $[0,\infty)$ with compact support.  Thus, upon replacing $K$ by $K_\varepsilon$ in \eqref{eqn:react-diff:K} and letting $\varepsilon\to0$, we formally recover \eqref{eqn:react-diff} from \eqref{eqn:react-diff:K}. In other words, the past becomes increasingly irrelevant as memory fades away.

This can be formalized in terms of the augmented system \eqref{eqn:react-diff:mu:original}. For this, we recall \eqref{def:mu:K}  and define
\begin{equation} \label{form:mu_epsilon}
   \mu_\varepsilon(s) = -(K_\varepsilon)'(s)= \frac{1}{\varepsilon^2}\mu\big(\frac{s}{\varepsilon}\big). 
\end{equation}
Our first main result is to prove uniqueness of invariant probability measures for \eqref{eqn:react-diff:mu:original} provided that the noise structure is sufficiently non-degenerate and that the Markovian dynamics exponentially relaxes towards the unique stationary distribution in a manner that is uniform in the memory parameter, $\varepsilon$. For clarity, we make use of the convention that $\eqref{eqn:react-diff:mu:original}_{\varepsilon}$ refers to the system \eqref{eqn:react-diff:mu:original} with $\mueps$ replacing $\mu$.

\begin{theorem} \label{thm:unique:pseudo}
Suppose that $K$ satisfies~\eqref{cond:K:exponential}, that noise acts in sufficiently many directions of the phase space, and that $\f$ grows at most algebraically and is dissipative, that is, $\f$ satisfies for all $x\in \rbb$,
\begin{align*}
    |\f(x)|\le c|x|^{p}+C,
\end{align*}
and 
\begin{align*}
    x\f(x)\le -\tilde{c}|x|^{p+1}+C,
\end{align*}
for some positive constants $c,\tilde{c}, C,p$. Then, given $\mueps$ as in \eqref{form:mu_epsilon}, $\eqref{eqn:react-diff:mu:original}_\varepsilon$ admits exactly one invariant probability measure $\nueps$. Moreover,  $\nueps$ is exponentially attractive for its dynamics with a rate that is \textbf{\textup{uniform}} in $\varepsilon$.
\end{theorem}

For the more precise statement of this result, we refer the reader to \cref{thm:react-diff:epsilon:geometric-ergodicity}. We note that the conditions imposed on the potential, $\f$, in \cref{thm:unique:pseudo} apply to a broad class of polynomial nonlinearities, one of which is the well-known Allen-Cahn cubic potential, $\f(x)=x-x^3$. The uniformity of the convergence rate with respect $\varepsilon$ in \cref{thm:unique:pseudo} is a significant fact. Indeed, with this in hand, we are able to study the short memory limit, $\varepsilon\to 0$. In our last main result, we establish the validity of the memory-less system \eqref{eqn:react-diff} as an approximation for the short memory system \eqref{eqn:react-diff:mu:original} on both finite-time and infinite-time horizons. The result can be summarized as follows. 

\begin{theorem} \label{thm:epsilon->0:pseudo} Let $p$ be the growth rate of $\f$ specified in \cref{thm:unique:pseudo}. Suppose that {$p<\infty$ when $d=1,2$,} that $p\le 5$ when $d= 3$, and that $p\leq1$ when $d\ge 4$. Then
\begin{enumerate}[noitemsep,topsep=0pt,wide=0pt,label=\arabic*.,ref=\theassumption.\arabic*]
    \item Let $(u_0,\eta_0)$ be a random variable satisfying suitable moment bounds. Let $(u^\varepsilon(t;u_0),\eta^\varepsilon(t;\eta_0))$ denote the unique solution of $\eqref{eqn:react-diff:mu:original}_{\varepsilon}$ corresponding to $(u_0,\eta_0)$. Let $u^0(t;u_0)$ denote the unique solution of the memory-less equation \eqref{eqn:react-diff} corresponding to $u_0$. Then for each time $T>0$, there exists a positive constant $C$, independent of $\varepsilon$, such that
\begin{align} \label{lim:epsilon->0:finite-time:pseudo}
\E\|\ueps(t)-u^0(t)\|_{L^2(\domain)}^2\le C\varepsilon^{1/3},\quad\text{for all}\ t\leq T.
\end{align}

    \item Let $\nueps$ denote the unique invariant probability measure of $\eqref{eqn:react-diff:mu:original}_{\varepsilon}$ and $\nu^0$ denote the unique invariant probability measure of \eqref{eqn:react-diff}. Let $(\ueps(t;u_0),\etaeps(t;\eta_0))$ denote the unique solution of \eqref{eqn:react-diff:mu:original} corresponding to initial data, $(u_0,\eta_0)$, distributed as $\nueps$ and let $u^0(t;u_0)$ denote the unique solution of \eqref{eqn:react-diff} corresponding to initial data, $u_0$, distributed as $\nu^0$. Then there exists a positive constant $C$, independent of $\varepsilon$, such that
\begin{align} \label{lim:epsilon->0:measure:pseudo}
\sup_{t\geq0}\big{|}\E \, f(\ueps(t))-\E\, f(u^0(t))\big{|}\le C\varepsilon^{1/12},
\end{align} 
holds for all suitable observables $f$.

    \item ($u^\varepsilon$ arbitrary initial distribution, but deterministic), given any $\lambda\in(0,1/12)$ and suitable observable $f$, it follows that
\begin{align} \label{lim:epsilon->0:test-function:pseudo}
\sup_{t\ge0}\big{|} \E f(\ueps(t))-\E f(u(t))\big{|} \le C\varepsilon^{\lambda},
\end{align}
for some positive constant $C$ independent of $\varepsilon$.
\end{enumerate}
\end{theorem}
We refer the reader to \cref{thm:epsilon->0:solution->limit-solution:finite-time}, \cref{thm:epsilon->0:invariant-measures->limit-measure} and \cref{thm:epsilon->0:sup_t.E[f(u^epsilon(t))-f(u^0(t))]} for the more precise version of \cref{thm:epsilon->0:pseudo}. We remark that the linear growth condition, $p\leq 1$, assumed on $\f$ above appears due to  difficulties that arise in controlling the nonlinearity in higher dimensions. In contrast, {when restricted to dimensions $d=1,2$, $p$ can be arbitrarily chosen whereas} in the physical dimension $d= 3$, our short memory results hold for potentials that grow at most as a fifth--order polynomial. Of course, this range still includes the important case of the Allen--Cahn cubic potential.

\subsection{Previous related literature and methodology of proofs}

Stochastic differential equations with memory were studied as early as in the seminal work \cite{ito1964stationary}. Since then, there have been many works concerning the theory of well-posedness {for} infinite-dimensional systems with memory {such as} \cite{barbu1975nonlinear,barbu52nonlinear,barbu1979existence,
bonaccorsi2004large,bonaccorsi2006infinite,clement1996some,clement1997white}. The existence of statistically steady states in {this} context {was addressed} for several {systems}, such as {the} stochastic Volterra equations \cite{clement1998white}, {the} Navier-Stokes equation \cite{weinan2001gibbsian}, {and the} Ginzburg–Landau equation \cite{weinan2002gibbsian}, while the {issue of uniqueness was studied in} \cite{bakhtin2005stationary,bonaccorsi2012asymptotic,
weinan2002gibbsian,weinan2001gibbsian,
glatt2020generalized,hairer2011asymptotic,
nguyen2022ergodicity}.

Turning to~\eqref{eqn:react-diff:K}, the extended phase space approach for studying~\eqref{eqn:react-diff:K} as manifested in the system ~\eqref{eqn:react-diff:mu:original} was introduced in the work \cite{dafermos1970asymptotic} and later popularized for many partial differential equations (PDEs) \cite{conti2005singular,conti2006singular,gatti2004exponential,
gatti2005memory} as well as stochastic PDEs (SPDEs) \cite{caraballo2007existence,caraballo2008pullback,li2019asymptotic,
liu2017well,liu2019asymptotic,
shangguan2024geometric,shu2020asymptotic}. The advantage of this method is that it allows one rewrite the noiseless counterpart of \eqref{eqn:react-diff:mu:original} as an autonomous system of evolution equations on product spaces (see \eqref{eqn:react-diff:epsilon:mu} below). Consequently, this provides one access to the classical Markovian framework to study statistically invariant structures of~\eqref{eqn:react-diff:mu:original}. It is through this form that we are able to obtain the powerful geometric ergodicity property. 

A crucial property that is leveraged in this article, as well as in many others \cite{bonaccorsi2012asymptotic,clement1997white, ottobre2011asymptotic,pavliotis2014stochastic}, is the assumption that the memory kernel decays exponentially into the past (see \eqref{cond:K:exponential} and \nameref{cond:mu}). For stochastic equations with memory decaying non-exponentially, e.g., sub--exponential or power-law, we refer the reader to \cite{baeumer2015existence, desch2011p,
 glatt2020generalized, nguyen2018small}. We emphasize that in the regime of exponential-type memory kernels that are considered in this article, the results we establish are able to accommodate a wider class of nonlinear potentials $\f$ with minimal growth conditions, and are compatible with results regarding their deterministic counterpart.

In the first main result of the paper, as captured in \cref{thm:unique:pseudo}, we address the issue of uniform mixing rate for \eqref{eqn:react-diff:mu:original}. Under a general set of structural assumptions on the potential and the noise, mainly, that the noise directly forces sufficiently many length scales, we will show that the system \eqref{eqn:react-diff:mu:original} with memory possesses only one invariant probability measure. We remark that the condition  we impose on the {potential allows for the maximum of its derivative to be positive}. Thus, our unique ergodicity result is valid for a large class of {potentials}, $\f$, including the important case of the cubic potential $\f(x)=x-x^3$.  We are furthermore able to obtain geometric ergodicity, {for which a defining characteristic is} that the law of the process governed by \eqref{eqn:react-diff:mu:original} {relaxes} exponentially fast to the unique stationary distribution with an exponential rate independent of memory regime $\varepsilon\to 0$.

Adopting the framework developed in \cite{weinan2002gibbsian,weinan2001gibbsian, hairer2006ergodicity,hairer2008spectral,hairer2011theory}, the proof of uniqueness relies on two crucial ingredients. The first is a suitable distance-like function $d$ that is \emph{contracting} for the Markov semigroup. The second is the $d$--\emph{smallness} of bounded sets to which the dynamics returns infinitely often \cite{butkovsky2020generalized,hairer2006ergodicity,hairer2008spectral,
hairer2011theory,kulik2017ergodic,kulik2015generalized} (see \cref{sec:main-result} and \cref{sec:ergodicity}). These properties will be established via the so-called \emph{generalized asymptotic coupling} approach by choosing appropriate controls that ultimately ensure a legitimate change of measure via Girsanov's Theorem. Then, by combining this with the Lyapunov structure of the system, we will be able to deduce the existence of a spectral gap, from which exponential convergence rates to the unique stationary distribution follows. The reader is referred to \cref{thm:react-diff:epsilon:geometric-ergodicity} for the precise statement. The consequences of a spectral gap are then exploited in a crucial way to study the short memory limit, which we now bring our attention to.

{The last set of main results of this article concern the stability of the long-time behavior of the extended variable $(\ueps(t),\etaeps(t))$ governed by~$\eqref{eqn:react-diff:mu:original}_\varepsilon$}. From the point of view of the homogenization problem for stochastic systems, traditionally one would like to compare the solutions of~\eqref{eqn:react-diff:mu:original} to those of~\eqref{eqn:react-diff} on any finite-time horizon, for instance as~\eqref{lim:epsilon->0:finite-time:pseudo}. Results similar to~\eqref{lim:epsilon->0:finite-time:pseudo} are commonly established in the so-called \emph{small mass limit} for a variety of finite and infinite dimensional settings \cite{birrell2017small,cerrai2006smoluchowski,cerrai2006smoluchowski2,
cerrai2017smoluchowski,cerrai2016smoluchowski,herzog2016small,hottovy2015smoluchowski,
lim2020homogenization,salins2019smoluchowski,shi2021small}. It is also studied in the \textit{small noise limit} \cite{blomker2007multiscale,da2014stochastic,lv2008limiting}, as well as in many other types of limiting situations \cite{foldes2015ergodic,foldes2019large}. On the other hand, results of the kind \eqref{lim:epsilon->0:finite-time:pseudo} for SPDEs with memory appear to be less available. We therefore hope to begin to fill this gap by establishing \cref{thm:epsilon->0:pseudo} (Part 1) and show that on any finite-time window, the solution of~\eqref{eqn:react-diff:mu:original} must converge to that of~\eqref{eqn:react-diff}. The result \eqref{lim:epsilon->0:finite-time:pseudo} ultimately relies on the framework of assumptions developed in \cite{conti2005singular,conti2006singular,gatti2004exponential,gatti2005memory,gatti2005navier} for the deterministic setting. Its proof requires a careful analysis of the memory term as well as the nonlinear term. Indeed, moment bounds for solutions of~\eqref{eqn:react-diff} in higher-order topologies are crucial for maintaining uniformity in the memory parameter. We refer the reader to \cref{thm:epsilon->0:solution->limit-solution:finite-time} for the precise statement and to \cref{sec:epsilon->0:finite-time} for its proof.

Once the short memory limit for~\eqref{eqn:react-diff:mu:original} on finite-time horizons has been settled, we consider the much more difficult problem of establishing it on an infinite-time horizon as captured by~\eqref{lim:epsilon->0:measure:pseudo}. This result {can be regarded as the convergence of the corresponding statistical steady states} in the short memory limit. Indeed, {by considering an appropriate class of observables, $f$, in~\eqref{lim:epsilon->0:measure:pseudo}, it can be shown that~\eqref{lim:epsilon->0:measure:pseudo} is a consequence of the assertion that}
\begin{align} \label{lim:Smoluchowski:measure}
\lim_{\varepsilon\to 0}\W(\pi_1^{-1}\nueps,\nu^0)=0,
\end{align}
where $\nueps$ denotes the statistical steady state of system~$\eqref{eqn:react-diff:mu:original}_{\varepsilon}$, $\pi_1^{-1}\nueps$ denotes the marginal on  $L^2(\domain)$ corresponding to the first component of the extended variable $(u,\eta)$, $\nu^0$ is the statistical steady state of the limiting memoryless system, and $\W$ is a Wasserstein-type distance that is found to be appropriate for the analysis. We note that there are situations where stationary solutions can be explicitly computed. In such systems, limits of the type~\eqref{lim:Smoluchowski:measure} are actually trivial to obtain since the steady state in question happen to be independent of the parameter. This is found to be the case for a gradient system~\cite{cerrai2006smoluchowski}, as a well as a generalized Langevin equation~\cite{nguyen2018small}. However, such a phenomenon appears to be more coincidental than generic since invariant probability measures typically do not present themselves in an explicit fashion. Consequently, the best that one can hope for in general is a statement of the form \eqref{lim:Smoluchowski:measure}.

The difficulties surrounding the proof of \eqref{lim:Smoluchowski:measure} amount to the fact that one has to deal with delicate estimates on solutions with initial random conditions as well as exponential moment bounds with respect to the invariant probability measures. In a recent series of papers \cite{cerrai2020convergence,
foldes2017asymptotic,foldes2019large}, results similar to~\eqref{lim:Smoluchowski:measure} were obtained in a variety of settings for SPDEs by employing an asymptotic coupling framework from~\cite{hairer2006ergodicity,hairer2008spectral} developed for infinite-dimensional settings. While there is a rich literature on proving statements of the form~\eqref{lim:epsilon->0:finite-time:pseudo} and~\eqref{lim:Smoluchowski:measure} for limiting regimes distinct from the short memory limit, there seems to be far less work for evolution systems with memory \cite{nguyen2018small,ottobre2011asymptotic}, particularly systems containing drag terms in convolution such the case presented in~\eqref{eqn:react-diff:mu:original}.

The general approach that we employ to prove~\eqref{lim:Smoluchowski:measure} is motivated by those employed in~\cite{cerrai2020convergence,foldes2017asymptotic,foldes2019large,
glatt2021mixing} tailored to our memory-dependent setting. {In a word, the strategy can be described as follows: we consider the Markovian transition functions $P^{\nueps}_t$ and $P^0_t$ associated with~$\eqref{eqn:react-diff:mu:original}_\varepsilon$ and \eqref{eqn:react-diff}, respectively. First, we reduce the convergence of $\pi_1^{-1}\nueps$ to $\nu^0$ as $\varepsilon\to 0$ to the task of obtaining bounds on $\W\big(\pi_1^{-1}(P^{\nueps}_t)^*\nueps,(P^0_t)^*\nu^0\big)$ at a \textit{fixed} time $t$; we do so by appealing to weak formulation of the Harris
theorem \cite{glatt2021mixing,hairer2008spectral,hairer2011asymptotic}, but suitably adapted to our setting. We then invoke the fact that $\W$ is described in terms of an optimal coupling of measures \cite{villani2008optimal} related to the norm $\|\cdot\|_{L^2(\domain)}$ and obtain further control on $\W\big(\pi_1^{-1}(P^{\nueps}_t)^*\nueps,(P^0_t)^*\nu^0\big)$ in terms of the difference $\|\ueps(t)-u^0(t)\|_{L^2(\domain)}$. As shown in~\cref{sec:epsilon->0:inv-measure}, the argument ultimately reduces to deriving~\eqref{lim:epsilon->0:finite-time:pseudo} for random initial conditions.} To the best of the authors' knowledge, the type of limit~\eqref{lim:Smoluchowski:measure} that we establish here seems to be the first in this direction for stochastic systems with memory. 

Establishing~\eqref{lim:Smoluchowski:measure} is interesting on its own right from the probabilistic point of view, as it essentially implies a stochastic version of the analogous result in the deterministic system (when $Q\equiv0$) regarding the convergence of global attractors with respect to an appropriate Hausdorff distance \cite[Theorem 9.1]{conti2006singular}. Results similar to those established in~\cite{conti2006singular} have also been established in a variety of other dynamical systems with memory \cite{conti2005singular,gatti2005navier,gatti2005memory,gatti2004exponential}.
We refer the reader to the precise versions of~\eqref{lim:Smoluchowski:measure} and \cref{thm:epsilon->0:pseudo} (Part 2) in~\cref{thm:epsilon->0:invariant-measures->limit-measure}.

As a final consequence of the short memory limits~\eqref{lim:epsilon->0:finite-time:pseudo}, \eqref{lim:epsilon->0:measure:pseudo} and~\eqref{lim:Smoluchowski:measure}, we obtain, through \cref{thm:epsilon->0:pseudo} (Part 3), the \textit{global-in-time} validity of solutions corresponding to the memoryless system as an approximation to solutions corresponding to the system possessing a short memory. In particular, provided one starts from sufficiently nice initial conditions and the observable $f$ possesses a suitable Lipschitz property (see \eqref{form:Lipschitz}), we show that \eqref{lim:epsilon->0:test-function:pseudo} holds. This generalizes~\eqref{lim:epsilon->0:finite-time:pseudo} over finite-time windows. The method we employ to prove~\eqref{lim:epsilon->0:test-function:pseudo} draws upon recent works from~\cite{glatt2021mixing} where similar issues were dealt with in the setting of two-dimensional incompressible Navier--Stokes equations. For the appropriate class of observables, $f$, the difference in terms of the observable that appears on the left-hand side of~\eqref{lim:epsilon->0:test-function:pseudo} are a consequence of~\eqref{lim:Smoluchowski:measure} as well as the uniform contraction property of~\eqref{eqn:react-diff:mu:original} toward $\nueps$. Although this might give the impression that there is a restriction on the observables $f$, it turns out that the limit~\eqref{lim:epsilon->0:test-function:pseudo} is actually applicable for a rather broad class observables (see \cref{remark:sup_t.E[f(u^epsilon(t))-f(u^0(t))]} (Part 1)). The precise version of~\cref{thm:epsilon->0:pseudo} (Part 3) is provided in \cref{thm:epsilon->0:sup_t.E[f(u^epsilon(t))-f(u^0(t))]}, while its detailed proof will be supplied in \cref{sec:proof-of-E[u^epsilon-u^0]}.

\subsection{Organization of the paper}
 
The rest of the paper is organized as follows: in \cref{sec:notation}, we {establish} the {precise} functional setting {that we work in}. Particularly, we will {formulate} \eqref{eqn:react-diff:mu:original} as an abstract Cauchy {problem} \eqref{eqn:react-diff:epsilon:mu} on an appropriate product space. In \cref{sec:main-result}, we identify the main assumptions that we make on the memory, the nonlinear potentials and the noise structure. We also {state} our main results in this section, including \cref{thm:react-diff:epsilon:geometric-ergodicity} on uniform geometric ergodicity and \cref{thm:epsilon->0:invariant-measures->limit-measure}--\cref{thm:epsilon->0:sup_t.E[f(u^epsilon(t))-f(u^0(t))]} on the short memory limits. In \cref{sec:moment}, we perform a priori moment bounds on the solutions {that ultimately ensures the uniform geometric ergodicity and the short memory limits}. In \cref{sec:ergodicity}, we prove the first main result of the paper, regarding the uniform geometric ergodicity. In \cref{sec:epsilon->0:inv-measure}, we prove the main result of the paper on the convergence of~\eqref{eqn:react-diff:mu:original} toward~\eqref{eqn:react-diff}. The main ingredients of the arguments are collected in \cref{sec:0-equation} and \cref{sec:nu^(0-epsilon)}. The paper concludes with two appendices, \cref{sec:Wasserstein} and \cref{sec:auxiliary-result}. In \cref{sec:Wasserstein}, we provide useful estimates on Wasserstein distances, that are exploited to establish the short memory limit, while in \cref{sec:auxiliary-result}, we collect several {technical} auxiliary results that are invoked in proving the main results.


\section{Functional Setting}\label{sec:notation}

Given a bounded open set $\domain$ in $\rbb^d$ with smooth boundary, we let $L^p(\domain)$, for $1\leq p\leq\infty$ denote the usual Lebesgue spaces.  In the particular case $p=2$, we let $H=L^2(\domain)$. We denote the {corresponding} inner product and norm in $H$ by $\langle\,\cdot\,,\,\cdot\,\rangle_H$ and $\|\cdot\|_H$, {respectively}.

Let $A$ denote the (negative) Dirichlet Laplacian operator, \textred{$-\De_D$}. It is well-known that there exists a complete orthonormal basis $\{e_k\}_{k\ge 1}$ in $H$ that diagonalizes $A$, i.e., there exists a positive sequence $0<\alpha_1<\alpha_2<\dots$ diverging to infinity such that
\begin{equation} \label{eqn:Ae_k=-alpha.e_k}
Ae_k=\alpha_k e_k, \quad k\ge 1.
\end{equation}

More generally, for each $r\in\rbb$, we denote by $H^r$, the domain of $A^{r/2}$ endowed with the inner product, i.e., $H^r:=D(A^{r/2})$, (see \cite{cerrai2020convergence,conti2006singular}). Then the corresponding inner product is defined by
$$\langle u,v\rangle_{H^r}=\sum_{k\ge 1}\alpha_k^{r}\la u,e_k\ra_H\la v,e_k\ra_H,$$
so that the corresponding induced norm is given by
$$\|u\|_{H^r}^2=\sum_{k\geq 1}\alpha_k^{r}\langle u,e_k\rangle^2_H.$$

Next, we {introduce} the notion of extended phase spaces {that was developed in} \cite{
conti2005singular,conti2006singular,
dafermos1970asymptotic}. {This will establish the formal framework in which we construct solutions to} \eqref{eqn:react-diff:mu:original}. First, letting $\mueps:[0,\infty)\to[0,\infty)$ be defined in \eqref{form:mu_epsilon}, we define the following weighted Hilbert spaces
\begin{equation}
M^\beta_{\mueps}=L^2_{\mueps}([0,\infty);H^{\beta+1}), \qquad \beta\in\rbb,
\end{equation}
endowed with the inner product (recalling from \eqref{cond:K:exponential}, \eqref{def:mu:K} and \eqref{form:mu_epsilon} that indeed $\mueps(s)>0$)
\begin{align}\label{def:Mnorm}
\langle \eta_1,\eta_2\rangle_{M^\beta_{\mueps}}=\int_0^\infty\close \mueps(s)\langle A^{(1+\beta)/2}\eta_1(s),A^{(1+\beta)/2}\eta_2(s)\rangle_H\d s.
\end{align}
It is important to note that while the usual embedding $H^{\beta_1}\subset H^{\beta_2}$, $\beta_1>\beta_2$, is compact, the embedding $M^{\beta_1}_{\mueps}\subset M^{\beta_2}_{\mueps}$ is only continuous \cite{pata2001attractors}. In order to establish sufficient regularity of invariant probability measures, we will work with additional spaces, which we develop now.

Next, let $\Tcal_{\mueps}$ be the operator on $\Mzeroeps$ defined by
\begin{align} \label{form:Tcal}
\Tcal_{\mueps}\eta := -\partial_s\eta,\qquad\dom(\Tcal_{\mueps}) =\{\eta\in\Mzeroeps:\partial_s\eta\in \Mzeroeps,\eta(0)=0\},
\end{align}
where $\partial_s$ is the derivative in the distribution sense. In other words, $\Tcal_{\mueps}$ is the infinitesimal generator of the right-translation semigroup acting on $\Mzeroeps$ \cite[Theorem 3.1]{grasselli2002uniform}. Furthermore, if $u\in L^1_{\text{loc}}([0,\infty);H^1)$ then the following functional formulation of the Cauchy initial-value problem
\begin{equation} \label{eqn:eta:Cauchy-problem}
\left\{
\begin{aligned}
\frac{\d}{\d t}\eta(t)&=\Tcal_{\mueps}\eta(t)+u(t),\\
\eta(0)&=\eta_0\in\Mzeroeps,
\end{aligned}
\right.
\end{equation}
has a unique solution $\eta\in C([0,\infty);\Mzeroeps)$ with the following representation: \cite{conti2006singular,grasselli2002uniform}
\begin{equation} \label{form:eta(t)-representation}
\eta(t,s;\eta_0)=\begin{cases} \int_0^s u(t-r)\d r,& 0<s\le t,\\
\eta_0(s-t)+\int_0^t u(t-r)\d r,& s>t.
\end{cases}
\end{equation}

Furthermore, given $\eta\in \Mzeroeps$, we introduce the {\it tail function}
\begin{equation}\label{form:tailfunction}
\T_\eta^{\mueps}(r) = \int_{(0,\frac{1}{r})\cup(r,\infty)}\close\close\close\|A^{1/2}\eta(s)\|_H^2\mueps(s)\d s,\qquad r\geq 1.
\end{equation}
Then we define the Banach space for $\beta\in\rbb$
\begin{equation} \label{form:space:L}
\Ecal^\beta_{\mueps} =\{\eta\in M^\beta_{\mueps}:\eta\in\dom(\Tcal_{\mueps}),\quad\sup_{r\geq 1}r\T_\eta^{\mueps}(r)<\infty\},
\end{equation}
with the norm {defined by}
\begin{equation} \label{form:space:L:norm}
\|\eta\|_{\Ecal^{\beta}_{\mueps}}^2 = \|\eta\|_{M^\beta_{\mueps}}^2+\|\Tcal_{\mueps}\eta\|^2_{\Mzeroeps}+\sup_{r\geq 1}r\T_\eta^{\mueps}(r).
\end{equation}

Having introduced these ``memory spaces", we define for $\beta\in\rbb$ the Banach spaces
\begin{align} \label{form:space:H_epsilon}
\H^\beta_{\mueps}&=H^\beta\times M^\beta_{\mueps} = \{(u,\eta):\|(u,\eta)\|^2_{\H^\beta_{\mueps}}= \|u\|^2_{H^\beta}+\|\eta\|^2_{M^\beta_{\mueps}}<\infty\}, \notag \\
\Z^\beta_{\mueps}&=H^\beta\times E^\beta_{\mueps}\\
&=\{(u,\eta):\|(u,\eta)\|^2_{\Z^\beta_{\mueps}}= \|u\|^2_{H^\beta}+\|\eta\|^2_{\Ecal^\beta_{\mueps}} = \|u\|^2_{H^\beta}+\|\eta\|_{M^\beta_{\mueps}}^2+\|\Tcal_{\mueps}\eta\|^2_{\Mzeroeps}+\sup_{r\geq 1}r\T_\eta^{\mueps}(r) <\infty\}.\notag 
\end{align}
where we recall the norm $M^\beta_{\mueps}$ from expression \eqref{def:Mnorm}. It was shown in \cite{gatti2004exponential,joseph1989heat,joseph1990heat} that $\Ecal^\beta_{\mueps}$ is compactly embedded into $M^0_{\mueps}$ for $\beta>0$, so that any bounded set in $\Z^\beta_{\mueps}$ is totally bounded in $\H^\beta_{\mueps}$. Moreover, for $(u,\eta)\in \H^\beta_{\mueps}$ (or $\Z^\beta_{\mueps}$), we denote by $\pi_i,\,i=1,2,$ the projection on marginal spaces, namely
$$\pi_1(u,\eta)=u,\qquad \pi_2(u,\eta)=\eta.$$

Finally, for each $\varepsilon\in(0,1]$, recalling $\mueps$ from \eqref{form:mu_epsilon}, we may now recast $\eqref{eqn:react-diff:mu:original}_{\varepsilon}$ as follows:
\begin{equation} \label{eqn:react-diff:epsilon:mu}
\begin{aligned}
\d \ueps(t)&=-\kappa A \ueps(t)\d t-(1-\kappa)\int_0^\infty\close \mueps(s)A\etaeps(t;s)\d s \d t+\f(\ueps(t))\d t+Q\d  w(t),\\
\frac{\d}{\d t} \etaeps(t)&= \Tcal_{\mueps}\etaeps(t)+\ueps(t), \\
( \ueps(0),\etaeps(0))&=(u_0,\eta_0)\in\Hzeroeps.
\end{aligned}
\end{equation}

\section{Assumptions and Statements of Main Results}\label{sec:main-result}

In this section, we state our main results and detail the various structural assumptions we impose on \eqref{eqn:react-diff:epsilon:mu} for them. In \cref{sect:well-posed}, we briefly review the well-posedness of the system \eqref{eqn:react-diff:epsilon:mu}. In \cref{sect:ergodicity}, we state the uniqueness of the invariant probability and the uniform geometric rate of convergence through \cref{thm:react-diff:epsilon:geometric-ergodicity}. Lastly, in \cref{sect:short-memory}, we {formulate our result concerning}
the validity of the short memory limit for the solutions on any finite time window, for the invariant probability measure as well as for the solutions on infinite time horizon, through \cref{thm:epsilon->0:solution->limit-solution:finite-time}, \cref{thm:epsilon->0:invariant-measures->limit-measure} and \cref{thm:epsilon->0:sup_t.E[f(u^epsilon(t))-f(u^0(t))]}, respectively.

\subsection{Well-posedness}\label{sect:well-posed}

We first start with the condition on the memory kernel $\mu$.
\subsection*{(M1)}\label{cond:mu} Let $\mu\in C^1([0,\infty))$ be a positive function such that 
\begin{equation} \label{ineq:mu}
\mu'+\delta\mu\le 0,
\end{equation}
for some $\delta>0$. 

\begin{remark} \label{rem:mu_epsilon}
    1. Let us recall~\eqref{form:mu_epsilon} from the introduction, namely, $\mueps(s):=\varepsilon^{-2}\mu(s\varepsilon^{-1})$, and observe that~\nameref{cond:mu} implies
\begin{equation} \label{ineq:mu_epsilon}
\mueps'+\frac{\delta}{\varepsilon}\mueps\le 0.
\end{equation} 

2. We remark that the result of uniform geometric ergodicity, cf. Theorem \ref{thm:react-diff:epsilon:geometric-ergodicity}, does not rely on the scaling \eqref{form:mu_epsilon} as long as the kernel $\mu$ satisfies \eqref{ineq:mu}, or $\mu_\varepsilon$ satisfies \eqref{ineq:mu_epsilon}. While this might give the impression that the scaling structures do not affect the analysis on the short memory limit, expression \eqref{form:mu_epsilon} is actually crucial in proving the convergence of the solutions on finite time windows, cf. Theorem \ref{thm:epsilon->0:solution->limit-solution:finite-time}. It also allows us to quantify the convergence rate \eqref{lim:epsilon->0:solutions->limit-solution} as $\varepsilon\to 0$.
\end{remark}

We will also make use of the following useful estimate \cite[Theorem 3.1]{grasselli2002uniform} that will be employed throughout {our analysis below}: for $\eta\in\dom(\Tcal_{\mueps})$, observe that integrating by parts gives
\begin{align}
\la \Tcal_{\mueps}\eta,\eta\ra_{\Mzeroeps}&=-\frac{1}{2}\int_0^\infty\close \mueps(s)\partial_s\|A^{1/2}\eta(s)\|^2_H\d s\notag \\
&=\frac{1}{2}\int_0^\infty\close \mueps'(s)\|A^{1/2}\eta(s)\|^2_H\d s \le -\frac{\delta}{2\varepsilon}\|\eta\|_{\Mzeroeps}^2. \label{ineq:<T_epsilon.eta,eta>}
\end{align}
More generally, for any $\beta\in\rbb$, we have that
\begin{equation} \label{ineq:<T_epsilon.eta,eta>_(M^n)}
\begin{aligned}
\la \Tcal_{\mueps}\eta,\eta\ra_{M^\beta_{\mueps}}\le -\frac{\delta}{2\varepsilon}\|\eta\|_{M^\beta_{\mueps}}^2,
\end{aligned}
\end{equation}
where $M^\beta_\mu$ is defined in \eqref{def:Mnorm}.

{Regarding the noise, we assume that} $w(t)$ is a cylindrical Wiener process on $H$, {whose decomposition is given by}
$$w(t)=\sum_{k\ge 1}e_kB_k(t),$$
where $\{e_k\}_{k\ge 1}$ is the orthonormal basis of $H$ as in \eqref{eqn:Ae_k=-alpha.e_k} and $\{B_k(t)\}_{k\ge 1}$ is a sequence of independent {standard} one-dimensional Brownian motions, each defined on the same stochastic basis $\mathcal{S}=(\Omega, \mathcal{F},\{\mathcal{F}_t\}_{t\ge 0},\P)$ \cite{karatzas2012brownian}. Concerning the linear operator $Q$, we impose the following assumption \cite{bonaccorsi2012asymptotic,cerrai2020convergence,da2014stochastic,glatt2017unique}:
\subsection*{(Q1)}\label{cond:Q}  
$Q:H\to H$ is a symmetric, non-negative, bounded linear map {such that}
\begin{align*}
\Tr(QAQ)<\infty,\quad \text{and}\quad \sup_{x\in\domain}\sum_{k\ge 1}|Qe_k(x)|^2<\infty.
\end{align*}
{In the above, we recall that $\Tr(QAQ)=\sum_{k\ge 1}\la QAQe_k,e_k\ra_H =\sum_{k\ge 1}\|Qe_k\|^2_{H^1}$.}

\begin{remark} We note that condition $ \sup_{x\in\domain}\sum_{k\ge 1}|Qe_k(x)|^2<\infty$ is required for the well--posedness \cite{bonaccorsi2012asymptotic,da2014stochastic}. We do not explicitly make use of this condition for the large--time asymptotic analysis of~\eqref{eqn:react-diff:epsilon:mu}.

\end{remark}

Finally, concerning the {potential, $\f:\rbb\to\rbb$}, we impose the following  conditions:

\subsection*{(P0)}\label{cond:phi} $\f\in C^1$ satisfies $\f(0)=0$.

\subsection*{(P1)}\label{cond:phi:1} There exist positive constants $a_1$ and $p_0>1$ such that for all $x\in\rbb$,
$$ |\f(x)|\le a_1(1+|x|^{p_0}).$$

\subsection*{(P2)}\label{cond:phi:2} There exist positive constants $a_2,a_3$ such that for all $x\in\rbb$, 
$$x\f(x)\le -a_2|x|^{p_0+1}+a_3,$$
where $p_0$ is the same constant from \textbf{(P1)}.

\subsection*{(P3)}\label{cond:phi:3} The derivative $\f'$ satisfies
$$\sup_{x\in\rbb}\f'(x)=:a_\f<\infty.$$

Fixing a stochastic basis $\mathcal{S}=(\Omega, \mathcal{F},\{\mathcal{F}_t\}_{t\ge 0},\P)$, let us now state what we mean by a ``weak solution" of \eqref{eqn:react-diff:epsilon:mu}  (see \cite{glatt2008stochastic}).
\begin{definition} \label{defn:mild-soln}
Given an initial condition $(u_0,\eta_0)\in \Hzeroeps$, a process $\Ueps(\cdot)=\big(\ueps(\cdot),\etaeps(\cdot)\big)$ is called a weak solution of \eqref{eqn:react-diff:epsilon:mu} if $\Ueps(\cdotp)$ is $\mathcal{F}_t$-adapted and there exists $q\geq1$ such that $\P$--a.s. one has
\begin{align*}
\ueps\in C_{w}([0,\infty);H)\cap L^2_{\emph{loc}}([0,\infty);H^1), \qquad \etaeps\in C({[0,\infty)};\Mzeroeps),\qquad \f(\ueps)\in L^q_{\emph{loc}}([0,\infty);L^q(\domain)),
\end{align*}
that is, $\ueps(\cdotp)$ is weakly continuous in $H$ with respect to $t$, and moreover, for $\P$--a.s.  
\begin{align*}
\la \ueps(t),v\ra_H&=\la u_0,v\ra_H  -\kappa\int_0^t\la \ueps(r),v\ra_{H^1} \emph{d} r-(1-\kappa)\int_0^t\la \etaeps(r),v\ra_{\Mzeroeps}+\int_0^t\la \f(\ueps(r)),v\ra_H \emph{d} r\\
&\quad+\int_0^t\la v(r),Q\emph{d} w(r)\ra_H,\\ 
\la \etaeps(t),\etatilde\ra_{\Mzeroeps}&=\la\eta_0,\etatilde\ra_{\Mzeroeps}+  \int_0^t\la \Tcal_{\mueps}\etaeps(r),\etatilde\ra_{\Mzeroeps} \emph{d}r+\int_0^t\la \ueps(r),\etatilde\ra_{\Mzeroeps}\emph{d} r,
\end{align*} 
{holds for all $v\in H^1\cap L^{q'}(\domain)$, where $q'\geq1$ is the H\"older conjugate of $q$, and for all $\etatilde\in \Mzeroeps$.}
\end{definition}

The first result of this note is the following well-posedness result ensuring the existence and uniqueness of weak solutions.
\begin{proposition}{\cite[Theorem 2.3]{glatt2024paperII} } \label{prop:well-posed} 
{Assume that \nameref{cond:mu}, \nameref{cond:Q} and \nameref{cond:phi}--\nameref{cond:phi:3} hold}. Then for all $U_0\in\Hzeroeps$, \eqref{eqn:react-diff:epsilon:mu} admits a unique weak solution $U(\,\cdot\,;\x)$ in the sense of \cref{defn:mild-soln}. Furthermore, the solution $\Ueps(t;U_0)$ is continuous in $\Hzeroeps$ with respect to $U_0$, for all fixed $t\ge 0$, i.e.,
\begin{align*}
\E\|\Ueps(t;U^{n}_0 )-\Ueps(t;U_0)\|^2_{\Hzeroeps}\to 0,\quad\text{as  }n\to \infty,
\end{align*}
whenever $\|U_0^n-U_0\|_{\Hzeroeps}\to 0$ as $n\to\infty$.
\end{proposition}
The method that we employ to construct solutions is the well-known Faedo-Galerkin approximation and can be found in many previous works for SPDE. We refer the reader to \cite[Theorem 2.3]{glatt2024paperII} for a more detailed argument. See also \cite{albeverio2008spde,
caraballo2007existence,caraballo2008pullback,
glatt2008stochastic} for instance.

\subsection{Uniform geometric ergodicity}\label{sect:ergodicity}

Under the assumptions of the well-posedness result \cref{prop:well-posed}, we may define Markov transition probabilities corresponding to the process $U(t;U_0)$ satisfying \eqref{eqn:react-diff:epsilon:mu} given by 
\begin{align*}
P_t^{\mueps}(U_0,A):=\P(\Ueps(t;\x)\in A),
\end{align*}
for each $t\geq 0$, $U_0\in\Hzeroeps$, and Borel sets $A\subseteq \Hzeroeps$. We let $\B_b(\Hzeroeps)$ denote the set of bounded Borel measurable functions defined on $\Hzeroeps$. The Markov semigroup associated to \eqref{eqn:react-diff:epsilon:mu} is the operator $P_t^{\mueps}:\B_b(\Hzeroeps)\to\B_b(\Hzeroeps)$ defined by
\begin{align}\label{form:P_t^epsilon}
P_t^{\mueps} f(U_0)=\E[f(\Ueps(t;\x))], \quad f\in \B_b(\Hzeroeps).
\end{align}
Recall that a probability measure $\nu\in \Pcal r(\Hzeroeps)$ is said to be \emph{invariant} for the semigroup $P_t^{\mueps}$ if for every $f\in \B_b(\Hzeroeps)$
\begin{align*}
\int_{\Hzeroeps} f(U_0) (P_t^{\mueps})^*\nu(\d U_0)=\int_{\Hzeroeps} f(U_0)\nu(\d U_0),
\end{align*}
where $(P_t^{\mueps})^*\nu$ denotes the push-forward measure of $\nu$ by $P_t^{\mueps}$, i.e., 
$$\int_{\Hzeroeps}f(\x)(P^{\mueps}_t)^*\nu(\d \x)=\int_{\Hzeroeps}P^{\mueps}_t f(\x)\nu(\d \x).$$

Next, to derive It\^o's formula for \eqref{eqn:react-diff:epsilon:mu}, let us recast \eqref{eqn:react-diff:epsilon:mu} in a more compact form as 
\begin{align*}
    \d U^\varepsilon = \A U^\varepsilon\d t+ F(U^\varepsilon)\d t + Q\d W,
\end{align*}
where 
\begin{align*}
    \A \begin{pmatrix}
        u\\ \eta
    \end{pmatrix} = \begin{pmatrix}
        -\kappa A u-(1-\kappa)\int_0^\infty\close \mueps(s)A\eta(s)\d s\\
         \Tcal_{\mueps}\eta+u
         \end{pmatrix},\quad F\begin{pmatrix}
             u\\ \eta
         \end{pmatrix} = \begin{pmatrix}
             \f(u) \\ 0
         \end{pmatrix},\quad W = \begin{pmatrix}
             w \\ 0
         \end{pmatrix}.
\end{align*}
In particular, for $U_1=(u_1,v_1)\in \D(\A)\subset \Hzeroeps$ and $U_2=(u_2,\eta_2)\in\Hzeroeps$, it holds that
\begin{align*}
     \la \A U_1,U_2\ra_{\Hzeroeps} & = -\kappa\la  Au_1, u_2\ra_H -(1-\kappa)\int_0^\infty\close \mueps(s)\la A\eta_1(s),u_2\ra_H \d s\\
     &\qquad+ \la \Tcal_{\mueps}\eta_1,\eta_2\ra_{\Mzeroeps}+\la u_1,\eta_2\ra_{\Mzeroeps}.
\end{align*}
Let $D$ denote the Frechet derivative and $D_u,D_{uu},D_\eta$ denote the partial derivatives in $\Hzeroeps$. In view of It\^o's formula from \cite[Theorem 4.32]{da2014stochastic}, for any $g\in C^2(\Hzeroeps)$ with $g=g(u,\eta)$ satisfying
\begin{align*}
\Tr(D_{uu}gQQ^*)<\infty,
\end{align*}
we have
\begin{align*}
    \d g(U^\varepsilon) = \la \A U^\varepsilon+ F(U^\varepsilon), Dg(U^\varepsilon)\ra_{\Hzeroeps}\d t  + \la D g(U^\varepsilon),Q\d W\ra_{\Hzeroeps}+\frac{1}{2}\Tr(D_{uu}gQQ^*)\d t.
\end{align*}
More specifically, 
\begin{align*}
    \d g(\ueps,\etaeps) = \L^{\mueps}g(\ueps,\etaeps)\d t  + \la D_u g(\ueps,\etaeps),Q\d w\ra_{H},
\end{align*}
where $\L^{\mueps}$ denotes the operator defined as
\begin{align}\label{form:L^epsilon}
\L^{\mueps} g(u,\eta)&:=-\kappa\la Au,D_u g\ra_H-(1-\kappa)\int_0^\infty\close \mueps(s) \la  A\eta(s),D_u g\ra_H\d s+\la \f(u),D_u g\ra_H \notag \\
&\qquad +\la \Tcal_{\mueps}\eta,D_\eta g\ra_{\Mzeroeps}+\la u,{D_\eta  g}\ra_{\Mzeroeps}+\frac{1}{2}\Tr(D_{uu}gQQ^*).
\end{align}


We now turn to the topic of uniform geometric ergodicity of \eqref{eqn:react-diff:epsilon:mu}. Following the framework of~\cite{butkovsky2020generalized,cerrai2020convergence,
hairer2006ergodicity,hairer2008spectral,nguyen2022ergodicity},
we recall that a function $d:\Hzeroeps\to[0,\infty)$ is called \emph{distance-like} if it is symmetric, lower semi-continuous, and $d(\x,\y)=0$ if and only if $ \x=\y$; see \cite[Definition 4.3]{hairer2011asymptotic}. Let $\W_{d}$ denote the corresponding coupling distance or Wasserstein-type distance in $Pr(\Hzeroeps)$ associated with $d$, given by
\begin{equation} \label{form:W_d:mu}
\W_{d}(\nu_1,\nu_2) := \inf \E\, d(X,Y),
\end{equation}
where the infimum is taken over all pairs of random variables $(X,Y)$ such that $X\sim \nu_1$ and $Y\sim\nu_2$. 
\begin{remark} \label{rem:W_d:dual}
When $d$ is a distance function in $\Hzeroeps$, the Kantorovich Theorem implies \cite[Theorem 5.10]{villani2008optimal}
\begin{align} \label{form:W_d:mu:dual-Kantorovich}
\W_{d}(\nu_1,\nu_2)=\sup_{[f]_{\text{Lip}_d}\leq1}\big{|}\int_{\Hzeroeps}f(U)\nu_1(\d U)-\int_{\H}f(U)\nu_2(\d U)\big{|},
\end{align}
where
\begin{align} \label{form:Lipschitz}
[f]_{\text{Lip}_{d}}=\sup_{U\neq \tilde{U}}\frac{|f(U)-f(\tilde{U})|}{d(U,\tilde{U})}.
\end{align}
On the other hand, if $d$ is a distance-like function, then the following one-sided inequality holds
\begin{equation} \label{ineq:W_0(nu_1,nu_2):dual}
\W_{d}(\nu_1,\nu_2) \ge \sup_{[f]_{\text{Lip}_{d}}\le 1}\big{|}\int_{\Hzeroeps} f(U)\nu_1(\d U)-\int_{\Hzeroeps} f(U)\nu_2(\d U)\big{|}.
\end{equation}
We refer the reader to \cite[Proposition A.3]{glatt2021mixing} for a further discussion of this point.
\end{remark}

\newcommand{\Ut}{\tilde{U}}

In our setting, we will pay particular attention to the following two distances in $\Hzeroeps$: the first is the discrete metric, i.e.,
$d(U,\tilde{U})=1$ when $U\neq \tilde{U}$ and $d(U,\tilde{U})=0$ otherwise; the corresponding $\W_d$ is the usual total variation distance, denoted by $\W_{\TV}$. The second is the distance $d_N^{\mueps}$, $N>0$, given by \cite{butkovsky2020generalized,hairer2011asymptotic,kulik2017ergodic,
kulik2015generalized,nguyen2022ergodicity}
\begin{equation} \label{form:d_N:mu}
d_N^{\mueps}(U,\tilde{U}):=N\|U-\tilde{U}\|_{\Hzeroeps}\mi 1.
\end{equation}
To estimate the convergence rate of~\eqref{eqn:react-diff:epsilon:mu} toward its unique invariant probability, we introduce the distance-like function
\begin{equation}\label{form:d_N,beta:mu}
d_{N,\beta}^{\mueps}(U,\tilde{U})=\sqrt{d_N^{\mueps}(U,\Ut)\big[1+e^{\beta\Psi_0(U)}+e^{\beta\Psi_0(\Ut)}\big]},
\end{equation}
where the functional $\Psi_0:\Hzeroeps\to \rbb$ is defined as
\begin{equation} \label{form:Psi_0}
\Psi_0(u,\eta)=\frac{1}{2}\|u\|^2_{H}+\frac{1}{2}(1-\kappa)\|\eta\|^2_{\Mzeroeps}.
\end{equation} 

Also, for $n\ge 1$, we denote by $P_n$ the projection of $(u,\eta)$ onto the {subspace spanned by the first $n$ eigenfunctions $e_k$}:
\begin{align} \label{form:P_Nu}
P_nu=\sum_{k=1}^n\la u,e_k\ra_He_k,\quad\text{and}\quad P_n\eta(s) =\sum_{k=1}^n\la \eta(s),e_k\ra_H e_k.
\end{align}

We will make the following additional assumptions on the potential and noise.

\subsection*{(Q3)} \label{cond:Q:ergodicity}  
Let $Q$ be as in \nameref{cond:Q}. For some $\nbar\in\nbb$ such that 
\begin{align*}
\kappa\alpha_{\nbar}>a_\f,
\end{align*}
where $\alpha_{\nbar}$ denotes the $\nbar$--th eigenvalue of $A$ as in~\eqref{eqn:Ae_k=-alpha.e_k} and $a_\f$ is as in~\nameref{cond:phi:3}, there exists $a_Q=a_Q(\nbar)>0$ such that $$\|Qu\|_H\ge a_Q\|P_{\nbar}u\|_H,\qquad u\in H,$$
where $P_{\nbar}$ is the spectral projection as in \eqref{form:P_Nu} for $\nbar$.

\begin{remark}\label{rem:ergodicity}
We note that \nameref{cond:Q:ergodicity} is a standard assumption and can be found in several preceding works \cite{butkovsky2020generalized,
cerrai2020convergence,glatt2021long,
glatt2017unique,nguyen2022ergodicity}. Essentially, we will require that the noise directly forces sufficiently many modes in order to dominate the effect of non-linear potentials.  
\end{remark}

We now state the first main result, concerning the unique ergodicity as well as the uniform geometric convergence rate with respect to the singular parameter $\varepsilon$.
\begin{theorem} \label{thm:react-diff:epsilon:geometric-ergodicity}
Assume that the conditions from \cref{prop:well-posed} hold. Assume additionally that \nameref{cond:Q:ergodicity} holds. Then, there exist $\beta_0>0$, $N_0>0$ independent of $\varepsilon$ such that for all $\beta\leq\beta_0$ and $N\geq N_0$, there exists $T^*=T^*(N,\beta)$ such that
\begin{equation} \label{ineq:react-diff:epsilon:geometric-ergodicity:beta/2}
\sup_{\varepsilon\in(0,1]}\sup_{\nu_1\neq \nu_2\in\Pcal r(\Hzeroeps)}\frac{\W_{d_{N,\beta}^{\mueps}}\big((P_t^{\mueps})^*\nu_1,(P_t^{\mueps})^*\nu_2\big)}{\W_{d_{N,\beta/2}^{\mueps}}(\nu_1,\nu_2)}\le C_1e^{-C_2 t},\quad\text{for all}\ t\geq T^*,
\end{equation}
for some positive constants, $C_1, C_2$, holds for all $\varepsilon$, where $\W_{d_{N,\beta}^{\mueps}}$ is the coupling distance~\eqref{form:W_d:mu} associated with distance-like function $d_{N,\beta}^{\mueps}$ as in~\eqref{form:d_N,beta:mu}.

\end{theorem}

\begin{remark}
By definition of $d_{N,\beta}^{\mueps}$, we observe that $d^{\mueps}_{N,\beta/2}(\x,\y)\le d^{\mueps}_{N,\beta}(\x,\y)$.  An immediate consequence of \cref{thm:react-diff:epsilon:geometric-ergodicity} is
\begin{equation} \label{ineq:react-diff:epsilon:geometric-ergodicity}
\sup_{\varepsilon\in(0,1]}\sup_{\nu_1\neq \nu_2\in\Pcal r(\Hzeroeps)}\frac{\W_{d_{N,\beta}^{\mueps}}\big((P_t^{\mueps})^*\nu_1,(P_t^{\mueps})^*\nu_2\big)}{\W_{d_{N,\beta}^{\mueps}}(\nu_1,\nu_2)}\le C_1e^{-C_2 t},\qquad t\ge T^*.
\end{equation}

\end{remark}

The proof of \cref{thm:react-diff:epsilon:geometric-ergodicity} makes use of the framework in \cite{butkovsky2020generalized,
hairer2008spectral,hairer2011theory,hairer2011asymptotic,
kuksin2012mathematics,nguyen2022ergodicity} tailored to our setting. The argument relies on two crucial ingredients: a \emph{contraction} property for $P_t^{\mueps}$ with respect to  $d_N^{\mueps}$, and the $d_N^{\mueps}-$\emph{smallness} of $P_t^{\mueps}$ in any bounded set of $\Hzeroeps$. Together with a suitable exponential moment bounds, we are able to deduce the convergence rate~\eqref{ineq:react-diff:epsilon:geometric-ergodicity} with respect to $d^{\mueps}_{N,\beta}$. This will be carried out in \cref{sec:ergodicity}. We turn now to the application of~\cref{thm:react-diff:epsilon:geometric-ergodicity} to study the short memory limit.

\subsection{The short memory limit $\varepsilon\to 0$}\label{sect:short-memory}

We now consider the short memory limit and describe the precise sense in which we compare~\eqref{eqn:react-diff:epsilon:mu} with~\eqref{eqn:react-diff}. We recall from~\cref{sec:intro} that upon setting $K_\varepsilon(s):=\frac{1}{\varepsilon}K\big(\frac{s}{\varepsilon}\big)$, together with the normalizing condition~\eqref{eqn:int.K=1}, one may formally see how~\eqref{eqn:react-diff:K} approximates~\eqref{eqn:react-diff}. Translating this heuristic argument to the system~\eqref{eqn:react-diff:epsilon:mu} with $\mu(s)=-K'(s)$, we introduce  the following normalizing condition for $\mu$:
\subsection*{(M2)} \label{cond:mu:int.s.mu(s)=1}
The memory kernel $\mu$ satisfies the following normalizing condition
\begin{equation}  \label{eqn:int.s.mu(s)=1}
\int_0^\infty\close s\mu(s)\d s=1.
\end{equation}

\begin{remark}
 Suppose that $\mu$ satisfies $c:=\int_0^\infty s\mu(s)\d s$, then instead of~\eqref{eqn:react-diff}, the limiting equation would read
\begin{align*}
\d u(t)=-\big(\kappa+(1-\kappa)c\big)Au(t)+\f(u(t))\d t+ Q\d w(t).
\end{align*}
We therefore impose the normalized condition~\nameref{cond:mu:int.s.mu(s)=1} for the sake of convenience and notation clarity only. Indeed, one may easily extend the proof of  of~\cref{thm:epsilon->0:solution->limit-solution:finite-time} to accommodate this case.
\end{remark}

Observe that as a consequence of \cref{thm:react-diff:epsilon:geometric-ergodicity}, for each $\mueps$, there exists a unique invariant probability measure, $\nueps$, for~\eqref{eqn:react-diff:epsilon:mu} with convergence rate~\eqref{ineq:react-diff:epsilon:geometric-ergodicity} holding uniformly with respect to $\varepsilon$. We turn now to stating our first result regarding the short memory limit as $\varepsilon\to 0$ over any finite-time window. To properly state this result, we introduce the following ersatz history variable, $\etazero$, analogous to $\etaeps$, but for the limiting \textit{memoryless} system, defined by
\begin{equation*}
\frac{\d}{\d t}\etazero(t)=\Tcal_{\mueps}\etazero(t)+u^0(t).
\end{equation*}
Then \eqref{eqn:react-diff} can be rewritten as
\begin{equation} \label{eqn:react-diff:eta^0_epsilon}
\begin{aligned}
\d u^0(t)&=-Au^0(t)\d t +\f(u^0(t))\d t+Q\d w(t),\\
\frac{\d}{\d t}\etazero(t)&=\Tcal_{\mueps}\etazero(t)+u^0(t),\\
 (u^0(0),\etazero(0))&=(u_0,\eta_0)\in\Hzeroeps.
\end{aligned}
\end{equation}
Note that in this formulation, both $u^0$ and $(u^0,\eta^{0,\varepsilon})$ are Markov in contrast to \eqref{eqn:react-diff:epsilon:mu}, where Markovianity only of the extended process ($u^\varepsilon,\eta^\varepsilon)$ is guaranteed.

\begin{theorem} \label{thm:epsilon->0:solution->limit-solution:finite-time} Assume the hypotheses of \cref{prop:well-posed} and, in addition, that \nameref{cond:mu:int.s.mu(s)=1} holds. Suppose that $\{\rx^\varepsilon\}_{\varepsilon\in(0,1)} \subset L^2(\Omega;\Hzeroeps)$, where $U_0^\varepsilon=(u_0^\varepsilon, \eta_0^{\varepsilon})$ satisfies
    \begin{align} \label{cond:U^epsilon_0:O(|x|)}
       \limsup_{\varepsilon\to 0}\E\|\rx^\varepsilon\|^2_{\Honeeps}<\infty.
    \end{align}
Let $\Phieps (t;U_0^\varepsilon)=(\ueps (t;u_0^\varepsilon),\etaeps(t;\eta_0^\varepsilon))$ and $\Phizero(t;U_0^\varepsilon)= (u^0(t;u_0^\varepsilon),\etazero(t;\eta_0^\varepsilon))$ be the solutions of~\eqref{eqn:react-diff:epsilon:mu} and~\eqref{eqn:react-diff:eta^0_epsilon} both corresponding to $\rx^\varepsilon$, respectively. Suppose that 

\begin{enumerate}[noitemsep,topsep=0pt,wide=0pt,label=\arabic*.,ref=\theassumption.\arabic*]

\item In dimensions $d\ge 4$, $\f$ satisfies
\begin{align} \label{cond:phi:O(|x|)}
|\f(x)|\le c_\f(|x|+1),\quad \text{for all}\ x\in\rbb,
\end{align}
for some positive constant $c_\f$; or

\item in dimensions $d=1,2$, $\rx^\varepsilon$ satisfies
\begin{align} \label{cond:U^epsilon_0:agmon:d=1-2}
\limsup_{\varepsilon\to 0} \left(\E\|\pi_1\rx^\varepsilon\|^{2\lceil p_0\rceil }_{H^1} +\E\|\pi_1\rx^\varepsilon\|^{4 \lceil p_0\rceil}_{H} \right)<\infty,
\end{align}
where $p_0$ is the constant as in \nameref{cond:phi:1}-\nameref{cond:phi:2}; or

\item in dimensions $d=3$, $\rx^\varepsilon$ and $\f$ satisfy
\begin{align} \label{cond:U^epsilon_0:agmon}
\limsup_{\varepsilon\to 0} \left(\E\|\pi_1\rx^\varepsilon\|^{10}_{H^1} +\E\|\pi_1\rx^\varepsilon\|^{20}_{H} \right)<\infty,
\end{align}
and
\begin{equation} \label{cond:phi:agmon}
|\f(x)|\le c_\f(|x|^5+1),\quad \text{for all}\ x\in\rbb.
\end{equation}
Then there exist positive constants $c, C$ such that
\begin{equation} \label{lim:epsilon->0:solutions->limit-solution}
\sup_{0\le t\le T}\E\|\Phieps(t)-\Phizero(t)\|^2_{\Hzeroeps}\le Ce^{cT}\varepsilon^{1/3},
\end{equation}
for all $T>0$ and $\varepsilon\in(0,1)$.
\end{enumerate}
\end{theorem}

\begin{remark} \label{rem:finite-time}
\cref{thm:epsilon->0:solution->limit-solution:finite-time} can be considered a stochastic analogue of the result~\cite[(5.1)]{conti2006singular} in the deterministic setting, where the convergence rate was shown only to be of the order $O(\varepsilon^{1/4})$. In the stochastic setting considered here, we identify an opportunity to optimize powers in our argument, which leads to the $1/3$ exponent stated in the above result. We refer the reader to \eqref{ineq:gamma=1/3} in  \cref{sec:epsilon->0:finite-time}, where the proof of \cref{thm:epsilon->0:solution->limit-solution:finite-time} is supplied.
\end{remark}

The proof of \cref{thm:epsilon->0:solution->limit-solution:finite-time} will be provided in \cref{sec:epsilon->0:finite-time}. Having established the convergence of $\Phieps(t)$ toward $\Phizero(t)$ in finite time, we will explore the limit behavior of $\nueps$ in suitable Wasserstein distances, which we introduce next. 

Similarly to $d^{\mueps}_{N}$ and $d^{\mueps}_{N,\beta}$ as in~\eqref{form:d_N:mu} and~\eqref{form:d_N,beta:mu}, respectively, we define  
\begin{equation} \label{form:d_0(x,y)}
d_N(u,v):=N\|u-v\|_H \mi 1,
\end{equation}
and
\begin{equation} \label{form:d_N,beta}
d_{N,\beta}(u,v):=\sqrt{d_N(u,v)\left(1+\exp\left({\frac{\beta}{2}\|u\|^2_H }\right)+\exp\left({\frac{\beta}{2}\|v\|^2_H }\right)\right)}.
\end{equation}
The Wasserstein distance in $\Pcal r(H)$ corresponding to $d_{N,\beta}$ is given by
\begin{equation} \label{form:W_0(nu_1,nu_2)}
\W_{d_{N,\beta}}(\nu_1,\nu_2)=\inf \E\,d_{N,\beta}(X,Y),  
\end{equation}
where the infimum is taken over all pairs $(X,Y)$ such that $X\sim \nu_1$ and $Y\sim\nu_2$. Under the hypothesis of \cref{thm:react-diff:epsilon:geometric-ergodicity}, it is well-known that \eqref{eqn:react-diff} admits a unique invariant probability $\nu^0$ in $H$ \cite{da1996ergodicity,da2014stochastic,hairer2011theory}. It will furthermore be shown in \cref{thm:react-diff:geometric-ergodicity} that its dynamics is exponentially attractive toward $\nu^0$ in $\W_{d_{N,\beta}}$ for all $N$ sufficiently large and $\beta$ small enough.

Next, for a probability distribution $\nu\in\Pcal r(\Hzeroeps)$, we denote the marginals of $\nu$ respectively in $u$ and $\eta$ as $\pi_1^{-1}\nu$ and $\pi_2^{-1}\nu$; the marginals are respectively defined as follows: for Borel sets $\B_1\in H$ and $\B_2\in \Mzeroeps$
$$\pi_1^{-1}\nu(\B_1)=\nu(\{(u,\eta)\in\Hzeroeps:u\in \B_1\}),\quad\text{and}\quad \pi_2^{-1}\nu(\B_2)=\nu(\{(u,\eta)\in\Hzeroeps:\eta\in \B_2\}).$$
We recall that the auxiliary variable, $\etaeps$, which encodes the integrated history of the variable $u$ becomes irrelevant as the memory kernel collapses to Dirac mass, that is, in the limit $\varepsilon\rightarrow0$. Thus, the limit that needs to be understood is that of $\ueps\rightarrow u^0$. Naturally, this leads one to compare the marginal $\pi_1^{-1}\nu^\varepsilon$ with $\nu^0$. 

As observed elsewhere \cite{lv2008limiting,nguyen2018small,ottobre2011asymptotic}, we recall that there are situations in which the marginal distribution $\pi_1^{-1}\nu^\varepsilon$ happens to be the same as $\nu^0$ (see discussion following \eqref{lim:Smoluchowski:measure}, for instance). This corresponds to the exceptional situation where the limiting measure can be computed explicitly. In our system and other general settings \cite{cerrai2020convergence,foldes2016ergodicity,
foldes2017asymptotic,foldes2015ergodic,foldes2019large}, explicit expressions are usually difficult to obtain. Nevertheless, under reasonable general conditions, we establish the convergence of $\pi_1^{-1}\nu^\varepsilon$ towards $\nu^0$ in the limit of increasingly negligible dependence on the past.

\begin{theorem} \label{thm:epsilon->0:invariant-measures->limit-measure}
In addition to the hypothesis of \cref{prop:well-posed}, assume that \nameref{cond:Q:ergodicity} and~\nameref{cond:mu:int.s.mu(s)=1} hold, and that $\f$ satisfies either~\eqref{cond:phi:O(|x|)} in dimension $d\ge 4$, or~\eqref{cond:phi:agmon} in dimension $d= 3$. Let $\nueps$ and $\nu^0$ respectively be the unique invariant probability measure of~\eqref{eqn:react-diff:epsilon:mu} and~\eqref{eqn:react-diff}. 
Then, for all $N$ sufficiently large and $\beta$ sufficiently small independent of $\varepsilon$, there exists a positive constant $c_{N,\beta}$ such that the following holds:
\begin{equation} \label{lim:epsilon->0:invariant-measures->limit-measure}
\W_{d_{N,\beta}}\big(\pi_1^{-1}\nueps,\nu^0\big)\le c_{N,\beta}\varepsilon^{1/12}, \quad\text{as}\quad \varepsilon\to 0,
\end{equation}
where $\W_{d_{N,\beta}}$ is the Wasserstein distance defined in~\eqref{form:W_0(nu_1,nu_2)}.
\end{theorem}

In order to prove~\cref{thm:epsilon->0:invariant-measures->limit-measure}, we will adapt the approach developed in~\cite{conti2005singular,conti2006singular} tailored to our setting. Denote by $\Phizero(t)$ and $P^{0,\varepsilon}_t$ the solution and Markov semigroup of~\eqref{eqn:react-diff:eta^0_epsilon}, respectively. Then we show that $P^{0,\varepsilon}$ admits an invariant probability $\nuzero$ (see ~\cref{thm:existence-inv-measure:nu^(0-epsilon)}). By invoking the invariance property of $\nueps$, $\nu^0$ and $\nuzero$ together with expression~\eqref{form:W_d:mu} and~\eqref{form:W_0(nu_1,nu_2)}, $\W_{d_{N,\beta}}\big(\pi_1^{-1}\nueps,\nu^0\big)$ can be shown to be dominated by $\W_{d^{\mueps}_{N,\beta}}( (P^{\mueps}_t)^*\nueps,(P^{0,\varepsilon}_t)^*\nuzero)$, which in turn is controlled by $\E\, d_{N,\beta}^{\mueps} (\Phieps(t;\x^\varepsilon),\Phizero(t;\x^\varepsilon))$, for time $t$ sufficiently large and for random initial condition $\x^\varepsilon$ distributed according to $\nuzero$. More precisely, in the proof of~\cref{thm:epsilon->0:invariant-measures->limit-measure} (in \cref{sec:epsilon->0:large-time}), we will see that for all $t$ sufficiently large, there exists a positive constant $c$ independent of $\varepsilon$ such that
\begin{align*}
\W_{d_{N,\beta}}\big(\pi_1^{-1}\nu^\varepsilon,\nu^0\big)\le c\W_{d^{\mueps}_{N,2\beta}}(\Phieps(t;\rx^\varepsilon),\Phizero(t;\rx^\varepsilon))\le c\E\, d^{\mueps}_{N,2\beta} (\Phieps(t;\rx^\varepsilon),\Phizero(t;\rx^\varepsilon)).
\end{align*}
(see \cref{thm:react-diff:epsilon:geometric-ergodicity}). Furthermore, in~\cref{lem:epsilon->0:alpha_0(solutions-limit-solution)->0} below, the distance $  d^{\mueps}_{N,2\beta} (\Phieps(t;\rx^\varepsilon),\Phizero(t;\rx^\varepsilon))$ can actually be estimated via the usual distance $\|\Phieps(t;\rx^\varepsilon)-\Phizero(t;\rx^\varepsilon))\|_{\Hzeroeps}$. In particular we show, for some appropriate power $q>0$ and constant $c>0$, that
$$\E\,  d^{\mueps}_{N,2\beta}  (\Phieps(t;\rx^\varepsilon),\Phizero(t;\rx^\varepsilon))\le c\,\E\|\Phieps(t;\rx^\varepsilon)-\Phizero(t;\rx^\varepsilon))\|_{\Hzeroeps}^q.$$
This relies crucially on moment bounds on the quantity $\|\Phieps(t;\rx^\varepsilon)-\Phizero(t;\rx^\varepsilon))\|_{\Hzeroeps}$, as well for the invariant probability $\nuzero$.

The main ingredients for the proof of~\cref{thm:epsilon->0:invariant-measures->limit-measure} will be developed in detail in~\cref{sec:epsilon->0:large-time}. The  moment bounds on $\Phieps(t;\rx^\varepsilon)$ and $\Phizero(t;\rx^\varepsilon)$ that are sufficient for our purposes will be collected in a series of lemmas in~\cref{sec:moment},~\cref{sec:0-equation} and~\cref{sec:nu^(0-epsilon)}. More importantly, the crucial estimates on $\nuzero$ will be established in~\cref{sec:nu^(0-epsilon)}. We point out that although the approach we employ is similar to those developed in in~\cite{conti2005singular,conti2006singular}, several arguments that were employed in deterministic settings treated there are not directly applicable to our stochastic setting and, in fact, require a more careful analysis; we refer the reader to 
\cref{rem:epsilon->0:lipschitz}. 

Finally, as a consequence of~\cref{thm:react-diff:epsilon:geometric-ergodicity} and~\cref{thm:epsilon->0:invariant-measures->limit-measure}, we obtain the following result:

\begin{theorem} \label{thm:epsilon->0:sup_t.E[f(u^epsilon(t))-f(u^0(t))]}
Under the same hypothesis of \cref{prop:well-posed}, assume additionally that \nameref{cond:Q:ergodicity} and~\nameref{cond:mu:int.s.mu(s)=1} hold and that $\f$ satisfies either~\eqref{cond:phi:O(|x|)} in dimension $d\ge 4$ or~\eqref{cond:phi:agmon} in dimension $d= 3$. Given $R>0$, let $\{\x^\varepsilon\}_{\varepsilon\in(0,1)}$ be a sequence of deterministic initial conditions such that $\|\x^\varepsilon\|_{\Honeeps}\le R$ for all $\varepsilon\in(0,1)$. Then, for all $N$ sufficiently large and $\beta$ sufficiently small independent of $\varepsilon$, there exists a positive constant $C_{N,\beta,R}$ independent of $\varepsilon$ such that for some $\lambda_*\in (0,1/12)$
\begin{equation} \label{lim:epsilon->0:W_0(P^epsilon_t,P^0_t)}
\sup_{t\ge 0}\W_{d_{N,\beta}}(\pi_1^{-1}(\Peps_t)^*\delta_{\x^\varepsilon},(P_t^0)^*\delta_{\pi_1 \x^\varepsilon})\le C_{N,\beta,R}\varepsilon^{\lambda_*}, \quad\text{as}\quad \varepsilon\to 0,
\end{equation}
where $\W_{d_{N,\beta}}$ is the Wasserstein distance associated with $d_{N,\beta}$ as in~\eqref{form:d_N,beta}. Consequently, for every $f\in C(H;\rbb)$ such that $L_f:=[f]_{\emph{Lip}_{d_{N,\beta}}}<\infty$
\begin{equation} \label{lim:epsilon->0:sup_t.E[f(u^epsilon(t))-f(u^0(t))]}
\sup_{t\ge 0}\, \big{|}\E f(\ueps(t;\x^\varepsilon))-\E f(u^0(t;\pi_1 \x^\varepsilon))\big{|} \le C_{N,\beta,R}\varepsilon^{\lambda_*},\quad\text{as}\quad \varepsilon\to 0.
\end{equation}
\end{theorem}

In order to prove~\cref{thm:epsilon->0:sup_t.E[f(u^epsilon(t))-f(u^0(t))]}, we will make use of the exponential rate~\eqref{ineq:react-diff:epsilon:geometric-ergodicity}, \cref{thm:epsilon->0:invariant-measures->limit-measure} as well as the estimate~\eqref{ineq:W_0(nu_1,nu_2):dual}. The proof of~\cref{thm:epsilon->0:sup_t.E[f(u^epsilon(t))-f(u^0(t))]} will be carried out in~\cref{sec:proof-of-E[u^epsilon-u^0]}.

\begin{remark} \label{remark:sup_t.E[f(u^epsilon(t))-f(u^0(t))]}

\begin{enumerate}[noitemsep,topsep=0pt,wide=0pt,label=\arabic*.,ref=\theassumption.\arabic*]
\item In view of \cite[Proposition A.9]{glatt2021mixing}, a sufficient condition for  a function $f\in C(H;\rbb)$ to be $d_{N,\beta}-$Lipschitz is that
\begin{align*}
\sup_{u\in H}\frac{\max\{|f(u)|,\|Df(u)\|_H\}}{\sqrt{1+e^{\frac{\beta}{2}\|u\|^2_H}}}<\infty.
\end{align*}
In particular, it is not difficult to check that the class of polynomial functions satisfies the above condition.

\item With regard to the hypothesis on $\x^\varepsilon$ in \cref{thm:epsilon->0:sup_t.E[f(u^epsilon(t))-f(u^0(t))]}, we can always find an admissible collection of $\x^\varepsilon$ such that $\pi_2\x^\varepsilon$ is non trivial. To see this, let us fix an initial past trajectory $u_0(\cdot)\in C((-\infty,0];H^2)$ satisfying 
$$\|u_0(-r)\|_{H^2}^2\le c_0e^{\delta_0 r},\quad r\ge r_0,$$
for some $c_0,\delta_0,\,r_0>0$. Then, set 
\begin{align*}
\x^\varepsilon=(u_0(0),\eta_0(\cdot)), \quad \varepsilon\in(0,1].
\end{align*}
We note that this choice of $\x^\varepsilon$ satisfies
\begin{align*}
\limsup_{\varepsilon\to 0}\|\x^\varepsilon\|^2_{\Honeeps}=\|u_0(0)\|^2_{H^1}+\limsup_{\varepsilon\to 0}\|\eta_0\|^2_{\Moneeps}<\infty.
\end{align*}
See \cref{lem:|x^epsilon|<R} for the proof of the last implication above.
\end{enumerate}
\end{remark}



\section{A priori moment estimates} \label{sec:moment}

Throughout the rest of the paper, $c$ and $C$ denote generic positive constants that may change from line to line. The main parameters that they depend on will appear between parenthesis, e.g., $c(T,q)$ is a function of $T$ and $q$. Also, for notational convenience, we will drop the superscript $\varepsilon$ in $\Ueps,\ueps$ and $\etaeps$.

First we state two moment bounds in $\Hzeroeps$ for $U(t)$ through \eqref{ineq:E.int_0^t|A^(1/2)u^epsilon(t)|-bound} and \eqref{ineq:exponential-bound:H^0_epsilon}. The former will be employed to derive estimates in higher regularity in \cref{lem:moment-boud:H^1_epsilon} whereas the latter will appear in the proof of uniform geometric ergodicity in \cref{sec:ergodicity:mixing-rate}. While the proof of \cref{lem:moment-bound:H^0_epsilon} is adapted to that of \cite[Lemma 3.1]{glatt2024paperII}, the key feature of the argument is to keep track of the dependence on $\varepsilon$.

\begin{lemma} \label{lem:moment-bound:H^0_epsilon}
Assume the hypotheses of~\cref{prop:well-posed} and let $U_0=(u_0,\eta_0)\in \Hzeroeps$. Then, 
\begin{align}\label{ineq:E.int_0^t|A^(1/2)u^epsilon(t)|-bound}
&\E\Psi_0(U(t))+\kappa\int_0^t\E\|A^{1/2}u(r)\|^2_{H}+\frac{1}{2}(1-\kappa)\delta\E\|\eta(r)\|^2_{\Mzeroeps}\emph{d} r \notag \\
&\le \E\Psi_0(\rx)+\Big(a_3|\domain|+\frac{1}{2}\emph{Tr}(QQ^*)\Big)t.
\end{align}
Furthermore, for all $\beta$ sufficiently small independent of $\varepsilon$, 
\begin{equation}\label{ineq:exponential-bound:H^0_epsilon}
\E \exp\left({\beta\Psi_0(U(t))}\right)\le e^{-c_0t}\E\exp\left({\beta\Psi_0(U_0)}\right)+C_0,\quad t\ge 0,
\end{equation}
where $\Psi_0$ is as in \eqref{form:Psi_0} and $c_0=c_0(\beta)>0$, $C_0=C_0(\beta)>0$ do not depend on $U_0\in\Hzeroeps$, $t\ge 0$ and $\varepsilon$. 
\end{lemma}

\begin{proof} 

We first start~\eqref{ineq:E.int_0^t|A^(1/2)u^epsilon(t)|-bound} and compute partial derivatives of $\Psi_0$
\begin{align*}
D_u\Psi_0 = u,\quad D_\eta \Psi_0=\eta,\quad \text{and}\quad D_{uu}\Psi_0=Id.
\end{align*}
Recalling $\L^{\mueps}$ as in \eqref{form:L^epsilon}, we have
\begin{align*}
\L^{\mueps}\Psi_0(u,v)& {\ =}-\kappa\|A^{1/2}u\|^2_H-(1-\kappa)\la \eta,u\ra_{\Mzeroeps}+\la \f(u),u\ra_H+\frac{1}{2}\Tr(QQ^*)\\
&\qquad +(1-\kappa)\la \Tcal_{\mueps}\eta,\eta\ra_{\Mzeroeps} +(1-\kappa)\la u,\eta\ra_{\Mzeroeps}\\
&= -\|A^{1/2}u\|^2_H+(1-\kappa)\la \Tcal_{\mueps}\eta,\eta\ra_{\Mzeroeps}+\la \f(u),u\ra_H+\frac{1}{2}\Tr(QQ^*).
\end{align*}
Recalling \nameref{cond:Q}, we readily have $\Tr(QQ^*)<\infty$. In light of \eqref{ineq:<T_epsilon.eta,eta>}, 
\begin{align*}
\la \Tcal_{\mueps}\eta,\eta\ra_{\Mzeroeps}\le -\frac{1}{2}\delta\|\eta\|^2_{\Mzeroeps}.
\end{align*}
 Using \nameref{cond:phi:2}, it holds that
\begin{align*}
\la \f(u),u\ra_H\le a_3|\domain|,
\end{align*}
where $|\domain|$ denotes the {Lebesgue measure} of $\domain$ in $\rbb^d$.
Combining the above estimates, we arrive at
\begin{align} \label{ineq:L^epsilon.Psi_0}
\L^{\mueps}\Psi_0(u,v)&\le -\kappa\|A^{1/2}u\|^2_H-\frac{1}{2}(1-\kappa)\delta\|\eta\|^2_{\Mzeroeps}+a_3|\domain|+\frac{1}{2}\Tr(QQ^*).
\end{align}
This immediately produces \eqref{ineq:E.int_0^t|A^(1/2)u^epsilon(t)|-bound}.

Turning to~\eqref{ineq:exponential-bound:H^0_epsilon}, we consider $g(u,\eta)=\exp\left({\beta\Psi_0(u,\eta)}\right)$. For $\xi\in\Hzeroeps$, the Frechet derivatives of $g$ along the direction of $\xi$ are given by
\begin{align*}
\la D_ug(u,\eta),\pi_1\xi\ra_{H}&=\beta \exp\left({\beta\Psi_0(u,\eta)}\right)\la u,\pi_1\xi\ra_{H},\\
\la D_\eta g(u,\eta),\pi_2\xi\ra_{\Mzeroeps}&=\beta(1-\kappa) \exp\left({\beta\Psi_0(u,\eta)}\right)\la\eta,\pi_2\xi\ra_{\Mzeroeps},\\
\text{and}\quad D_{uu}g(u,\eta)(\xi)&=\beta \exp\left({\beta\Psi_0(u,\eta)}\right)\pi_1\xi+\beta^2\exp\left({\beta\Psi_0(u,\eta)}\right)\la u,\pi_1\xi\ra_{H} u.
\end{align*}
Applying $\L^{\mueps}$ to $g$ gives
\begin{align*}
\L^{\mueps} g(u,\eta)&=\beta e^{\beta\Psi_1(u,\eta)}\Big(-\kappa\|A^{1/2}u\|^2_H+(1-\kappa)\la \Tcal_{\mueps}\eta,\eta\ra_{\Mzeroeps}+\la \f(u),u\ra_H\\
&\qquad\qquad\qquad\qquad+\frac{1}{2}\Tr(QQ^*)+\frac{1}{2}\beta\sum_{k\ge 1}\la u,QQ^*e_k\ra_H\la u,e_k\ra_{H}\Big).
\end{align*}
To estimate the above right hand side, we recall from~\eqref{ineq:<T_epsilon.eta,eta>},~\nameref{cond:Q}, and ~\nameref{cond:phi:2} that
\begin{align*}
\la \Tcal_{\mueps}\eta(t),\eta(t)\ra_{\Mzeroeps}&\le -\frac{\delta}{2}\|\eta(t)\|^2_{\Mzeroeps},\qquad
\la \f(u),u\ra_H\le a_3|\domain|,
\end{align*}
and
\begin{align*}
\frac{1}{2}\beta\sum_{k\ge 1}\la u,QQ^*e_k\ra_H\la u,e_k\ra_{H}=\frac{1}{2}\beta\|Qu\|^2_{H}&=\frac{1}{2}\beta\sum_{k\ge 1}|\la u,Qe_k\ra_H|^2\\
&\le  \frac{1}{2}\beta\Tr(QQ^*)\|u\|^2_H,
\end{align*}
which can be subsumed into $\kappa\|A^{1/2}u\|_{H}^2$ by taking $\beta$ sufficiently small, namely, 
\begin{align*}
\beta<\frac{2\kappa\alpha_1}{\Tr(QQ^*)}.
\end{align*}
We thus combine the above estimates to infer the existence of positive constants $c=c(\beta,\kappa,Q,\f)$ and $C=C(\beta,\kappa,Q,\f)$ such that the following holds uniformly in $t$ and $\varepsilon$.
\begin{align*}
\frac{\d}{\d t}\E \,\exp\left({\beta\Psi_0(U(t))}\right)&\le -c\E \,\exp\left({\beta\Psi_0(U(t))}\right)\Big(\Psi_0(U(t))-C\Big).
\end{align*}
To further bound the above right hand side, we employ the elementary fact that there exists a constant $\tilde{C}=\tilde{C}(\beta,C)>0$ such that for all $r\ge 0$,
\begin{equation} \label{ineq:e^r(r-C)>e^r-C}
e^{\beta r}(r-C)>e^{\beta r}-\tilde{C}.
\end{equation}
So, there exist $c$ and $C$ independent of $t$, $\varepsilon$ and initial condition $U_0$ such that
\begin{equation} \label{ineq:d/dt.E.e^(beta.Psi_1(t))}
\frac{\d}{\d t}\E \,\exp\left({\beta\Psi_0(U(t))}\right)\le - c\,\E \,\exp\left({\beta\Psi_0(U(t))}\right)+C.
\end{equation} 
This establishes~\eqref{ineq:exponential-bound:H^0_epsilon} by virtue of Gronwall's inequality.

\end{proof}

Next, we state and prove~\cref{lem:E.sup.|u^epsilon(t)|-bound} concerning the sup norm in $\Hzeroeps$. This result will be employed later to study the short memory limit in~\cref{sec:epsilon->0:large-time}.

\begin{lemma} \label{lem:E.sup.|u^epsilon(t)|-bound}
Assume the hypotheses of~\cref{prop:well-posed} and let $\rx\in L^2(\Omega;\Hzeroeps)$. Then, for all $T>0$, there exists a constant $c(T)$ independent of $\varepsilon$ such that
\begin{equation} \label{ineq:E.sup.|u^epsilon(t)|-bound}
\E\sup_{0\le t\le T}\|U(t)\|^2_{\Hzeroeps}\le c(T)\big(\E\|\rx\|^2_{\Hzeroeps}+1\big).
\end{equation}
where $\Psi_0(u,\eta)$ is as in~\eqref{form:Psi_0}.
\end{lemma}

\begin{proof} Integrating~\eqref{ineq:d.Psi_0} with respect to time $t$, we observe that
\begin{align*}
\Psi_0(U(t))\le \Psi_0(\rx) +\Big(a_3|\domain|+\frac{1}{2}\Tr(QQ^*)\Big) t+\int_0^t \la u(r),Q\d w(r)\ra_H.
\end{align*} 
In the above, $a_3$ is as in \nameref{cond:phi:2}. By Burkholder's inequality, the Martingale term can be estimated as
\begin{align*}
\E\sup_{r\in[0,t]}\big{|}\int_0^r \la u(\ell),Q\d w(\ell)\ra_H\big{|}\le \Big(\E\int_0^t \|Qu(r)\|^2_H\d r\Big)^{1/2}&\le \frac{1}{2}\Tr(QQ^*)\E\int_0^t\|u(r)\|^2_H\d r+\frac{1}{2}\\
&\le \Tr(QQ^*)\E\int_0^t\Psi_0(U(r))\d r+\frac{1}{2}.
\end{align*}
We therefore arrive at the bound in sup norm
\begin{align*}
\E\sup_{r\in[0,t]}\Psi_0(U(r))\le \E\Psi_0(\rx)+ ct+c\int_0^t\E\sup_{\ell\in [0,r]}\Psi_0(U(\ell))\d r.
\end{align*}
This produces~\eqref{ineq:E.sup.|u^epsilon(t)|-bound} by virtue of Gronwall's inequality.
\end{proof}

We next introduce the function 
\begin{equation} \label{form:Psi_1}
\Psi_1(u,\eta)=\frac{1}{2}\|u\|^2_{H^1}+\frac{1}{2}(1-\kappa)\|\eta\|^2_{\Moneeps}.
\end{equation}
The $\Honeeps$--analog of \cref{lem:moment-bound:H^0_epsilon} is presented below. 

\begin{lemma} \label{lem:moment-boud:H^1_epsilon}
Assume the hypotheses of~\cref{prop:well-posed} and let $U_0=(u_0,\eta_0)\in  L^2(\Omega;\Honeeps)$. Then, for all $t>0$ it holds that
\begin{align}\label{ineq:solution:|Phi|^2_(H^1_epsilon)-bound:random-initial-cond}
&\E\Psi_1(U(t))+\int_0^t\frac{1}{2}\kappa\E\|Au(t)\|^2_H+\frac{1}{2}(1-\kappa)\delta\E\|\eta(t)\|_{\Moneeps}^2\emph{d} t \notag \\
&\le \E\Psi_1(\rx)+\frac{2a_\f^2}{\alpha_1\kappa^2}\E\Psi_0(\rx)+\Big[\Big(a_3|\domain|+\frac{1}{2}\emph{Tr}(QQ^*)\Big)\frac{2a_\f^2}{\alpha_1\kappa^2}+\frac{1}{2}\emph{Tr}(QAQ^*)\Big]t,
\end{align}
where $\Psi_0$ and $\Psi_1$ are respectively as in~\eqref{form:Psi_0} and~\eqref{form:Psi_1}. Furthermore, for all $\beta$ sufficiently small independent of $\varepsilon$, 
\begin{equation}\label{ineq:exponential-bound:H^1_epsilon}
\E \exp\left({\beta\Psi_1(U(t))}\right)\le  e^{-c_{1}t}\exp\left({\beta C_{1}\Psi_1(U_0)}\right)+C_{1},\quad t\ge 0,
\end{equation}
where $\Psi_1(u,\eta)$ is as in \eqref{form:Psi_1} and $c_{1}=c_{1}(\beta)>0$, $C_{1}=C_{1}(\beta)>0$ do not depend on $U_0$, $\varepsilon$ and $t\ge 0$. 

\end{lemma}

\begin{proof} We apply $\L^{\mueps}$, (see \eqref{form:L^epsilon}), to $\Psi_1(u,\eta)$ to obtain
\begin{align*}
\L^{\mueps}\Psi_1(u,\eta)&=  -\kappa\|Au\|^2_{H}+(1-\kappa)\la \Tcal_{\mueps}\eta,\eta\ra_{\Moneeps}+\la \f'(u)\grad u,\grad u\ra_{H}+\frac{1}{2}\Tr(QAQ^*).
\end{align*}
In view of \nameref{cond:Q}, we readily have
\begin{align*}
\Tr(QAQ^*)<\infty.
\end{align*}
Recalling \eqref{ineq:<T_epsilon.eta,eta>_(M^n)} (for $\beta=1$), it holds that
\begin{align*}
\la \Tcal_{\mueps}\eta,\eta\ra_{\Moneeps}\le -\frac{1}{2}\delta\|\eta\|^2_{\Moneeps}.
\end{align*}
To deal with the nonlinear term, note that by \nameref{cond:phi:3} and Cauchy-Schwarz inequality,
\begin{align*}
\la \f'(u)\grad u,\grad u\ra_{H} \le a_\f\la A^{1/2}u,A^{1/2}u\ra_{H}=a_\f \la u,Au\ra_{H}\le \frac{2a_\f^2}{\kappa}\|u\|^2_H+\frac{1}{2}\kappa\|Au\|^2_{H}.
\end{align*}
It follows that
\begin{align}
\L^{\mueps}\Psi_1(u,\eta)&\le  -\frac{1}{2}\kappa\|Au\|^2_{H}-\frac{1}{2}(1-\kappa)\delta\|\eta\|^2_{\Moneeps}+\frac{2a_\f^2}{\kappa}\|u\|^2_H+\frac{1}{2}\Tr(QAQ^*) \notag \\
&\le -c_{1,1}\Psi_1(u,\eta)+C_{1,1}\Psi_0(u,\eta)+C_{1,1}.\label{ineq:L^epsilon.Psi_1}
\end{align}
As a consequence, we see that
\begin{align}
\d \Psi_1(U(t))&= \L^{\mueps}\Psi_1(U(t))\d t+ \la A^{1/2}u(t), A^{1/2}Q\d w(t)\ra_H \notag\\
&\le -\frac{1}{2}\kappa\|Au(t)\|^2_{H}\d t-\frac{1}{2}(1-\kappa)\delta\|\eta(t)\|^2_{\Moneeps}\d t+\frac{2a_\f^2}{\kappa}\|u(t)\|^2_H\d t\notag\\
&\qquad+\frac{1}{2}\Tr(QAQ^*)\d t+ \la A^{1/2}u(t), A^{1/2}Q\d w(t)\ra_H. \label{ineq:d.Psi_1}
\end{align}
Given the random initial condition $\rx$, it is clear that
\begin{align*}
&\E\Psi_1(U(t))+\int_0^t \frac{1}{2}\kappa\E\|Au(r)\|^2_H+\frac{1}{2}(1-\kappa)\delta\E\|\eta(r)\|^2_{\Moneeps}\d r \\
&\le \E\Psi_1(\rx)+\frac{2a_\f^2}{\kappa}\int_0^t\E\|u(r)\|^2_H\d r+\frac{1}{2}t\Tr(QAQ^*).
\end{align*}
In view of~\eqref{ineq:E.int_0^t|A^(1/2)u^epsilon(t)|-bound} combined with Poincare's inequality, we see that
\begin{align*}
\int_0^t\E\|u(r)\|^2_H\d r\le \frac{1}{\alpha_1}\int_0^t\E\|A^{1/2}u(r)\|^2_H\d r\le\frac{1}{\alpha_1\kappa}\Big[\E\Psi_0(\rx) +\Big(a_3|\domain|+\frac{1}{2}\Tr(QQ^*)\Big)t\Big].
\end{align*}
Altogether, we arrive at the estimate~\eqref{ineq:solution:|Phi|^2_(H^1_epsilon)-bound:random-initial-cond}, as claimed.

Turning to~\eqref{ineq:exponential-bound:H^1_epsilon}, for $\beta_{1,0}$ to be chosen later, we consider
\begin{equation} \label{form:g_1(u,eta)}
g_1(u,\eta)=\Psi_1(u,\eta)+\beta_{1,0}\Psi_0(u,\eta).
\end{equation}
In view of~\eqref{ineq:L^epsilon.Psi_0} and~\eqref{ineq:L^epsilon.Psi_1}, observe that
\begin{align*}
\L^{\mueps} g_1(u,\eta)&=  \L^{\mueps}\Psi_1(u,\eta)+\beta_{1,0}\L^{\mueps} \Psi_0(u,\eta)\\
&\le -c_{1,1}\Psi_1(u,\eta)+C_{1,1}\Psi_0(u,\eta)+C_{1,1}\\
&\qquad+\beta_{1,0}\big(-c_{0,1}\Psi_0(u,\eta)+ C_{0,1}\big).
\end{align*}
By picking $\beta_{1,0}$ sufficiently large (independent of $\varepsilon$), we obtain
\begin{align*}
\L^{\mueps} g_1(u,\eta)\le -c g_1(u,\eta)+C.
\end{align*}
Similarly to the proof of~\eqref{ineq:exponential-bound:H^0_epsilon}, we compute
\begin{align*}
\L^{\mueps} e^{\beta g_1(u,\eta)}= \beta e^{\beta g_1(u,\eta)}\L^{\mueps} g_1(u,\eta)+\frac{1}{2}\beta^2e^{\beta g_1(u,\eta)}\sum_{k\ge 1}\big{|} \la u,Qe_k\ra_{H^1}+\beta_{1,0}\la u,Qe_k\ra_H \big{|}^2.
\end{align*}
By~\nameref{cond:Q}, we have
\begin{align*}
\frac{1}{2}\sum_{k\ge 1}\big{|} \la u,Qe_k\ra_{H^1}+\beta_{1,0}\la u,Qe_k\ra_H \big{|}^2&\le \sum_{k\ge 1}|\la u,Qe_k\ra_{H^1}|^2+\sum_{k\ge 1}\beta_{1,0}^2|\la u,Qe_k\ra_H|^2 \\
&\le \Tr(QAQ^*)\|A^{1/2}u\|^2_H+\beta_{1,0}^2\Tr(QQ^*)\|u\|^2_H\\
&\le \Tr(QAQ^*)\Psi_1(u,\eta)+\beta_{1,0}^2\Tr(QQ^*)\Psi_0(u,\eta)\\
&\le C\,g_1(u,\eta),
\end{align*}
for some constant $C>0$ independent of $\beta$ and $\varepsilon$. It follows that
\begin{align*}
\L^{\mueps} e^{\beta g_1(u,\eta)}\le \beta e^{\beta g_1(u,\eta)}\Big(-c g_1(u,\eta)+\beta Cg_1(u,\eta)+C\Big).
\end{align*}
By choosing $\beta_1$ sufficiently small, observe that for all $\beta\in (0,\beta_1)$, we obtain the bound
\begin{align*}
\L^{\mueps} e^{\beta g_1(u,\eta)}\le \beta e^{\beta g_1(u,\eta)}\Big(-c g_1(u,\eta)+C\Big).
\end{align*}
We now invoke~\eqref{ineq:e^r(r-C)>e^r-C} to deduce further
\begin{align*}
\L^{\mueps} e^{\beta g_1(u,\eta)}\le -c e^{\beta g_1(u,\eta)}+C.
\end{align*}
By It\^o's formula, this yields 
\begin{align*}
\frac{\d}{\d t}\E \,e^{\beta g_1(U(t))}\le -c \, \E\, e^{\beta g_1(U(t))}+C,
\end{align*} 
implying
\begin{align*}
\E \,e^{\beta g_1(U(t))} \le e^{-ct}\E \,e^{\beta g_1(U_0)}+C.
\end{align*}
Recalling the expression~\eqref{form:g_1(u,eta)} of $g_1$, 
\begin{align*}
\Psi_1\le g_1\le \left(1+\frac{\beta_{1,0}}{\alpha_1}\right)\Psi_1,
\end{align*}
we immediately obtain~\eqref{ineq:exponential-bound:H^1_epsilon}, as claimed.

\end{proof}

\section{Uniform geometric ergodicity}\label{sec:ergodicity}

\subsection{Uniform mixing rate} \label{sec:ergodicity:mixing-rate}

In this section, we establish the uniqueness of invariant probability $\nueps$ and the exponential rate of convergence of the Markov semigroup $P_t^{\mueps}$ towards $\nueps$. For the convenience of the reader, we recall the notions of $d$--\emph{contracting} and $d$--\emph{small} sets
\cite{butkovsky2020generalized,hairer2011asymptotic}, which are key notions for proving \cref{thm:react-diff:epsilon:geometric-ergodicity}.

\begin{definition} \label{def:contracting:small}
A distance-like function $d^{\mueps}$ bounded by 1 is called \emph{contracting} for $P_t^{\mueps}$ if there exists $\gamma_1<1$ such that for any $U,\Ut\in \Hzeroeps$ with $d^{\mueps}(U,\Ut)<1$, it holds that
\begin{equation} \label{ineq:contracting:W_d<gamma.d}
\W_{d^{\mueps}}(P_t^{\mueps}(U,\cdot),P_t^{\mueps}(\Ut,\cdot))\le \gamma_1\, d^{\mueps}(U,\Ut).
\end{equation}
On the other hand, a set $B\subset\Hzeroeps$ is called $d^{\mueps}$-\emph{small} for $P_t^{\mueps}$ if for some $\gamma_2=\gamma_2(B)>0$, 
\begin{equation} \label{ineq:d-small:W_d<1-epsilon}
\sup_{U,\Ut\in B}\W_{d^{\mueps}}(P_t^{\mueps}(U,\cdot),P_t^{\mueps}(\Ut,\cdot))\le 1-\gamma_2.
\end{equation}
\end{definition}

In~\cref{prop:contracting-d-small} below, we establish that one can always tune $d_N^{\mueps}$ as in~\eqref{form:d_N:mu} such that $d^{\mueps}_N$ is contracting for $P^{\mueps}_t$ and that all bounded sets in $\Hzeroeps$ are $d_N^{\mueps}-$small. 

\begin{proposition} \label{prop:contracting-d-small} Assume that the hypotheses of \cref{prop:well-posed} and Assumption~\nameref{cond:Q:ergodicity} hold. Then

\begin{enumerate}[noitemsep,topsep=0pt,wide=0pt,label=\arabic*.,ref=\theassumption.\arabic*]
\item There exists $N_1$ and $t_1$ sufficiently large, independently of ${\mueps}$, such that for all $N\geq N_1$ and $t\ge t_1$, the distance $d_{N}^{\mueps}(U,\Ut)=(N\|U-\Ut\|^2_{\Hzeroeps})\mi 1$ as in~\eqref{form:d_N:mu} is contracting for $P_t^{\mueps}$ in the sense of~\cref{def:contracting:small}, namely, 
\begin{equation} \label{ineq:contracting:W_d<1/2.d}
\W_{d^{\mueps}_{N}}(P_t^{\mueps}(U,\cdot),P_t^{\mueps}(\Ut,\cdot))\le \frac{1}{2} d^{\mueps}_{N}(U,\Ut),\quad t\ge t_1,
\end{equation}
whenever $d^{\mueps}_N(U,\Ut)<1$.

\item For all $N,\,R>0$, there exists $t_2=t_2(R,N)$ such that for all $t\ge t_2$, the set $\mathcal{D}_R=\{U:\|U\|_{\Hzeroeps}\le R\}$ is $d_{N}^{\mueps}-$small for $P_t^{\mueps}$ in the sense of~\cref{def:contracting:small} with a constant $\gamma_2=\gamma_2(R)\in (0,1)$ independent of $\varepsilon$.
\end{enumerate}
\end{proposition}

Next, we provide an estimate on $\E\, \exp\left({\beta\Psi_0(U(t))}\right)$ which is slightly different form \eqref{ineq:exponential-bound:H^0_epsilon}.  

\begin{lemma} \label{lem:exponential-bound:H^0_epsilon:beta/2}
Let $\x\in\Hzeroeps$. Assume that the conditions from \cref{prop:well-posed} hold. Then for all $\gamma$ sufficiently small, independently of ${\mueps}$, and any $\beta\in \big(0,\gamma(\emph{Tr}(QQ^*))^{-1}\big)$, it holds that
\begin{equation} \label{ineq:exponential-bound:H^0_epsilon:beta/2}
\E\, \exp\left({\beta\Psi_0(U(t))}\right)\le \tilde{C}_0\,\exp\left({\beta\Psi_0(\x){-\gamma t}}\right),\quad \x\in\Hzeroeps,\quad t\ge 0,
\end{equation}
for some positive constant $\tilde{C}_0=\tilde{C}_0(\beta,\gamma)$ independent of $\x,\, t$ and ${\mueps}$. In the above, $\Psi_0$ is as in \eqref{form:Psi_0}.
\end{lemma}

The proof of \cref{prop:contracting-d-small} and \cref{lem:exponential-bound:H^0_epsilon:beta/2} will be deferred to the end of this section. In what follows, we will assume that \cref{prop:contracting-d-small} and \cref{lem:exponential-bound:H^0_epsilon:beta/2} hold and prove \cref{thm:react-diff:epsilon:geometric-ergodicity}; the proof employs the weak-Harris theorem proved in \cite[Theorem 4.8]{hairer2011asymptotic}. We return to the task of proving \cref{prop:contracting-d-small} and \cref{lem:exponential-bound:H^0_epsilon:beta/2} afterwards; of course, these proofs will be independent of \cref{thm:react-diff:epsilon:geometric-ergodicity}.

\begin{proof}[Proof of \cref{thm:react-diff:epsilon:geometric-ergodicity}]
First of all, fix $N$ sufficiently large satisfying \cref{prop:contracting-d-small}, $\beta$ sufficiently small satisfying \eqref{ineq:exponential-bound:H^0_epsilon} and \eqref{ineq:exponential-bound:H^0_epsilon:beta/2}. Leting $U(t)$ be the solution of~\eqref{eqn:react-diff:epsilon:mu} with initial condition $\x$, we invoke the Cauchy-Schwarz inequality, \eqref{ineq:exponential-bound:H^0_epsilon} and \eqref{ineq:exponential-bound:H^0_epsilon:beta/2} to see that
\begin{align*}
\E \exp\left({\beta \Psi_0(U(t))}\right) &\le  \Big(\E \exp\left({\frac{1}{4}\beta \Psi_0(U(t))}\right)\Big)^{1/2}\Big(\E \exp\left({\frac{7}{4}\beta \Psi_0(U(t))}\right)\Big)^{1/2}\\
&\le \Big(e^{-c_0t} \exp\left({\frac{1}{4}\beta \Psi_0(\x)}\right)+C_0\Big)^{1/2}\Big(\tilde{C}_0 \exp\left({\frac{7}{4}\beta \Psi_0(\x)e^{-\gamma t}}\right) \Big)^{1/2},
\end{align*}
where $c_0,C_0$ are as in~\eqref{ineq:exponential-bound:H^0_epsilon}, and $\tilde{C}_0,\gamma$ are as in~\eqref{ineq:exponential-bound:H^0_epsilon:beta/2}. By taking $t$ large enough and using Young's inequality, we may deduce a constant $t_3>0$ independent of ${\mueps}$ and $\x$ such that
\begin{align} \label{ineq:E.e^(beta.Psi_0(U(t)))<e^(beta/2.U_0)}
\E \exp\left({\beta \Psi_0(U(t))}\right) \le \frac{1}{8} \exp\left({\frac{1}{2}\beta \Psi_0(\x)}\right)+C_\beta, \quad t\ge t_3,
\end{align}
for some positive constant $C_\beta$ independent of $\x$, $t$ and ${\mueps}$.

For $\lambda>0$ to be chosen later, we introduce the distance
\begin{align*} 
d^{\mueps}_{N,\beta,\lambda}(U,\Ut)=\sqrt{d^{\mueps}_{N}(U,\Ut)\big(1+\lambda \exp\left({\beta\Psi_0(U)}\right)+\lambda \exp\left({\beta\Psi_0(\Ut)}\right)\big)}.
\end{align*}
We claim that there exist constants $\lambda,\alpha_*\in (0,1)$ and $T^*>0$ independent of $\varepsilon$ such that
\begin{align} \label{ineq:W(P_t(x),P_t(y))<d(x,y)}
\W_{d^{\mueps}_{N,\beta,\lambda}}\big(P^{\mueps}_t(\x,\cdot),P^{\mueps}_t(\y,\cdot)    \big)\le \alpha_* d^{\mueps}_{N,\beta/2,\lambda}(\x,\y),
\end{align}
for all $t\ge T^*$ and $\x,\y\in\Hzeroeps$.

There are three cases depending on the relative positions of $\x,\y$ with respect to $d_N^{\mueps}$: Region 1) $d^{\mueps}_N(\x,\y)<1$, Region 2a) $d^{\mueps}_N(\x,\y)\ge 1$ and $ \exp\left({\tfrac{1}{2}\beta\Psi_0(\x)}\right)+\exp\left({\tfrac{1}{2}\beta\Psi_0(\x)}\right)>2(1+2C_\beta)$, and Region 2b) $d^{\mueps}_N(\x,\y)\ge 1$ and $ \exp\left({\tfrac{1}{2}\beta\Psi_0(\x)}\right)+\exp\left({\tfrac{1}{2}\beta\Psi_0(\x)}\right)\le 2(1+2C_\beta)$.

\noindent\textbf{Region 1}: $\x,\y\in\mathcal{R}_1:=\{d^{\mueps}_N(\x,\y)<1\}$. In this case, we make use of~\eqref{ineq:contracting:W_d<1/2.d},~\eqref{ineq:E.e^(beta.Psi_0(U(t)))<e^(beta/2.U_0)} and the Cauchy-Schwarz inequality to obtain
\begin{align*}
&\big{|}\W_{d^{\mueps}_{N,\beta,\lambda}}\big(P^{\mueps}_t(\x,\cdot),P^{\mueps}_t(\y,\cdot)    \big)\big{|}^2\\
&\le \inf_{\pi}\int_{\Hzeroeps}d^{\mueps}_N(U,U')\pi(\d U,\d U')\int_{\Hzeroeps}\big( 1+\lambda \exp\left({\beta\Psi_0(U)}\right)+\lambda \exp\left({\beta\Psi_0(U')}\right)\big)  \pi(\d U,\d U')\\
&\le \frac{1}{2}d^{\mueps}_N(\x,\y)\left( 1 +\frac{1}{8}\lambda \exp\left({\frac{1}{2}\beta\Psi_0(\x)}\right)+\lambda\frac{1}{8}\exp\left({\frac{1}{2}\beta\Psi_0(\y)}\right)+2\lambda C_\beta\right),
\end{align*}
for $t\ge \max\{t_1,t_3\}$ where $t_1$ and $t_3$ are respectively as in \eqref{ineq:contracting:W_d<1/2.d} and \eqref{ineq:E.e^(beta.Psi_0(U(t)))<e^(beta/2.U_0)}. Note that the infimum runs over all coupling $\pi$ of $\big(P^{\mueps}_t(\x,\cdot),P^{\mueps}_t(\y,\cdot)    \big)$. Choosing $\lambda$ small enough so that
\begin{align*}
\frac{1}{2}(1+2\lambda C_\beta)\le \frac{3}{4}.
\end{align*}
It then follows that
\begin{align*}
\big{|}\W_{d^{\mueps}_{N,\beta,\lambda}}\big(P^{\mueps}_t(\x,\cdot),P^{\mueps}_t(\y,\cdot)    \big)\big{|}^2
&\le \frac{3}{4} d^{\mueps}_N(\x,\y)\left( 1 +\lambda \exp\left({\frac{1}{2}\beta\Psi_0(\x)}\right)+\lambda \exp\left({\frac{1}{2}\beta\Psi_0(\y)}\right)\right)\\
&=\frac{3}{4}d^{\mueps}_{N,\beta/2,\lambda}(\x,\y)^2.
\end{align*}

\noindent \textbf{Region 2a}: $\x,\y\in\mathcal{R}_2^a:=\{d^{\mueps}_N(\x,\y)\ge 1\}\cap\left\{\exp\left({\tfrac{1}{2}\beta\Psi_0(\x)}\right)+\exp\left({\tfrac{1}{2}\beta\Psi_0(\x)}\right)>2(1+2C_\beta)\right\}$. In this case, observe that for $\lambda<\frac{3}{8}$, we have
\begin{align*}
1+2C_\beta\le \left(\frac{7}{8}-\lambda\right)\left(\exp\left({\frac{1}{2}\beta\Psi_0(\x)}\right)+\exp\left({\frac{1}{2}\beta\Psi_0(\y)}\right)\right).
\end{align*}
Hence
\begin{align*}
&1+\frac{1}{8}\lambda \exp\left({\frac{1}{2}\beta\Psi_0(\x)}\right)+\frac{1}{8}\lambda \exp\left({\frac{1}{2}\beta\Psi_0(\y)}\right)+2\lambda C_\beta\\
&\le (1-\lambda)\left(1+\lambda \exp\left({\frac{1}{2}\beta\Psi_0(\x)}\right)+\lambda \exp\left({\frac{1}{2}\beta\Psi_0(\y)}\right)\right)\\
&=(1-\lambda) d^{\mueps}_{N,\beta/2,\lambda}(\x,\y)^2.
\end{align*}
It follows that 
\begin{align*}
\big{|}\W_{d^{\mueps}_{N,\beta,\lambda}}\big(P^{\mueps}_t(\x,\cdot),P^{\mueps}_t(\y,\cdot) \big)\big{|}^2
&\le 1+\frac{1}{8}\lambda \exp\left({\frac{1}{2}\beta\Psi_0(\x)}\right)+\frac{1}{8}\lambda \exp\left({\frac{1}{2}\beta\Psi_0(\y)}\right)+2\lambda C_\beta\\
&\le (1-\lambda) d^{\mueps}_{N,\beta/2,\lambda}(\x,\y)^2,
\end{align*}
for $t\ge t_3$.

\noindent \textbf{Region 2b}: $\x,\y\in\mathcal{R}_2^b:=\{d^{\mueps}_N(\x,\y)\ge 1\}\cap\{\exp\left({\tfrac{1}{2}\beta\Psi_0(\x)}\right)+\exp\left({\tfrac{1}{2}\beta\Psi_0(\x)}\right)\le 2(1+2C_\beta)\}$. In this case, we invoke~\eqref{ineq:E.e^(beta.Psi_0(U(t)))<e^(beta/2.U_0)} and the fact that any bounded set is $d^{\mueps}_N-$small for $P^{\mueps}_t$ (see \cref{prop:contracting-d-small}) to infer the existence of $\gamma_2\in (0,1)$ and a time $t_2>0$ independent of $\x,\y,t,\lambda$ and ${\mueps}$ such that
\begin{align*}
\sup_{\x,\y\in\tilde{\mathcal{R}}^b_2} \W_{d^\mueps}(P_t^{\mueps}(\x,\cdot),P_t^{\mueps}(\y,\cdot))\le 1-\gamma_2,
\end{align*}
for all $t\ge t_2$, where $\tilde{\mathcal{R}}_2^b:=\{\exp\left({\tfrac{1}{2}\beta\Psi_0(\x)}\right)+\exp\left({\tfrac{1}{2}\beta\Psi_0(\x)}\right)\le 2(1+2C_\beta)\}$. As a consequence, we obtain 
\begin{align*}
\big{|}\W_{d^{\mueps}_{N,\beta,\lambda}}\big(P^{\mueps}_t(\x,\cdot),P^{\mueps}_t(\y,\cdot)    \big)\big{|}^2
&\le \inf_{\pi}\int_{\Hzeroeps}d^{\mueps}_N(U,U')\pi(\d U,\d U')\int_{\Hzeroeps} \big(  1+\lambda e^{\beta\Psi_0(U)}+\lambda e^{\beta\Psi_0(U')}  \big) \pi(\d U,\d U')\\
&\le (1-\gamma_2)\left(1 +\frac{1}{8}\lambda \exp\left({\frac{1}{2}\beta\Psi_0(\x)}\right)+\lambda\frac{1}{8}\exp\left({\frac{1}{2}\beta\Psi_0(\y)}\right)+2\lambda C_\beta\right),
\end{align*}
for all $t\ge \max\{t_2,t_3\}$. We then choose $\lambda$ small enough such that
\begin{align*}
(1-\gamma_2)(1+2\lambda C_\beta)\le 1-\frac{1}{2}\gamma_2.
\end{align*}
It follows that
\begin{align*}
\big{|}\W_{d^{\mueps}_{N,\beta,\lambda}}\big(P^{\mueps}_t(\x,\cdot),P^{\mueps}_t(\y,\cdot)    \big)\big{|}^2
&\le \max\{1-\gamma_2/2,1/8  \}\left( 1 +\lambda \exp\left({\frac{1}{2}\beta\Psi_0(\x)}\right)+\lambda \exp\left({\frac{1}{2}\beta\Psi_0(\y)}\right)\right)\\
&=\max\{1-\gamma_2/2,1/8  \}d^{\mueps}_{N,\beta/2,\lambda}(\x,\y)^2.
\end{align*}

\noindent\textbf{Resolution}: In all three cases, we obtain~\eqref{ineq:W(P_t(x),P_t(y))<d(x,y)} for $T^*=\max\{t_1,t_2,t_3\}$, as claimed. Consequently, we arrive at the following bound
\begin{align}\label{ineq:react-diff:epsilon:geometric-ergodicity:lambda}
\W_{d_{N,\beta,\lambda}^{\mueps}}\big((P_t^{{\mueps}})^*\nu_1,(P_t^{{\mueps}})^*\nu_2\big)\le Ce^{-ct}\W_{d_{N,\beta/2,\lambda}^{\mueps}}\big(\nu_1,\nu_2\big),
\end{align}
for all $t\ge T^*$ and $\nu_1,\nu_2\in \Pcal r(\Hzeroeps)$.
Now observe that $d^{\mueps}_{N,\beta,\lambda}$ is equivalent to $d^{\mueps}_{N,\beta}$, that is
\begin{align*}
c\,d^{\mueps}_{N,\beta,\lambda}(U,\Ut)\le d^{\mueps}_{N,\beta}(U,\Ut)\le C d^{\mueps}_{N,\beta}(U,\Ut),
\end{align*}
for some positive constants $c,C$ depending only on $\lambda$, for all $U,\,\Ut\in\Hzeroeps$,  This, together with~\eqref{form:W_d:mu} and~\eqref{ineq:react-diff:epsilon:geometric-ergodicity:lambda}, produces~\eqref{ineq:react-diff:epsilon:geometric-ergodicity:beta/2}.

\end{proof}

We now turn to the proof of \cref{prop:contracting-d-small}. This requires a careful analysis on the two solutions $U(t;\x)$ and $U(t;\y)$ of~\eqref{eqn:react-diff:epsilon:mu}. We will employ a \emph{generalized coupling argument} as developed in \cite{butkovsky2020generalized,kulik2015generalized,nguyen2022ergodicity}, which will be suitable for handling the nonlinearity of the system. We do so by modifying \eqref{eqn:react-diff:epsilon:mu} as follows: First recall that $\alpha_{\nbar}$ is the eigenvalue of $Q$ corresponding to $k=\nbar$ (see \nameref{cond:Q}, \nameref{cond:Q:ergodicity}), {and} $P_{\nbar}$ is the projection of $H$ onto $H_{\nbar}=\text{span}\{e_1,\dots,e_{\nbar}\}$. Then we consider the solution $\Utilde(t;\x,\y)=\big(\utilde(t;\x,\y),\etatilde(t;\x,\y)\big)$ of the system
\begin{equation}\label{eqn:react-diff:mu:pi_n(u_x-u_xy)}
\begin{aligned}
\d\,\utilde(t)&=-\kappa A\utilde(t)\d t-(1-\kappa)\int_0^\infty\close{\mueps}(s)A\etatilde(t;s)\d s\,\d t+\f(\utilde(t))\d t+Q\d w(t)\\
&\quad+\kappa\alpha_{\nbar}P_{\nbar}\big(u(t;\x)-\utilde(t)\big)\d t,\\
\frac{\d}{\d t}\etatilde(t)&=\Tcal_{\mueps}\etatilde(t)+\utilde(t),\\
(\utilde(0),\etatilde(0))&=\y\in\Hzeroeps.
\end{aligned}
\end{equation}

The process $\Utilde(t;\x,\y)$ satisfies an important property, namely, upon making use of condition~\nameref{cond:Q:ergodicity}, it can be shown that $\Utilde(t;\x,\y)$ is driven close to $U(t;\x)$ as $t$ tends to infinity; this phenomenon is summarized in \cref{lem:dissipative-bound}. This property can then be leveraged to show that the deviation in the laws corresponding to $\Utilde(t;\x,\y)$ and $U(t;\y)$ can be effectively controlled by the difference in initial conditions; this is captured in~\cref{lem:error-in-law}.

\begin{lemma} \label{lem:dissipative-bound} 
Under the same hypothesis of~\cref{prop:well-posed}, suppose that Assumption~\nameref{cond:Q:ergodicity} holds. Let $U(t;\x)$ and $\Utilde(t;\x,\y)$ respectively be the solutions of~\eqref{eqn:react-diff:epsilon:mu} and~\eqref{eqn:react-diff:mu:pi_n(u_x-u_xy)} with initial conditions $\x,\,\y\in\Hzeroeps$. Then, 
\begin{equation}\label{ineq:dissipative-bound}
\|U(t;\x)-\Utilde(t;\x,\y)\|_{\Hzeroeps}\le Ce^{-c t} \|\x-\y\|_{\Hzeroeps} ,\quad t>0,
\end{equation}
for some positive constants $c,\,C$ independent of $\x,\,\y$, $t$ and $\varepsilon$.
\end{lemma}

\begin{lemma} \label{lem:error-in-law} Under the same hypothesis of~\cref{prop:well-posed}, suppose that Assumption~\nameref{cond:Q:ergodicity} holds. Let $U(t;\x)$ and $\Utilde(t;\x,\y)$ respectively be the solutions of~\eqref{eqn:react-diff:epsilon:mu} and~\eqref{eqn:react-diff:mu:pi_n(u_x-u_xy)} with initial conditions $\x,\,\y\in\Hzeroeps$. Then,
there exists a positive constant $\gamma_1$ independent of $\varepsilon$ such that
\begin{equation} \label{ineq:error-in-law:1}
\W_{\emph{TV}}(\emph{Law}(\Utilde(t;\x,\y)),P_t^{\mueps}(\y,\,\cdot\,))\le \gamma_1\|\x-\y\|_{\Hzeroeps},\quad \x,\y\in\Hzeroeps,\, t\ge 0.
\end{equation}
Furthermore, for all $R>0$, 
\begin{equation}\label{ineq:error-in-law:2}
\W_\emph{TV}(\emph{Law}(\Utilde(t;\x,\y)),P_t^{\mueps}(\y,\,\cdot\,))\le 1-\gamma_2, \quad
\x,\,\y\in \mathcal{D}_R,\quad t\ge 0
\end{equation}
for some $\gamma_2=\gamma_2(R)\in(0,1)$ independent of $\varepsilon$.
\end{lemma}
Lastly, to conclude the proof of \cref{prop:contracting-d-small}, we will employ the fact that $\W_{\TV}$ dominates $\W_{d_N^{\mueps}}$ as stated in \cref{lem:W_(d_N)<W_TV} below.
\begin{lemma} \label{lem:W_(d_N)<W_TV}
For all probability measures $\nu_1$, $\nu_2$ in $\Pcal r(\Hzeroeps)$, and $N>0$, 
\begin{equation} \label{ineq:W_(d_N)<W_TV}
\W_{d_N^{\mueps}}(\nu_1,\nu_2)<\W_{\emph{TV}}(\nu_1,\nu_2),
\end{equation}
where $\W_{d_N^{\mueps}}$ is the Wasserstein distance associated with $d_N^{\mueps}$ as in~\eqref{form:d_N:mu}.
\end{lemma}

Note that with \cref{lem:dissipative-bound}, \cref{lem:error-in-law}, and \cref{lem:W_(d_N)<W_TV} in hand,   \cref{prop:contracting-d-small} will then follow.
Since the proof of \cref{lem:W_(d_N)<W_TV} is short, we supply it now for the sake of convenience.

\begin{proof}[Proof of \cref{lem:W_(d_N)<W_TV}]
Let $(X,Y)$ be any coupling of $(\nu_1,\nu_2)$. Recall that $d_N^{\mueps}(X,Y)$ is given by~\eqref{form:d_N:mu}. In particular, $d^{\mueps}_N\leq1$. We then see from~\eqref{form:W_d:mu:dual-Kantorovich} that
\begin{align*}
\W_{d^{\mueps}_N}(\nu_1,\nu_2)\le \E\, d_N^{\mueps}(X,Y)& =\E\big[d_N^{\mueps}(X,Y)\mathbf{1}_{\{X\neq Y\}}\big]+\E\big[d_N^{\mueps}(X,Y)\mathbf{1}_{\{X= Y\}}\big]\\
& =\E\big[d_N^{\mueps}(X,Y)\mathbf{1}_{\{X\neq Y\}}\big]\\
&\le \E\big[\mathbf{1}_{\{X\neq Y\}}\big],
\end{align*}
Since the last implication above holds for any such pair $(X,Y)$, we invoke~\eqref{form:W_d:mu:dual-Kantorovich} again to deduce
\begin{align*}
\W_{d_N}(\nu_1,\nu_2) \le \W_{\TV}(\nu_1,\nu_2),
\end{align*}
as desired.
\end{proof}

Assuming that \cref{lem:dissipative-bound} and  \cref{lem:error-in-law} hold, let us now prove \cref{prop:contracting-d-small}. The proofs of \cref{lem:dissipative-bound} and  \cref{lem:error-in-law} will be given at the end of this section. Note that the argument we employ for proving \cref{prop:contracting-d-small} is based on the one presented in \cite[Theorem 2.4]{butkovsky2020generalized}.

\begin{proof}[Proof of \cref{prop:contracting-d-small}]
First we prove 1. For this, let us fix $N$ to be chosen later. Let $\x,\y\in\Hzeroeps$ be such that $d_N^{\mueps}(\x,\y)<1$. Upon recalling~\eqref{form:d_N:mu}, we see that 
\begin{align*}
d^{\mueps}_N(\x,\y)=N\|\x-\y\|_{\Hzeroeps}<1.
\end{align*}
We may thus apply the triangle inequality to deduce
\begin{align} 
&\W_{d^{\mueps}_{N}}\big(P^{\mueps}_t(\x,\cdot),P^{\mueps}_t(\y,\cdot)   \big) \nt \\
&\le \W_{d^{\mueps}_N}\big(P^{\mueps}_t(\x,\cdot),\text{Law}(\Utilde(t;\x,\y))  \big) +\W_{d^{\mueps}_{N}}\big( \text{Law}(\Utilde(t;\x,\y)), P^{\mueps}_t(\y,\cdot)\big),\label{ineq:contracting:a}
\end{align}
where $\Utilde(t;\x,\y)$ is the solution of~\eqref{eqn:react-diff:mu:pi_n(u_x-u_xy)}. To estimate the second term on the right-hand side of~\eqref{ineq:contracting:a}, we invoke~\eqref{ineq:error-in-law:1} and~\eqref{ineq:W_(d_N)<W_TV} to see that
\begin{align*}
\W_{d^{\mueps}_{N}}\big( \text{Law}(\Utilde(t;\x,\y)), P^{\mueps}_t(\y,\cdot)\big)\le \gamma_1\|\x-\y\|_{\Hzeroeps}.
\end{align*}
On the other hand, using the formula~\eqref{form:W_d:mu:dual-Kantorovich} as well as estimate~\eqref{ineq:dissipative-bound}, we see that
\begin{align}
\W_{d^{\mueps}_N}\big(P^{\mueps}_t(\x,\cdot),\text{Law}(\Utilde(t;\x,\y))  \big) &\le \E \big[N\|U(t;\x)-\Utilde(t;\x,\y)\|_{\Hzeroeps} \mi 1  \big]\notag\\
&\le N\E\|U(t;\x)-\Utilde(t;\x,\y)\|_{\Hzeroeps}\notag\\
&\le CNe^{-ct}\|\x-\y\|_{\Hzeroeps}.\label{ineq:contraction:b}
\end{align}
Together with~\eqref{ineq:contracting:a}, we obtain
\begin{align*}
\W_{d^{\mueps}_{N}}\big(P^{\mueps}_t(\x,\cdot),P^{\mueps}_t(\y,\cdot)   \big)&\le CNe^{-ct}\|\x-\y\|_{\Hzeroeps}+\gamma_1\|\x-\y\|_{\Hzeroeps}\\
&=\left(Ce^{-ct}+\frac{\gamma_1}{N}\right)N\|\x-\y\|_{\Hzeroeps}\\
&= \left(Ce^{-ct}+\frac{\gamma_1}{N}\right) d_{N}^{\mueps}(\x,\y).
\end{align*}
In the above, we emphasize that $c,C$ are independent of $N$, $t$, $\x,\,\y$, and $\mueps$. We now choose $t_1$ and $N_1$ sufficiently large such that
\begin{align*}
Ce^{-ct_1}<\frac{1}{4},\qquad \frac{\gamma_1}{N_1}<\frac{1}{4},
\end{align*}
both hold. This implies that
\begin{align*}
\W_{d^{\mueps}_{N}}\big(P^{\mueps}_t(\x,\cdot),P^{\mueps}_t(\y,\cdot)   \big) <\frac{1}{2}d^{\mueps}_{N}(\x,\y),
\end{align*}
for all $t\ge t_1$ and $N>N_1$, which yields~\eqref{ineq:contracting:W_d<1/2.d}.

Now we prove 2. Turning to $d_{N}^{\mueps}-$smallness of $\mathcal{D}_R$, we combine~\eqref{ineq:contracting:a}--\eqref{ineq:contraction:b} with \eqref{ineq:error-in-law:2} and~\eqref{ineq:W_(d_N)<W_TV} to obtain 
\begin{align*}
\W_{d^{\mueps}_{N}}\big(P^{\mueps}_t(\x,\cdot),P^{\mueps}_t(\y,\cdot)   \big)&\le \W_{d^{\mueps}_{N}}\big(P^{\mueps}_t(\x,\cdot),\text{Law}(\Utilde(t;\x,\y))  \big) +\W_{d^{\mueps}_{N}}\big( \text{Law}(\Utilde(t;\x,\y)), P^{\mueps}_t(\y,\cdot)\big)\\
&\le N e^{-ct}\|\x-\y\|_{\Hzeroeps}+1-\gamma_2\\
&\le 2NR e^{-ct}+1-\gamma_2,
\end{align*}
for all $\x,\y\in \mathcal{D}_R$ and $t\ge 0$. We now pick $t_2=t_2(R,N)$ sufficiently large such that
\begin{align*}
2NR e^{-ct_2}<\frac{1}{2}\gamma_2.
\end{align*}
As a consequence
\begin{align*}
\W_{d^{\mueps}_{N}}\big(P^{\mueps}_t(\x,\cdot),P^{\mueps}_t(\y,\cdot)   \big)\le 1-\frac{1}{2}\gamma_2,
\end{align*}
for all $t\ge t_2$ and $\x,\y\in B^{\mueps}_R$. This establishes the $d^{\mueps}_{N}-$smallness of $B^{\mueps}_R$ for $P_t^{\mueps}$, thus completing the proof.
\end{proof}

We now turn to the exponential decay of $\|U(t;\x)-\Utilde(t;\x,\y)\|_{\Hzeroeps}$ and prove \cref{lem:dissipative-bound}. The argument relies crucially on condition~\nameref{cond:Q:ergodicity}.
\begin{proof}[Proof of \cref{lem:dissipative-bound}]
Recalling the systems~\eqref{eqn:react-diff:epsilon:mu} and~\eqref{eqn:react-diff:mu:pi_n(u_x-u_xy)}, we set $\ubar=u-\utilde$ and $\zetaeps=\eta-\etatilde$. Observe that $\big(\ubar(t),\zetaeps(t)\big)$ solves the equation
\begin{align*}
\frac{\d}{\d t}\ubar(t)&=-\kappa A\ubar(t)-(1-\kappa)\int_0^\infty\close{\mueps}(s)A\zetaeps(t;s)\d s+\f(u(t))-\f(\utilde(t))-\kappa\alpha_{\nbar}P_{\nbar}\ubar(t),\\
\frac{\d}{\d t}\zetaeps(t)&=\Tcal_{\mueps}\zetaeps(t)+\ubar(t),\\
(\ubar(0),\zetaeps(0))&=\x-\y\in\Hzeroeps.
\end{align*}
We recall
\eqref{form:Psi_0} and compute
\begin{align*}
\frac{\d}{\d t}\Psi_0(\ubar(t),\zetaeps(t))&= -\kappa\|A^{1/2}\ubar(t)\|^2_H+(1-\kappa)\la \Tcal_{\mueps}\zetaeps(t),\zetaeps(t)\ra_{\Mzeroeps}\\
&\qquad+\la \f(u(t))-\f(\utilde(t)),\ubar(t)\ra_H-\kappa\alpha_{\nbar}\|P_{\nbar}\ubar(t)\|^2_H.
\end{align*}
Since $\varepsilon$, we may invoke \eqref{ineq:<T_epsilon.eta,eta>} to obtain
\begin{align*}
\la \Tcal_{\mueps}\zetaeps(t),\zetaeps(t)\ra_{\Mzeroeps}\le -\frac{1}{2}\delta\|\zetaeps(t)\|^2_{\Mzeroeps}.
\end{align*}
In light of \nameref{cond:phi:3}, it holds that
\begin{align*}
\la \f(u(t))-\f(\utilde(t)),\ubar(t)\ra_H\le a_\f\|\ubar(t)\|^2_H.
\end{align*}
To further bound the above right-hand side, we recall from condition~\nameref{cond:Q:ergodicity} that
$$\kappa\alpha_{\nbar}>a_\f.$$
As a consequence
\begin{align*}
-\kappa\|A^{1/2}\ubar(t)\|^2_H&+\la \f(u(t))-\f(\utilde(t)),\ubar(t)\ra_H-\kappa\alpha_{\nhat}\|P_{\nbar}\ubar(t)\|^2_H\\
&\le -\kappa\|(I-P_{\nbar})A^{1/2}\ubar(t)\|^2_H+a_\f\|\ubar(t)\|^2_H-\kappa\alpha_{\nbar}\|P_{\nbar}\ubar(t)\|^2_H\\
&\le -(\kappa\alpha_{\nbar}-a_\f)\|\ubar(t)\|^2_H.
\end{align*}
Altogether, we arrive at
\begin{align*}
\frac{\d}{\d  t}\Psi_0(\ubar(t),\zetaeps(t))&\le  -(\kappa\alpha_{\nbar}-a_\f)\|\ubar(t)\|^2_H-\frac{1}{2}(1-\kappa)\delta \|\zetaeps(t)\|^2_{\Mzeroeps}\\
&\le -\min\{2(\kappa\alpha_{\nbar}-a_\f),\delta \}\Psi_0(\ubar(t),\zetaeps(t)).
\end{align*}
In particular
\begin{align*}
\Psi_0(\ubar(t),\zetaeps(t)) &\le e^{-\min\{2(\kappa\alpha_{\nbar}-a_\f),\delta\}t}\Psi_0(\ubar(0),\zetaeps(0)).
\end{align*}
Since
\begin{align*}
\frac{1}{2}(1-\kappa)\|\x\|^2_{\Hzeroeps}\le \Psi_0(\x)\le \frac{1}{2}\|\x\|^2_{\Hzeroeps},
\end{align*}
we deduce
\begin{align*}
\|{(\ubar(t),\zetaeps(t))}\|^2_{\Hzeroeps} &\le\frac{1}{1-\kappa} e^{-\min\{2(\alpha_{\nbar}-a_\f),\delta\}t}\|{(\ubar(0),\zetaeps(0))}\|^2_{\Hzeroeps}\\
&=\frac{1}{1-\kappa} e^{-\min\{2(\alpha_{\nbar}-a_\f),\delta\}t}\|\x-\y\|^2_{\Hzeroeps},
\end{align*}
which implies the estimate \eqref{ineq:dissipative-bound}, as desired.

\end{proof}

Next, we consider the difference $\Utilde(t;\x,\y)-U(t;\y)$ and prove \cref{lem:error-in-law}.

\begin{proof}[Proof of \cref{lem:error-in-law}]
Let us first consider the following cylindrical Wiener process
\begin{equation} \label{form:w.tilde}
\d\wtilde(t)=\beta(t)\d t+\d w(t),
\end{equation}
where
\begin{equation} \label{form:beta(t)}
\beta(t)= \kappa\alpha_{\nbar} Q^{-1}P_{\nbar}(u(t)-\utilde(t)),
\end{equation}
and $u(t)=\pi_1 U(t;\x)$ and $\utilde(t)=\pi_1\Utilde(t;\x,\y)$. The aim is to compare the law of $\wtilde$ with that of $w$ over $[0,t]$. 
Indeed, thanks to~\nameref{cond:Q:ergodicity}, $Q$ is invertible on $H_{\nbar}=\text{span}\{e_1,\dots,e_{\nbar}\}$. We may then employ~\eqref{ineq:dissipative-bound} from \cref{lem:dissipative-bound} to infer
\begin{align*}
\|\beta(t)\|^2_H=\|\kappa\alpha_{\nbar} Q^{-1}P_{\nbar}(u(t)-\utilde(t))\|_H^2&\le (\kappa\alpha_{\nbar}a_Q^{-1})^2\|P_{\nbar}(u(t)-\utilde(t))\|_H^2\\
&\le C e^{-c t}\|\x-\y\|^2_{\Hzeroeps},
\end{align*}
where $C=C(\nbar,Q,\kappa,\delta),c=c(\nbar,Q,\kappa,\delta)$ are positive constants independent of $\x,\,\y,\,t$, and ${\mueps}$. As a consequence
\begin{align} \label{ineq:int.beta}
\E\int_0^\infty\close \|\beta(t)\|^2_H\d t
& =\E \int_0^\infty\close\|\kappa\alpha_{\nbar} Q^{-1}P_{\nbar}(u(t)-\utilde(t))\|^2_H\d t \le C\|\x-\y\|^2_{\Hzeroeps}.
\end{align}
In light of \cite[Inequality (A.1) and Theorem A.2]{butkovsky2020generalized} together with~\eqref{ineq:int.beta}, it holds that
\begin{align} \label{ineq:error-in-law:1:w}
\W_{\TV}(\Law(\wtilde_{[0,t]}),\Law(w_{[0,t]}))&\le\frac{1}{2}\Big(\E\int_0^\infty\close \|\beta(t)\|^2_H\d t\Big)^{1/2} \le C\|\x-\y\|_{\Hzeroeps}.
\end{align}
On the other hand, we invoke \cite[Inequality (A.2)]{butkovsky2020generalized} to see that for all $\x,\y\in \mathcal{D}_R$
\begin{equation}\label{ineq:error-in-law:2:w}
\begin{aligned} 
\W_{\TV}(\Law(\wtilde_{[0,t]}),\Law(w_{[0,t]}))&\le 1-\frac{1}{2}e^{-\frac{1}{2}\E\int_0^\infty \|\beta(t)\|^2_H\d t}\le 1-\frac{1}{2}e^{-C\|\x-\y\|^2_{\Hzeroeps}}\le 1-\frac{1}{2}e^{-CR^2}.
\end{aligned}
\end{equation}

Next, we observe that
\begin{equation} \label{ineq:Law.P_t<Law.w}
\W_{\TV}(\Law(\Utilde(t;\x,\y)),P_t(\y,\cdot))\le\W_{\TV}(\Law(\wtilde_{[0,t]}),\Law(w_{[0,t]})).
\end{equation} 
Indeed, to see that \eqref{ineq:Law.P_t<Law.w} holds, consider any coupling $(\tilde{W},W)$ of $(\Law(\wtilde_{[0,t]}),\Law(w_{[0,t]}))$ and denote by $\Xtil(t;\x,\y)$, $X(t;\y)$ respectively the solutions of~\eqref{eqn:react-diff:mu:pi_n(u_x-u_xy)} and~\eqref{eqn:react-diff:epsilon:mu} associated with $\tilde{W}$ and $W$. It is clear that $(\Xtil(t;\x,\y),X(t;\y))$ is a coupling for $(\Utilde(t;\x,\y),U(t;\y))$. By uniqueness of weak solutions, we have
\begin{align*}
\{\Wtilde(r)=W(r): r\in[0,t]\}\subseteq\{\Xtil(t;\x,\y)=X(t;\y)\},
\end{align*}
which implies
\begin{align*}
\boldsymbol{1}_{\left\{\Wtilde(r)\neq W(r), \text{ for some }r\in[0,t]\right\}}\ge \boldsymbol{1}_{\left\{\Xtil(t;\x,\y)\neq X(t;\y)\right\}}.
\end{align*}
In particular, we have
\begin{align*}
\E\boldsymbol{1}_{\left\{\Wtilde(r)\neq W(r), \text{ for some }r\in[0,t]\right\}}\ge \W_{\TV}(\Law(\Utilde(t;\x,\y)),P_t(\y,\cdot)).
\end{align*}
Since the coupling $(\Wtilde,W)$ was arbitrary, we deduce~\eqref{ineq:Law.P_t<Law.w}.

To conclude the proof, we combine~\eqref{ineq:error-in-law:1:w}, \eqref{ineq:error-in-law:2:w}, \eqref{ineq:Law.P_t<Law.w} to deduce~\eqref{ineq:error-in-law:1}--\eqref{ineq:error-in-law:2}, as desired.

\end{proof}

Lastly, we provide the proof of \cref{lem:exponential-bound:H^0_epsilon:beta/2}.

\begin{proof}[Proof of~\cref{lem:exponential-bound:H^0_epsilon:beta/2}]
First of all, by It\^o's formula and~\eqref{ineq:L^epsilon.Psi_0}, we readily have
\begin{align}
\d \Psi_0(U(t))
&=\L^{\mueps}\Psi_0(U(t))\d t+\la u(t),Q\d w(t)\ra_H \notag\\
&\le -\kappa\|A^{1/2}u\|^2_H\d t-\frac{1}{2}(1-\kappa)\delta\|\eta\|^2_{\Mzeroeps}\d t+\Big(a_3|\domain|+\frac{1}{2}\Tr(QQ^*)\Big)\d t+\la u(t),Q\d w(t)\ra_H,\label{ineq:d.Psi_0}
\end{align}
Let $\gamma\in(0,1),\,\beta>0$ be given and to be chosen later. Fixing $T>0$, in view of~\eqref{ineq:d.Psi_0}, we have
\begin{align*}
&\d \Big(\beta e^{\gamma(t-T)}\Psi_0(U(t))\Big)\\
&\le\beta \gamma e^{\gamma(t-T)}\Psi_0(U(t))\d t -\beta\kappa e^{\gamma(t-T)}\|A^{1/2}u(t)\|^2_H\d t-\frac{1}{2}\beta(1-\kappa)\delta  e^{\gamma(t-T)}\|\eta(t)\|^2_{\Mzeroeps}\d t\\
&\qquad+\beta e^{\gamma(t-T)} C\d t+\beta e^{\gamma(t-T)}\la u(t),Q\d w(t)\ra_H.
\end{align*}
In the above, $C$ is a positive constant independent of $\varepsilon$. By choosing $\gamma$ small enough, namely, 
\begin{align*}
\gamma< \frac{1}{4} \min\{2\kappa\alpha_1,(1-\kappa)\delta\},
\end{align*}
it holds that
\begin{align*}
\gamma \Psi_0(U(t)) -\kappa \|A^{1/2}u(t)\|^2_H-\frac{1}{2}(1-\kappa)\delta  \|\eta(t)\|^2_{\Mzeroeps}\le - \gamma\Psi_0(U(t)).
\end{align*}
Also
\begin{align*}
-\beta\gamma \Psi_0(U(t)) \le -\frac{1}{2}\beta \gamma \|u(t)\|^2_H &=  -\frac{1}{2}\cdot\frac{\gamma}{\beta\Tr(QQ^*)}\cdot \beta^2\Tr(QQ^*) \|u(t)\|^2_H\\
&\le -\frac{1}{2}\frac{\gamma}{\beta\Tr(QQ^*)}\beta^2\sum_{k\ge 1}| \la u(t),Qe_k\ra_H|^2.
\end{align*}
In the above, we recall that $\Tr(QQ^*)<\infty$ as in~\nameref{cond:Q}. As a consequence, for $t\in [0,T]$, we obtain
\begin{align*}
&\d \Big(\beta e^{\gamma(t-T)}\Psi_0(U(t))\Big)\\
&\le \beta e^{\gamma(t-T)} C\d t+\beta e^{\gamma(t-T)}\la u(t),Q\d w(t)\ra_H-\frac{1}{2}\frac{\gamma}{\beta\Tr(QQ^*)}\beta^2e^{2\gamma(t-T)}\sum_{k\ge 1}| \la u(t),Qe_k\ra_H|^2\d t.
\end{align*}
Integrating the above inequality with respect to time $t\in [0,T]$ yields
\begin{align*}
&\beta\Psi_0(U(T))-\beta e^{-\gamma T}\Psi_0(\x)-\frac{\beta}{\gamma}C\big(1-e^{-\gamma T}\big) \\
&\le \int_0^T\beta e^{\gamma(t-T)}\la u(t),Q\d w(t)\ra_H-\frac{1}{2}\frac{\gamma}{\beta\Tr(QQ^*)}\int_0^T\beta^2e^{2\gamma(t-T)}\sum_{k\ge 1}| \la u(t),Qe_k\ra_H|^2\d t\\
&\le M(T)-\frac{1}{2}\frac{\gamma}{\beta\Tr(QQ^*)}\la M\ra(T),
\end{align*}
where $M(r)$ denotes the semi-Martingale given by
\begin{align*}
M(r)= \int_0^r\beta e^{\gamma (t-T)}\la u(t),Q\d w(t)\ra_H.
\end{align*}
Recall the exponential Martingale inequality that
\begin{align*}
\P\Big(\sup_{r\ge 0}\Big[M(r)-\frac{c}{2}\la M\ra(r)\Big] >R\Big)\le e^{-c\,R}, \quad R\ge 0,
\end{align*}
whence,
\begin{align*}
\P\Big(\beta\Psi_0(U(T))-\beta e^{-\gamma T}\Psi_0(\x)-\frac{\beta}{\gamma}\big(1-e^{-\gamma T}\big)\ge R\Big) \le e^{- \frac{\gamma}{\beta\Tr(QQ^*)}R},
\end{align*}
for any $R\ge 0$. We now choose $\beta=\beta(\gamma)$ small enough such that
\begin{align*}
\beta<\frac{\gamma}{\Tr(QQ^*)}.
\end{align*}
This implies that
\begin{align*}
\E\exp\left( \beta\Psi_0(U(T))-\beta e^{-\gamma T}\Psi_0(\x)-\frac{\beta}{\gamma}\big(1-e^{-\gamma T}\big)\right)\le C,
\end{align*}
for some positive constant $C=C(\gamma,\beta)$ independent of $\varepsilon$. This produces~\eqref{ineq:exponential-bound:H^0_epsilon:beta/2}, thereby concluding the proof.
\end{proof}

\section{Short memory limit as $\varepsilon\to 0$} \label{sec:epsilon->0:inv-measure}

In this section, we proceed to prove the main theorems concerning the convergence of~\eqref{eqn:react-diff:epsilon:mu} toward~\eqref{eqn:react-diff} as $\varepsilon\to 0$.

\subsection{Short memory limit over a finite time-horizon} \label{sec:epsilon->0:finite-time}

This section is dedicated to the proof of \cref{thm:epsilon->0:solution->limit-solution:finite-time}, which concerns the limit~\eqref{lim:epsilon->0:solutions->limit-solution} in any finite time window. In order to establish~\eqref{lim:epsilon->0:solutions->limit-solution}, we must navigate two natural difficulties: one coming from the nonlinear structure of the potential, $\f$, and another coming from memory. The difficulties concerning $\f$ have been dealt with in the past (see \cite{cerrai2020convergence,hottovy2015smoluchowski,nguyen2018small}) and is quite standard to deal with. The significant difficulty in our case derives from the memory terms, which require a more delicate analysis that depends on the dimension of the spatial. To be more precise, in general dimension $d$, we may control the potential $\f$ by making use of~\nameref{cond:phi:3}, whereas we will employ~\eqref{cond:phi:O(|x|)} together with the first part of ~\cref{lem:react-diff:phi.Lipschit:|u(t)-u(s)|-bound} in order to estimate the memory. However, in dimension $d\le3$, when $\f$ satisfies~\eqref{cond:phi:agmon}, we must make use of the condition~\eqref{cond:U^epsilon_0:agmon} regarding higher moments in $H$ and $H^1$, as well as the second part of~\cref{lem:react-diff:phi.Lipschit:|u(t)-u(s)|-bound} in order to overcome both the memory and the non-linear structure of the potential. In essence, the argument we develop here can be viewed as a properly stochastic version of the one developed in~\cite[Lemma 5.5]{conti2006singular} for the deterministic analog of our problem.

\begin{proof}[Proof of~\cref{thm:epsilon->0:solution->limit-solution:finite-time} ]

Setting $\zeps(t) = \ueps(t)-u^0(t)$ and  $\zetaeps(t) = \etaeps(t)-\etazero(t)$, the pair $(\zeps(t),\zetaeps(t))$ obeys the following system in $\Hzeroeps$
 \begin{equation} \label{eqn:phi-lipschitz:difference}
\begin{aligned}
\frac{\d}{\d t} \zeps(t)&=-\kappa A\zeps(t)-(1-\kappa)\int_0^\infty\close \mu_\varepsilon(s)A\etaeps(s)\d s+(1-\kappa)A u^0(t)+\f(\ueps(t))-\f(u^0(t)) , \\
\frac{\d}{\d t}\zetaeps(t)&=\Tcal_{\mueps}\zetaeps(t)+\zeps(t),\\
(\zeps(0),\zetaeps(0))&=(0,0)\in\Hzeroeps.
\end{aligned}
\end{equation}
Multiplying the first equation with $\zeps(t)$ in $H^0$ and the second equation with $(1-\kappa)\zetaeps(t)$ in $\M^0_\varepsilon$ and taking expectation, we obtain the relation (recalling $\Psi_0$ as in~\eqref{form:Psi_0})
\begin{equation} \label{eqn:E.Psi_1(u.bar(T),eta.bar(T))}
\begin{aligned}
\frac{\d}{\d t}\E\Psi_0(\zeps(t),\zetaeps(t))
&=-\kappa\E\|A^{1/2}\zeps(t)\|^2_H-(1-\kappa)\E\int_0^\infty\close \mu_\varepsilon(s)\la A^{1/2}\etazero(t;s),A^{1/2}\zeps(t) \ra_H\d s\\
&\qquad+(1-\kappa)\E\la A^{1/2}u^0(t),A^{1/2}\zeps(t)\ra_H+\E\la \Tcal_{\mueps}\zetaeps(t),\zetaeps(t)\ra_{\Mzeroeps}\\
&\qquad+\E\la \f(\ueps(t))-\f(u^0(t)),\zeps(t)\ra_H.
\end{aligned}
\end{equation}
Since $\f$ satisfies~\nameref{cond:phi:3}, we readily have
\begin{align*}
\la \f(\ueps(t))-\f(u^0(t)),\zeps(t)\ra_H\leq a_\f\|\zeps\|^2_H,
\end{align*}
We note that $-\kappa\E\|A^{1/2}\zeps(t)\|^2_H$ is negligible. So is the inner product involving $\Tcal_{\mueps}$, by virtue of~\eqref{ineq:<T_epsilon.eta,eta>}. For the other two remaining terms, using the normalized condition~\nameref{cond:mu:int.s.mu(s)=1}, namely,
$$\int_0^\infty\close s\mu_\varepsilon(s)\d s=\int_0^\infty \close s \frac{1}{\varepsilon^2}\mu\big(\frac{s}{\varepsilon}\big)\d s=\int_0^\infty \close s\mu(s)\d s=1,$$
we rewrite them as following
\begin{align*}
\MoveEqLeft[5]-\int_0^\infty\close \mu_\varepsilon(s)\la A^{1/2}\etazero(t;s),A^{1/2}\zeps(t) \ra_H\d s+\la A^{1/2}u^0(t),A^{1/2}\zeps(t)\ra_H\\
&=-\int_0^\infty\close \mu_\varepsilon(s)\la A^{1/2}\etazero(t;s),A^{1/2}\zeps(t) \ra_H\d s+\int_0^\infty\close s\mu_\varepsilon(s)\la A^{1/2}u^0(t),A^{1/2}\zeps(t)\ra_H\d s\\
&=-\int_{t\varepsilon^\gamma}^\infty\close \mu_\varepsilon(s)\la A^{1/2}\etazero(t;s),A^{1/2}\zeps(t) \ra_H \d s+\int_{t\varepsilon^\gamma}^\infty \close s\mu_\varepsilon(s) \la A^{1/2}u^0(t) ,A^{1/2}\zeps(t)\ra_H \d s\\
&\qquad\qquad - \int_0^{t\varepsilon^\gamma}\close\close \mu_\varepsilon(s)\la A^{1/2}\etazero(t;s)-s A^{1/2}u^0(t),A^{1/2}\zeps(t) \ra_H \d s\\
&=-I_1^\varepsilon(t)+I_2^\varepsilon(t)-I_3^\varepsilon(t).
\end{align*}
In the above, $\gamma\in(0,1)$ is a positive constant to be chosen later. With regard to the first term $I_1^\varepsilon$, thanks to~\eqref{form:mu_epsilon} and \eqref{ineq:mu}, we note that
\begin{align} \label{ineq:int.mu_epsilon(s)ds}
\int_{t\varepsilon^\gamma}^\infty \mu_\varepsilon(s)\d s = \frac{1}{\varepsilon} \int_{t\varepsilon^{\gamma-1}}^\infty \mu(s)\d s \leq \frac{\mu(0)}{\delta\varepsilon}e^{-\delta t\varepsilon^{\gamma-1}}. 
\end{align}
Hence, we estimate $I^\varepsilon_1(t)$ using the Cauchy-Schwarz inequality and obtain
\begin{align*}
\big{|}\E I^\varepsilon_1(t)\big{|}&=\big{|}\E\int_{t\varepsilon^\gamma}^\infty\close \mu_\varepsilon(s)\la A^{1/2}\etazero(t;s),A^{1/2}\zeps(t) \ra_H \d s\big{|}\\
&\leq \left(\E\int_{t\varepsilon^{\gamma}}^\infty\mu_\varepsilon(s)\d s\right)^{1/2} \left(\E\int_{t\varepsilon^{\gamma}}^\infty \mu_{\varepsilon}(s)|\la A^{1/2}\eta^{0,\varepsilon}(t;s),A^{1/2}\zeps(t)\ra_H|^2\d s\right)^{1/2}\\
&\leq c\frac{e^{-\delta t\varepsilon^{\gamma-1}/2}}{\varepsilon^{1/2}}\E\big[\|\etazero(t)\|_{M^0_\varepsilon}\|A^{1/2}\zeps(t)\|_H\big]\\
&\le c\frac{e^{-\delta t\varepsilon^{\gamma-1}/2}}{\varepsilon^{1/2}}\big(\E\|\etazero(t)\|^2_{M^0_\varepsilon}\big)^{1/2}\big(\E\|A^{1/2}\zeps(t)\|_H^2\big)^{1/2}.
\end{align*}
To estimate the two expected values appearing in the right-hand side above, we invoke~\cref{lem:moment-boud:H^1_epsilon}, part 1, and~\cref{lem:react-diff:|u|^2_(H^1)-bound.and.int_0^t|Au(s)|^2ds-bound:random-initial-cond} to deduce
\begin{equation} \label{ineq:|A^(1/2)ubar(t)-bound|}
\begin{aligned}
\E\|A^{1/2}\zeps(t)\|^2_H\le \E\|A^{1/2}\ueps(t)\|^2_H+\E\|A^{1/2}u^0(t)\|^2_H
&\le c(\E\|\rx^\varepsilon\|^2_{\Honeeps}+t)\le c(1+t),
\end{aligned}
\end{equation}
for some $c>0$, independent of $U_0^\varepsilon$ and $t$. Note that in obtaining the final inequality we employed~\eqref{cond:U^epsilon_0:O(|x|)}. 

Likewise, in view of~\cref{lem:uniform-energy-bounds}, we have
\begin{equation}\label{ineq:|eta^0(t)|_(M_0^epsilon)-bound|}
\begin{aligned}
\E\|\etazero(t)\|^2_{\Mzeroeps}&\le e^{-\frac{\delta}{2\varepsilon}t}\E\|\pi_2\rx^\varepsilon\|^2_{\Mzeroeps}+c(\E\|\pi_1\rx^\varepsilon\|^2_{H^1}+1)\varepsilon\le c\big( e^{-\frac{\delta}{2\varepsilon}t}+\varepsilon\big).
\end{aligned}
\end{equation}
Upon combining \eqref{ineq:|A^(1/2)ubar(t)-bound|} and \eqref{ineq:|eta^0(t)|_(M_0^epsilon)-bound|} with the elementary inequality $\sqrt{a+b}\le \sqrt{a}+\sqrt{b}$, for $a,b\geq0$, we arrive at
\begin{align*}
\Big(\E\|A^{1/2}\zeps(t)\|^2_H\E\|\etazero(t)\|^2_{\Mzeroeps}\Big)^{1/2}
&\le c(1+T)\big( e^{-\frac{\delta}{4\varepsilon}t}+\varepsilon^{1/2}\big).
\end{align*}
Returning to $I^\varepsilon_1(t)$, it is now clear that
\begin{align*}
|\E\,I^\varepsilon_1(t)|&\le c(1+T)\frac{e^{-\delta t\varepsilon^{\gamma-1}/2}}{\varepsilon^{1/2}}\big(\E\|A^{1/2}\zeps(t)\|_H^2\big)^{1/2}\big(\E\|\etazero(t)\|^2_{M^0_\varepsilon}\big)^{1/2}\\
&\le c(1+T)\frac{e^{-\delta t\varepsilon^{\gamma-1}/2}}{\varepsilon^{1/2}}\big( e^{-\frac{\delta}{4\varepsilon}t}+\varepsilon^{1/2}\big),
\end{align*}
where we emphasize that $c>0$ is independent of $\varepsilon$ and $T$. We now integrate with respect to time $t$ on $[0,T]$ to see that
\begin{equation}  \label{ineq:epsilon->0:I_1(t)}
\big{|}\int_0^T\E\,I^\varepsilon_1(t) \d t\big{|}\le c(1+T)\int_0^T \frac{e^{-\delta t\varepsilon^{\gamma-1}/2}}{\varepsilon^{1/2}}\big( e^{-\frac{\delta}{4\varepsilon}t}+\varepsilon^{1/2}\big)\d t\le c(T)\big(\varepsilon^{1/2}+\varepsilon^{1-\gamma}\big).
\end{equation}

With regard to $I^\varepsilon_2(t)$, we combine~\eqref{ineq:int.mu_epsilon(s)ds} and Holder's inequality to estimate
\begin{align*}
|\E\,I^\varepsilon_2(t)|&=\big{|}\E\int_{t\varepsilon^{\gamma}}^\infty\close s\mu_\varepsilon(s) \la A^{1/2}u^0(t) ,A^{1/2}\zeps(t)\ra \d s\big{|}\\
& \leq \int_{t\varepsilon^{\gamma}}^\infty\close s\mu_\varepsilon(s)\d s\Big(\E\|A^{1/2}u^0(t)\|^2_H\Big)^{1/2}\Big(\E\| A^{1/2}\zeps(t)\|^2_H\Big)^{1/2}.
\end{align*}
To bound the integral involving $\mu_\varepsilon$, we again employ~\eqref{form:mu_epsilon}, \eqref{ineq:mu} and Holder's inequality to see that
\begin{align*}
\int_{t\varepsilon^{\gamma}}^\infty\close s\mu_\varepsilon(s)\d s&\le \Big(\int_{t\varepsilon^{\gamma}}^\infty\close s^2\mu_\varepsilon(s)\d s\Big)^{1/2}\Big(\int_{t\varepsilon^{\gamma}}^\infty\close \mu_\varepsilon(s)\d s\Big)^{1/2}\\
&=\Big(\varepsilon\int_{t\varepsilon^{\gamma-1}}^\infty\close\close s^2\mu(s)\d s\Big)^{1/2}\Big(\int_{t\varepsilon^{\gamma}}^\infty\close \mu_\varepsilon(s)\d s\Big)^{1/2},
\end{align*}
which, in view of~\eqref{ineq:int.mu_epsilon(s)ds}, is further dominated by
\begin{align*}
c\Big(\varepsilon\int_0^\infty\close s^2\mu(s)\d s\Big)^{1/2}\,\frac{e^{-\delta t\varepsilon^{\gamma-1}/2}}{\varepsilon^{1/2}}\le c\,e^{-\delta t\varepsilon^{\gamma-1}/2}.
\end{align*}
Also, similar to~\eqref{ineq:|A^(1/2)ubar(t)-bound|}, we have
\begin{align*}
\Big(\E\|A^{1/2}u^0(t)\|^2_H\Big)^{1/2}\Big(\E\| A^{1/2}\zeps(t)\|^2_H\Big)^{1/2}\le c(\E\|\rx^\varepsilon\|^2_{\Honeeps}+t)
&\le c(1+t).
\end{align*}
For the last inequality above, note that we once again employed the uniform-in-$\varepsilon$ bound in $\Honeeps$ for $\rx^\varepsilon$  implied by~\eqref{cond:U^epsilon_0:O(|x|)}. It follows that
\begin{align*}
|\E\,I^\varepsilon_2(t)|\le c(T)e^{-\delta t\varepsilon^{\gamma-1}/2},
\end{align*}
which in turn implies
\begin{equation}  \label{ineq:epsilon->0:I_2(t)}
\big{|}\int_0^T\E\,I^\varepsilon_2(t) \d t\big{|}\le c(T)\int_0^T e^{-\delta t\varepsilon^{\gamma-1}/2}\d t\le c(T)\frac{2}{\delta\varepsilon^{\gamma-1}}\left(1-e^{-\frac{\delta T\varepsilon^{\gamma-1}}2}\right)\leq c\varepsilon^{1-\gamma}.
\end{equation}

Concerning the last term $I_3^\varepsilon(t)$, we first recall that $\etazero(t)$ satisfies its corresponding evolution equation from the Cauchy problem~\eqref{eqn:react-diff:eta^0_epsilon}. Thus, in light of~\eqref{form:eta(t)-representation},  $\etazero(t;s)$ is explicitly given by
\begin{align*}
\etazero(t;s)=\int_0^s u^0(t-r)\d r,\quad \text{for all}\ 0\le s\le t\,\varepsilon^{\gamma}\le t.
\end{align*} 
Using the above formula, we can rewrite $I^\varepsilon_3(t)$ as
\begin{align*}
\MoveEqLeft[5] \int_0^{t\,\varepsilon^{\gamma}}\close\close \mu_\varepsilon(s)\la A^{1/2}\etazero(t;s)-s\, A^{1/2}u^0(t),A^{1/2}\zeps(t) \ra_H \d s \\
 & =  \int_0^{t\,\varepsilon^{\gamma}}\close \mu_\varepsilon(s)\Big\la \int_0^s\big[ A^{1/2}u^0(t-r)- A^{1/2}u^0(t)\big] \d r,A^{1/2}\zeps(t)\Big \ra_H \d s \\
 &=  \int_0^{t\,\varepsilon^{\gamma}}\close \mu_\varepsilon(s)\int_0^s \la  u^0(t-r)- u^0(t) ,A\zeps(t) \ra_H\d r\, \d s.
 \end{align*}
We then apply the Cauchy-Schwarz inequality repeatedly to estimate
\begin{align}
|\E\, I^\varepsilon_3(t)|&\le\E\Big[\|A\zeps(t)\|_H  \int_0^{t\,\varepsilon^{\gamma}}\close \mu_\varepsilon(s)\int_0^s \|u^0(t-r)- u^0(t)\|_H \d r\, \d s\Big] \nonumber\\
&\le \big(\E\|A\zeps(t)\|_H^2\big)^{1/2} \left[\E\left( \int_0^{t\,\varepsilon^{\gamma}}\close \mu_\varepsilon(s)\int_0^s \|u^0(t-r)- u^0(t)\|_H \d r\, \d s\right)^2\right]^{1/2}\nonumber\\
&\le \big(\E\|A\zeps(t)\|_H^2\big)^{1/2} t^{1/2}\varepsilon^{\gamma/2}\Big(\int_0^{t\,\varepsilon^{\gamma}} \close\mu_\varepsilon^2(s)\E\big{|}\int_0^s \|u^0(t-r)- u^0(t)\|_H \d r\big{|}^{2}\, \d s\Big)^{1/2}\nonumber\\
&\le \varepsilon^{\gamma/2}t^{1/2}\big(\E\|A\zeps(t)\|_H^2\big)^{1/2} \Big( \int_0^{t\,\varepsilon^{\gamma}}\close\close s\mu_\varepsilon^2(s)\int_0^s \E\|u^0(t-r)- u^0(t)\|_H^2 \d r\, \d s\Big)^{1/2}.\label{ineq:I^epsilon_3(t)}
\end{align}
Now, we further study the final term on the  above right-hand side. To do so, we will consider two cases depending on the structure of the non-linear potential $\f$ and dimension of the spatial domain: either $d\geq 4$ and $\f$ satisfies~\eqref{cond:phi:O(|x|)}, or $d= 3$ and $\f$ satisfies~\eqref{cond:phi:agmon}, or $d=1,2$.

\noindent\textbf{Case 1}. Then $d\geq1$ and $\f$ satisfies~\eqref{cond:phi:O(|x|)}. In this case, we invoke the first assertion in ~\cref{lem:react-diff:phi.Lipschit:|u(t)-u(s)|-bound} to see that
\begin{equation} \label{ineq:I^epsilon_3(t):lipschitz}
\begin{aligned}
\MoveEqLeft[5]\int_0^{t\varepsilon^{\gamma}}\close\close s\mu_\varepsilon^2(s)\int_0^s \E\|u^0(t-r)- u^0(t)\|_H^2 \d r\, \d s\\
&\le  C(t+1)\big(\E\|\pi_1\rx^\varepsilon\|^2_{H^1}+1\big)\int_0^{t\varepsilon^{\gamma}}\close\close s\mu_\varepsilon^2(s)\int_0^s r\,\d r\, \d s\\
&\le C(t+1)\Big(\E\|\pi_1\rx^\varepsilon\|^2_{H^1}+1\Big)\int_0^{t\varepsilon^{\gamma}} \close s^3\mu_\varepsilon(s)^2\d s\\
&\le C (t+1)\int_0^{t\varepsilon^{\gamma-1}}\close\close s^3\mu(s)^2\d s\le C \int_0^{\infty}\close s^3\mu(s)^2\d s\le C(t+1),
\end{aligned}
\end{equation}
where to obtain the third inequality we again employed the uniform-in-$\varepsilon$ bound in $\Honeeps$ for $\rx^\varepsilon$  implied by~\eqref{cond:U^epsilon_0:O(|x|)}. Combining~\eqref{ineq:I^epsilon_3(t)}--\eqref{ineq:I^epsilon_3(t):lipschitz}, we obtain the bound
\begin{align*}
|\E\, I^\varepsilon_3(t)|\le c\varepsilon^{\gamma/2}(t+1)\big(\E\|A\zeps(t)\|_H^2\big)^{1/2}.
\end{align*}
This together with~\eqref{ineq:solution:|Phi|^2_(H^1_epsilon)-bound:random-initial-cond} and \eqref{ineq:react-diff:|Phi|^2_(H^1_epsilon)-bound:random-initial-cond} implies
\begin{equation} \label{ineq:epsilon->0:I_3(t)}
\begin{aligned}
\int_0^T|\E\, I^\varepsilon_3(t)|\d t &\le c\varepsilon^{\gamma/2}\int_0^T (t+1) \big(\E\|A\zeps(t)\|_H^2\big)^{1/2}\d t\le c(1+T^3)\varepsilon^{\gamma/2}.
\end{aligned}
\end{equation}
Collecting everything, it follows from~\eqref{eqn:E.Psi_1(u.bar(T),eta.bar(T))},~\eqref{ineq:epsilon->0:I_1(t)},~\eqref{ineq:epsilon->0:I_2(t)} and~\eqref{ineq:epsilon->0:I_3(t)} that
\begin{align*}
\E\Psi_0(\zeps(t),\zetaeps(t)) \le a_\f\int_0^t \|\zeps(r)\|^2_H\d r+c(1+T^3)\big(\varepsilon^{1/2}+\varepsilon^{1-\gamma} +\varepsilon^{\gamma/2}\big).
\end{align*}
By Gronwall's inequality, we obtain 
\begin{align*} 
\E\Psi_0(\zeps(t),\zetaeps(t))\le C e^{cT}\big(\varepsilon^{1/2}+\varepsilon^{1-\gamma} +\varepsilon^{\gamma/2}\big),
\end{align*}
which holds for $t\in [0,T]$ and $\gamma\in(0,1)$. To optimize the powers on $\varepsilon$, we note that
\begin{align}\label{ineq:gamma=1/3}
\max_{\gamma\in(0,1)}\min\{1/2,1-\gamma,\gamma/2\}=1/3,
\end{align}
whence
\begin{align} \label{ineq:E.Psi_1(u.bar(T),eta.bar(T))<epsilon}
\sup_{t\in [0,T]}\E\Psi_0(\zeps(t),\zetaeps(t))\le C e^{cT}\varepsilon^{1/3}.
\end{align}
This yields~\eqref{lim:epsilon->0:solutions->limit-solution}.

\noindent\textbf{Case 2.} Then $d= 3$ and $\f$ satisfies~\eqref{cond:phi:agmon}. In this case, we invoke the third assertion in~\cref{lem:react-diff:phi.Lipschit:|u(t)-u(s)|-bound} to estimate
\begin{equation} \label{ineq:I^epsilon_3(t):agmon}
\begin{aligned}
&\int_0^s \E\|u^0(t-r)- u^0(t)\|_H^2 \d r\, \d s\\
&\le c(t^2+1)\big(\E\|\pi_1\rx^\varepsilon\|^{10}_{H^1}+\E \|\pi_1\rx^\varepsilon\|^{20}_H+1\big)\int_0^{t\varepsilon^{\gamma}}\close s\mu_\varepsilon^2(s)\int_0^s r\,\d r\, \d s\\
&\le c(t^2+1)\int_0^{t\varepsilon^{\gamma}} \close s^3\mu_\varepsilon^2(s)\d s= c(t^2+1) \int_0^{t\varepsilon^{\gamma-1}}\close\close s^3\mu(s)\d s\le c(t^2+1) \int_0^{\infty}\close s^3\mu(s)\d s\le c(t^2+1).
\end{aligned}
\end{equation}
To obtain the second inequality, we employed the uniform bound~\eqref{cond:U^epsilon_0:agmon}. An argument similar to~\eqref{ineq:epsilon->0:I_3(t)} also yields the estimate
\begin{equation*}
\int_0^T|\E\, I^\varepsilon_3(t)|\d t \le c\varepsilon^{\gamma/2}\int_0^T(t^{3/2}+1)\big(\E\|A\zeps(t)\|_H^2\big)^{1/2}\d t\le c(1+T^3)\varepsilon^{\gamma/2}.
\end{equation*}
We now combine the above estimate with~\eqref{eqn:E.Psi_1(u.bar(T),eta.bar(T))},~\eqref{ineq:epsilon->0:I_1(t)}, and~\eqref{ineq:epsilon->0:I_2(t)} to obtain
\begin{align*}
\sup_{t\in [0,T]}\E\Psi_0(\zeps(t),\zetaeps(t))\le Ce^{cT}\big(\varepsilon^{1/2}+\varepsilon^{1-\gamma} +\varepsilon^{\gamma/2}\big).
\end{align*}
Using~\eqref{ineq:gamma=1/3} again, we arrive at
\begin{align*}
\sup_{t\in [0,T]}\E\Psi_0(\zeps(t),\zetaeps(t))\le C e^{cT}\varepsilon^{1/3}.
\end{align*}
This establishes limit~\eqref{lim:epsilon->0:solutions->limit-solution} when $\f$ satisfies~\eqref{cond:phi:agmon}.

\noindent\textbf{Case 3.} Then $d=1,2$. Similar to Case 2, we invoke \eqref{ineq:react-diff:phi.Lipschit:|u(t)-u(s)|-bound:agmon:d=1-2} with \eqref{cond:U^epsilon_0:agmon:d=1-2} to obtain
\begin{align*}
&\int_0^s \E\|u^0(t-r)- u^0(t)\|_H^2 \d r\, \d s\\
&\le c(t^2+1)\big(\E\|\pi_1\rx^\varepsilon\|^{2 \lceil p_0\rceil }_{H^1}+\E \|\pi_1\rx^\varepsilon\|^{4 \lceil p_0\rceil}_H+1\big)\int_0^{t\varepsilon^{\gamma}}\close s\mu_\varepsilon^2(s)\int_0^s r\,\d r\, \d s\le c(t^2+1).
\end{align*}
Reasoning as in Case 2 also produces the bound
\begin{align*}
\sup_{t\in [0,T]}\E\Psi_0(\zeps(t),\zetaeps(t))\le C e^{cT}\varepsilon^{1/3},
\end{align*}
which establishes limit~\eqref{lim:epsilon->0:solutions->limit-solution} when $d=1,2$. The proof is thus finished.

\end{proof}

\begin{remark} \label{rem:epsilon->0:lipschitz}  In the proof of ~\cref{thm:epsilon->0:solution->limit-solution:finite-time} above, we  made use of a bound for $\E\|u^0(t-r)- u^0(t)\|_H^2 $, where $0\le r\le t$ (see the last two assertions in \cref{lem:react-diff:phi.Lipschit:|u(t)-u(s)|-bound}). The deterministic analog of \cref{lem:react-diff:phi.Lipschit:|u(t)-u(s)|-bound}, part 3 can be found in \cite[Lemma 5.5]{conti2006singular}; in this setting, such a bound can be obtained by studying $\|\partial_t u^0(t)\|_H$. However, this approach is clearly not applicable in the stochastic setting due to the fact that Brownian motion does not posses a traditional
time-derivative. This indicates the crucial role played by \cref{lem:react-diff:phi.Lipschit:|u(t)-u(s)|-bound} in the stochastic setting.
\end{remark}


\subsection{Uniform moment bounds on $\nueps$} \label{sec:ergodicity:moment-bound:nu^mu}

In this subsection, we proceed to show that the unique invariant measure $\nueps$ obtained from Theorem \ref{thm:react-diff:epsilon:geometric-ergodicity} satisfies uniform exponential moment bound in $\Honeeps$. More specifically, we state the following regularity theorem whose result combined with \cref{thm:epsilon->0:solution->limit-solution:finite-time} will be employed to establish \cref{thm:epsilon->0:invariant-measures->limit-measure}.

\begin{theorem} \label{thm:exponential-bound:H^1_epsilon:nu^mu}
Let $\nueps$ be the unique invariant probability measure of $P_t^{\mueps}$. Then, for all $\beta>0$ sufficiently small,
\begin{align} \label{ineq:exponential-bound:H^1_epsilon:nu^mu}
\sup_{\varepsilon\in(0,1]}\int_{\Hzeroeps} \exp\big\{ \beta \|U\|^2_{\Honeeps}\big\}\nueps(\textup{d} U)<\infty.
\end{align}

\end{theorem}

The proof of \cref{thm:exponential-bound:H^1_epsilon:nu^mu} essentially consists of two steps: first, we prove that $\nueps(\Honeeps)=1$. Then, we establish the exponential moment bound \eqref{ineq:exponential-bound:H^1_epsilon:nu^mu} by invoking the result from \cref{lem:moment-boud:H^1_epsilon}. We remark that while it is a consequence of \cref{thm:react-diff:epsilon:geometric-ergodicity} that $\nueps$ is full in $\Hzeroeps$, it is not a priori clear from the moment estimates in \cref{sec:moment} that $\nueps$ is also supported in $\Honeeps$. To overcome the regularity issue, we shall follow along the line of the existence of an invariant measure presented in \cite[Section 4]{glatt2024paperII}, while making use of the compact embedding $Z^1_{\mueps}\subset \Hzeroeps$. Here we recall the space $Z^1_{\mueps}=H^1\times \Ecal_{\mueps}^1$ from \eqref{form:space:H_epsilon}. We introduce the following family of measures $\{\numu_T\}_{T>0}$ given by
\begin{equation} \label{form:nu.time-average}
\nueps_T(\cdot)=\frac{1}{T}\int_0^T\close P_t^{\mueps}(0,\cdot)\d t.
\end{equation}
We now state an auxiliary result giving stronger bounds in $\Tcal_\mu$ and $\T^\mu$ uniform in time $t$, where $\Tcal_\mu$ and $\T^\mu$ are as in \eqref{form:Tcal} and \eqref{form:tailfunction}, respectively.

\begin{lemma} \label{lem:T:bound}
Under the same hypothesis of~\cref{prop:well-posed}, let $(\ueps(t;0),\etaeps(t;0))$ be the solution of \eqref{eqn:react-diff:epsilon:mu} with zero initial condition in $\Hzeroeps$. Then there exists a constant $c$ independent of $t$ such that the following bounds hold for all $t\geq0$:
$$\E\| \Tcal_{\mueps}\etaeps(t;0)\|^2_{\Mzeroeps}\le c\qquad\text{and}\qquad \E\sup_{r\ge 1}r\T^{\mueps}_{\etaeps(t)}(r)\le c.$$
\end{lemma}

\begin{remark}
It is important to point out that the constant $c$ appearing in the results of \cref{lem:T:bound} may depend on $\varepsilon$ as $\varepsilon\to 0$. Nevertheless, this does not affect the argument that $\nueps(\Honeeps)=1$ presented in the proof of \cref{thm:exponential-bound:H^1_epsilon:nu^mu}.
\end{remark}

Assuming that \cref{lem:T:bound} is true, we shall now give the proof of \cref{thm:exponential-bound:H^1_epsilon:nu^mu} giving the uniform exponential moment bound on $\nueps$.

\begin{proof}[Proof of \cref{thm:exponential-bound:H^1_epsilon:nu^mu}] First of all, we establish that $\nueps(\Honeeps)=1$. To see this, recalling the space $\Z^1_{\mueps}=H^1\times \Ecal^1_{\mueps}$ as in \eqref{form:space:H_epsilon}, since $\Ecal^1_{\mueps}$ is compactly embedded into $\Mzeroeps$  \cite{gatti2004exponential,joseph1989heat,joseph1990heat}, it is clear that the following bounded set
$$\mathcal{B}_R=\{(u,\eta)\in \Z^1_{\mueps} :\|(u,\eta)\|_{\Z^1_{\mueps} }\le R\},$$
is precompact in $\Hzeroeps=H\times \Mzeroeps$. Using Markov's inequality, we estimate
\begin{align}\label{eq:time:avg}
\nueps_t(\mathcal{B}_R^c)=\frac{1}{t}\int_0^t P^{\mueps}_s(0;\mathcal{B}_R^c)\d s\le \frac{1}{t}\int_0^t \frac{\E\|A^{1/2}\ueps(s)\|^2_H+\E\|\etaeps(s)\|^2_{\Ecal^1_{\mueps} }}{R^2}\d s.
\end{align}
To bound the term $\E\|A^{1/2}\ueps(s)\|^2_H$, we invoke \eqref{ineq:E.int_0^t|A^(1/2)u^epsilon(t)|-bound} with zero initial condition to deduce
\begin{equation}\label{ineq:int_0^t|A^(1/2)u^epsilon|ds<t}
\int_0^t\E\|A^{1/2}\ueps(r)\|^2_H\d r\le C(1+t).
\end{equation}
With regards to the second term, $\E\|\etaeps(s)\|^2_{\Ecal^1_{\mueps}}$, appearing in \eqref{eq:time:avg}, thanks to \cref{lem:moment-boud:H^1_epsilon} and \cref{lem:T:bound}, we infer the existence of a positive constant $c$, independent of $t$, $R$, and such that
    $$
        \int_0^t \E\|\etaeps(s)\|^2_{\Ecal^1_{\mueps} }\d s\le ct. 
    $$
It follows that there exists $t_R>1$, sufficiently large, such that 
$$\nueps_t(\mathcal{B}_R)\le \frac{c}{R^2},$$
for all $t\geq t_R$. This, in turn, implies that $\{\nueps_t\}_{t>0}$ is tight in $\Hzeroeps$. An application of Krylov-Bogoliubov's Theorem and the uniqueness of $\nueps$ implies that $\nueps_t$ converges weakly to $\nueps$ as $t$ tends to infinity. On the one hand, as a consequence of weak convergence, for each $R>0$,
\begin{align*}
\int_{\Hzeroeps}\big(\|U\|^2_{\Honeeps}\mi R  \big)\nueps_t(\d U)\to \int_{\Hzeroeps}\big(\|U\|^2_{\Honeeps}\mi R  \big)\nueps(\d U),\quad t\to\infty.
\end{align*}
On the other hand, we invoke \eqref{ineq:E.int_0^t|A^(1/2)u^epsilon(t)|-bound} and \eqref{ineq:solution:|Phi|^2_(H^1_epsilon)-bound:random-initial-cond} to deduce
\begin{align*}
\int_{\Hzeroeps}\big(\|U\|^2_{\Honeeps}\mi R  \big)\nueps_t(\d U) \le \frac{1}{t}\int_0^t \E\|\ueps(s)\|^2_{H^1}+\E\|\etaeps(r)\|^2_{\Moneeps}\d r\le c,
\end{align*}
for some positive constant $c$ independent of $t$ and $R$. It follows that 
\begin{align*}
\int_{\Hzeroeps}\big(\|U\|^2_{\Honeeps}\mi R  \big)\nueps(\d U)\le c(\varepsilon),
\end{align*}
whence
\begin{align*}
\int_{\Hzeroeps}\|U\|^2_{\Honeeps}\nueps(\d U)\le c(\varepsilon),
\end{align*}
by virtue of the Monotone Convergence Theorem. In particular, this implies that $\nueps$ is full in $\Honeeps$, as claimed.

Now the $\varepsilon$-uniform estimate \eqref{ineq:exponential-bound:H^1_epsilon:nu^mu} can be established by invoking the exponential moment bound \eqref{ineq:exponential-bound:H^1_epsilon}. Since the argument is standard and can be found in literature, we omit the details and refer the reader to \cite{weinan2001gibbsian} for instance.
\end{proof}

In order to prove \cref{lem:T:bound}, we need to obtain estimates for $\ueps(t)$ in {the} $H^1$--norm. More precisely, we have the following result, whose argument is similar to that of \cite[Lemma 4.2]{glatt2024paperII}.

\begin{lemma} \label{lem:int_0^t e^(delta/epsilon)|A^(1/2)u|ds}
Under the same hypothesis of \cref{prop:well-posed}, let $\Ueps(t)=(\ueps(t),\etaeps(t))$ denote the solution of \eqref{eqn:react-diff:epsilon:mu} corresponding to initial condition $U_0\in\Hzeroeps$. Then for all $0<q\le\delta$, where $\delta$ is the constant from \eqref{ineq:<T_epsilon.eta,eta>}, there exists a positive constant $c=c(U_0,q)$ independent of $t$ and $\varepsilon$ such that 
\begin{equation}\label{ineq:int_0^t e^(delta/epsilon)|A^(1/2)u|ds}
\E\int_0^t e^{qr}\|A^{1/2}\ueps(r)\|^2_H\emph{d} r\leq\|\x\|^2_{\Hzeroeps}+ c\,e^{qt}.
\end{equation}

\end{lemma}
\begin{proof}
We apply Ito's formula to $e^{qt}\Psi_0(\Ueps(t))$ to see that
\begin{align*}
\d \big(e^{qt}\Psi_0(\Ueps(t))\big)&=e^{qt}\Big[q\Psi_0(\Ueps(t))-\kappa\|A^{1/2}\ueps(t)\|^2_H+(1-\kappa)\la \Tcal_{\mueps}\etaeps(t),\etaeps(t)\ra_{\Mzeroeps}\\
&\qquad+\la\f(\ueps(t)),\ueps(t)\ra_H+\frac{1}{2}\Tr(QQ^*)\Big]\d t +e^{qt}\la \ueps(t),Q\d w(t)\ra_H.
\end{align*}
Thanks to \eqref{ineq:<T_epsilon.eta,eta>} and the assumption $q\le\delta$, we have
$$\frac{1}{2}(1-\kappa)q\|\eta\|^2_{\Mzeroeps}+ (1-\kappa)\la \Tcal_{\mueps}\eta,\eta\ra_{\Mzeroeps}\le 0.$$
Also, we invoke \nameref{cond:phi:2} and Young's inequality to see that
\begin{align*}
\frac{q}{2}\|u\|^2_H+\la\f(u),u\ra_H
&\le \frac{q}{2 }\|u\|^2_H-a_2\|u\|^{p_0+1}_{L^{p_0+1}(\domain)}+a_3|\domain|\\
&\le a_2\big(\frac{q}{2a_2 }\big)^{\frac{p_0+1}{p_0-1}}|\domain|+a_3|\domain|.
\end{align*}
We then infer the existence of a positive constant $c=c(q,\f,\domain,Q)$, independent of $t$ and $\varepsilon$, such that
\begin{equation*}
e^{qt}\E\Psi_0(\Ueps(t))+\E\int_0^t e^{qr}\|A^{1/2}\ueps(r)\|_H^2\d r\le \|\x\|^2_{\Hzeroeps}+c\int_0^t e^{qr}\d r\le\|\x\|^2_{\Hzeroeps}+ c\,e^{qt}.
\end{equation*}
This implies \eqref{ineq:int_0^t e^(delta/epsilon)|A^(1/2)u|ds}, as claimed.
\end{proof}
Finally, we complete the proof of \cref{thm:exponential-bound:H^1_epsilon:nu^mu} by supplying the proof of \cref{lem:T:bound}, whose argument is similar to that of \cite[Lemma 4.1]{glatt2024paperII}. {In order to prove \cref{lem:T:bound}, we recall the following elementary fact for locally integrable functions $f:\rbb\to[0,\infty)$ which are non-negative: for all $0<s\le t$ and $\delta>0$:}
\begin{align}\label{eq:tech:2}
e^{-\frac{\delta}{2}s} \int_0^s f(t-r)\d r&=\int_0^s e^{-\frac{\delta}{2}r} f(t-r)\d r-\frac{\delta}{2}\int_0^s e^{-\frac{\delta}{2}r} \int_0^r f(t-r')\d r'\d r  \notag \\
&\le \int_0^t e^{-\frac{\delta}{2}r} f(t-{r})\d r.
\end{align} 
\begin{proof}[Proof of \cref{lem:T:bound}]
Given the initial condition $\eta_0=0$, we recast \eqref{form:eta(t)-representation} as
\begin{align}
\etaeps(t,s)&=\chi_{(0,t]}(s)\int_0^s \ueps(t-r)\d r+
\chi_{(t,\infty)}(s)\int_0^t \ueps(t-r)\d r= \int_0^{s\wedge t}u(t-r)\d r.\label{form:eta(t)-representation:eta_0=0}
\end{align} 
This implies that $\Tcal_{\mueps}\etaeps(t)$ can be written explicitly as
$$\Tcal_{\mueps}\etaeps(t,s)=-\ueps(t-s)\chi_{(0,t]}(s).$$
It follows that 
$$\E\| \Tcal_{\mueps}\etaeps(t)\|^2_{\Mzeroeps}=\E\int_0^t\mueps(s)\|A^{1/2}\ueps(t-s)\|^2_H\d s=\E\int_0^t\mueps(t-s)\|A^{1/2}\ueps(s)\|^2_H\d s,$$
where we have simply made a change of variable in the last equality. By \eqref{ineq:mu_epsilon}
$$\mueps(s)\le \mueps(0)e^{-\delta s}.$$
Thus
\begin{align*}
\E\| \Tcal_{\mueps}\etaeps(t)\|^2_{\Mzeroeps}\le\mueps(0)e^{-\delta t}\E\int_0^te^{\delta s}\|A^{1/2}\ueps(s)\|^2_H\d s.
\end{align*}
In light of \cref{lem:int_0^t e^(delta/epsilon)|A^(1/2)u|ds} (with $q=\delta$), the claimed uniform-in-time on $\E\| \Tcal_{\mueps}\etaeps(t)\|^2_{\Mzeroeps}$ follows.

With regards to $\sup_{r\ge 0}r\T^{\mueps}_{\etaeps(t)}(r)$, we first note that \eqref{form:eta(t)-representation:eta_0=0} {and the Cauchy-Schwarz inequality together yield} 
\begin{align*}
\|A^{1/2}\etaeps(t,s)\|^2_H&= \int_0^{s\mi t}\close\int_0^{s\mi t}\close \la A^{1/2}\ueps(t-r_1),A^{1/2}\ueps(t-r_2)\ra_H \d r_1\d r_2\\
&\le (s\mi t)\int_0^{s\mi t}\close\|A^{1/2}\ueps(t-r')\|^2_H\d r'\\
&\le  s\int_0^{s\mi t}\close\|A^{1/2}\ueps(t-r')\|^2_H\d r'.
\end{align*}
Now for $s\in[0,t]$, let
    \[
        I(s):=e^{-\frac{\delta}{2}s}\int_0^{s}\|A^{1/2}\ueps(t-r')\|_H\d r'.
    \]
Recalling the expression of $\T^{\mueps}_{\etaeps(t)}$ in \eqref{form:tailfunction} and the fact that $\mueps(s)\le \mueps(0)e^{-\delta s}$, we may estimate for all $r\ge 1$ as follows:
\begin{align*}
r\T^{\mueps}_{\etaeps(t)}(r)&= r\int_{(0,\frac{1}{r})\cup(r,\infty)}\close\close\mueps(s)\|A^{1/2}\etaeps(t,s)\|^2_H\d s\\
&\le r\mueps(0)\int_{(0,\frac{1}{r})\cup(r,\infty)}\close \close e^{-\delta s}s\int_0^{s\mi t}\close\|A^{1/2}\ueps(t-r;)\|^2_H\d r'\,\d s\\
&=r\mueps(0)\int_{(0,\frac{1}{r})\cup(r,\infty)}\close\close s\,e^{-\frac{\delta}{2}s} e^{-\frac{\delta}{2}s} \int_0^{s\mi t}\close\|A^{1/2}\ueps(t-r')\|^2_H\d r'\,\d s\\
&\le r\mueps(0)\int_{(0,\frac{1}{r})\cup(r,\infty)}\close\close s\,e^{-\frac{\delta}{2}s}  e^{-\frac{\delta}{2}(s\mi t)} \int_0^{s\mi t}\close\|A^{1/2}\ueps(t-r')\|^2_H\d r'\,\d s\\
&=r\mueps(0)\int_{(0,\frac{1}{r})\cup(r,\infty)}\close\close s\,e^{-\frac{\delta}{2}s}{I(s\mi t)}\d s.
\end{align*}
By \eqref{eq:tech:2}, we have
$$I(s\mi t)\le \int_0^t e^{-\frac{\delta}{2}r'} \|A^{1/2}\ueps(t-r')\|^2_H\d r'.$$
It follows that
$$ r\T_{\etaeps(t)}^{\mueps}(r)\le r\mueps(0)\int_{(0,\frac{1}{r})\cup(r,\infty)}\close\close s\,e^{-\frac{\delta}{2}s} \d s \int_0^t e^{-\frac{\delta}{2} r'}\|A^{1/2}\ueps(t-r')\|^2_H\d r'.$$
For $r\ge 1$, we have
\begin{align*}
r\int_{(0,\frac{1}{r})\cup(r,\infty)}\close\close s\,e^{-\frac{\delta}{2}s} \d s<c(\delta)<\infty.
\end{align*}
Note that $c(\delta)$ is independent of $r$. Thus, taking supremum on $r\ge1$ yields
\begin{align*}
\sup_{r\ge 1} r\T_{\etaeps(t)}^\mu(r)&\le \sup_{r\ge 1} r \mueps(0)\int_{(0,\frac{1}{r})\cup(r,\infty)}\close\close s\,e^{-\frac{\delta}{2}s}\d s\int_0^t e^{-\frac{\delta}{2}r'}\|A^{1/2}\ueps(t-r')\|^2_H\d r'\\&\le c(\mueps) \int_0^t e^{-\frac{\delta}{2}r'}\|A^{1/2}\ueps(t-r')\|^2_H\d r'\\
&=c(\mueps)e^{-\frac{\delta}{2}t} \int_0^t e^{\frac{\delta}{2}r'}\|A^{1/2}\ueps(r')\|^2_H\d r'.
\end{align*}
Taking expectation on both sides and employing \cref{lem:int_0^t e^(delta/epsilon)|A^(1/2)u|ds} with $q=\delta/2$, we obtain the estimate
\begin{align*}
\E\sup_{r\ge 1}r\T^{\mueps} _{\etaeps(t)}(r)&\le  c(
\mueps),
\end{align*}
where $c(\mueps)$ is independent of $t$, as desired.
\end{proof}


\subsection{Short memory limit over an infinite time-horizon} \label{sec:epsilon->0:large-time}

We now turn our attention to the main result of the paper, namely we will prove that the marginal measure $\pi_1^{-1}\nueps$ converges to $\nu^0$, the invariant probability of~\eqref{eqn:react-diff} as $\varepsilon\to 0$. The general approach that we apply here has been recently developed in previous works \cite{cerrai2020convergence,foldes2015ergodic,foldes2019large,glatt2021mixing} and allows the proof of \cref{thm:epsilon->0:invariant-measures->limit-measure} to be summarized in three steps.

\hfill

\noindent\emph{Step 1}. Demonstrate that for any time $t>0$, the process $\Phieps(t)$ converges to $\Phizero(t)$ in $\Hzeroeps-$norm provided that the initial conditions are identically distributed from suitable random variables; this is guaranteed by~\cref{thm:epsilon->0:solution->limit-solution:finite-time}.

\hfill

\noindent\emph{Step 2}. Deduce that $\Phieps(t)$  converges to $\Phizero(t)$ with respect to the distance-like function $ d_{N,\beta}^{\mueps}$ given by \eqref{form:d_N,beta:mu} (see~\cref{lem:epsilon->0:alpha_0(solutions-limit-solution)->0} below). This limit is established by having access to exponential moment bounds that are established later in~\cref{sec:0-equation} and~\cref{sec:nu^(0-epsilon)}.

\hfill

\noindent\emph{Step 3}. Finally, combine invariance and the geometric ergodicity of for~\eqref{eqn:react-diff:epsilon:mu} guaranteed by \cref{thm:react-diff:epsilon:geometric-ergodicity}, to deduce that $\W_{d_{N,\beta}}(\pi_1^{-1}\nueps,\nu^0)$ is bounded by $\W_{d_{N,\beta}^{\mueps}}(\nueps,\nuzero)$. Due to \cref{lem:epsilon->0:alpha_0(solutions-limit-solution)->0}, the latter quantity converges to zero as $\varepsilon\to 0$, and subsequently $\W_{d_{N,\beta}}(\pi_1^{-1}\nueps,\nu^0)$ as well.

\hfill

The crucial lemma that is invoked in \textit{Step 2} asserts the convergence of $\Phieps(t)$ toward $\Phizero(t)$ in $d^{\mueps}_{N,\beta}$. We state this result now, but defer its proof to the end of the section.

\begin{lemma} \label{lem:epsilon->0:alpha_0(solutions-limit-solution)->0}
Under the same hypothesis of \cref{thm:epsilon->0:invariant-measures->limit-measure}, suppose that $\rx^\varepsilon=(\pi_1\rx^\varepsilon,\pi_2\rx^\varepsilon)\in L^2(\Omega;\Honeeps)$ is a random variable distributed as~$\nuzero$, the invariant probability for~\eqref{eqn:react-diff:eta^0_epsilon} as in~\cref{thm:existence-inv-measure:nu^(0-epsilon)}. Let $\Phieps (t)=(\ueps (t),\etaeps(t))$ and $\Phizero(t)= (u^0(t),\etazero(t))$ respectively be the solutions of~\eqref{eqn:react-diff:epsilon:mu} and~\eqref{eqn:react-diff:eta^0_epsilon} both with initial condition $\rx^\varepsilon$. Then, for all $N$ sufficiently large and $\beta$ sufficiently small independent of $\varepsilon$, the following holds for all $T>0$:
\begin{equation} \label{lim:epsilon->0:alpha_0(solutions-limit-solution)->0}
\sup_{0\le t\le T}\E\, d^{\mueps}_{N,\beta} \big(\Phieps(t),\Phizero(t)\big)\le C e^{cT}\varepsilon^{1/12},\quad\text{as}\quad \varepsilon\to 0,
\end{equation}
where $d^{\mueps}_{N,\beta} $ is the distance defined in~\eqref{form:d_N,beta:mu}, $c$ and $C$ are positive constants independent of $T$ and $\varepsilon$.
\end{lemma}

Assuming that \cref{lem:epsilon->0:alpha_0(solutions-limit-solution)->0} holds, let us prove \cref{thm:epsilon->0:invariant-measures->limit-measure}.

\begin{proof}[Proof of~\cref{thm:epsilon->0:invariant-measures->limit-measure}] Fix $N,\beta>0$ such that \eqref{ineq:react-diff:epsilon:geometric-ergodicity},~\eqref{ineq:react-diff:epsilon:geometric-ergodicity:beta/2} and \eqref{lim:epsilon->0:alpha_0(solutions-limit-solution)->0} hold true for $N$ and $2\beta$. Let $\nuzero$ be the invariant probability of~\eqref{eqn:react-diff:eta^0_epsilon} as in~\cref{thm:existence-inv-measure:nu^(0-epsilon)}. Since $\pi_1^{-1}\nuzero\sim\nu^0$ by~\cref{lem:pi_1nu^(0-epsilon)=nu^0}, we apply \cref{lem:W_0(nu_1-nu_2)<W_epsilon(nu_1-nu_2)} to estimate
\begin{align*}
\W_{d_{N,\beta}}(\pi_1^{-1}\nueps,\nu^0)\le 
\W_{d^{\mueps}_{N,\beta}}(\nueps,\nuzero).
\end{align*} 
By the invariance of $\nueps$ and $\nuzero$ corresponding to $P^{\mueps}_t$ and $\Pzero_t$, respectively, we invoke~\cref{lem:W_epsilon:triangle-ineq} to see that
\begin{align*}
\W_{d^{\mueps}_{N,\beta}}(\nueps,\nuzero)&= \W_{d^{\mueps}_{N,\beta}}\big((P^{\mueps}_t)^*\nueps,(\Pzero_t)^*\nuzero\big)\\
&\le  C\Big(\W_{d^{\mueps}_{N,2\beta}}\big((P^{\mueps}_t)^*\nueps,(P^{\mueps}_t)^*\nuzero)+\W_{d^{\mueps}_{N,2\beta}}\big((P^{\mueps}_t)^*\nuzero,(\Pzero_t)^*\nuzero\big)\Big)\\
&\le Ce^{-ct}\W_{d^{\mueps}_{N,\beta}}(\nueps,\nuzero)+C\W_{d^{\mueps}_{N,2\beta}}\big((P^{\mueps}_t)^*\nuzero,(\Pzero_t)^*\nuzero\big).
\end{align*}
Note that in obtaining the final inequality above, we applied~\eqref{ineq:react-diff:epsilon:geometric-ergodicity:beta/2} for $t\ge T^*=T^*(N,\beta)$. Since $c$ and $C$ are independent of the measures $\nu^\varepsilon$, for all $\varepsilon\ge 0$, there exists a sufficiently large time, $t^*$, such that $Ce^{-ct^*}<1/2$. It follows that 
$$\W_{d^{\mueps}_{N,\beta}}(\nueps,\nuzero)\le \frac{1}{2}\W_{d^{\mueps}_{N,\beta}}(\nueps,\nuzero)+C\W_{d^{\mueps}_{N,2\beta}}\big((P^{\mueps}_t)^*\nuzero,(\Pzero_t)^*\nuzero\big),$$
for all $t\ge t^*$. Thus, thanks to~\eqref{form:W_d:mu}, we have
$$\frac{1}{2} \W_{d^{\mueps}_{N,\beta}}(\nueps,\nuzero)\le C\W_{d^{\mueps}_{N,2\beta}}\big((P^{\mueps}_t)^*\nuzero,(\Pzero_t)^*\nuzero\big)\le C\,\E\, d^{\mueps}_{N,2\beta} \big(\Phieps(t),\Phizero(t)\big),$$
where $\Phieps(0)=\Phizero(0)$ are identically distributed as $\nuzero$. Consequently, by virtue of \eqref{lim:epsilon->0:alpha_0(solutions-limit-solution)->0} from~\cref{lem:epsilon->0:alpha_0(solutions-limit-solution)->0}, we obtain 
\begin{align} \label{ineq:W_0(nu^epsilon,nu^0)<d_epsilon}
\W_{d_{N,\beta}}(\pi_1^{-1}\nueps,\nu^0)\le C\,\E\, d^{\mueps}_{N,2\beta} \big(\Phieps(t),\Phizero(t)\big) \le Ce^{ct}\varepsilon^{1/12},\quad \text{as} \quad\varepsilon\to0,
\end{align}
for all sufficiently large $t$, as desired.
\end{proof}

Finally, we logically conclude the proof of \cref{thm:epsilon->0:invariant-measures->limit-measure} by providing the proof of~\cref{lem:epsilon->0:alpha_0(solutions-limit-solution)->0}.

\begin{proof}[Proof of~\cref{lem:epsilon->0:alpha_0(solutions-limit-solution)->0}] 
Recalling \eqref{form:d_N,beta:mu}, we see from the Cauchy-Schwarz inequality that
\begin{align*}
\big{|}\E\, d^{\mueps}_{N,\beta} \big(\Phieps(t),\Phizero(t)\big)\big{|}^2
& \le \left(\E d_N^\mu(U,\Ut)\right)\left(\E\big[1+\exp\left({\beta\Psi_0(U)}\right)+\exp\left({\beta\Psi_0(\Ut)}\right)\big]\right)\\
&\le N \E\|\Phieps(t)-\Phizero(t)\|_{\Hzeroeps}\Big( 1+ \E\exp\left({\beta\Psi_0(\Phieps(t))}\right)+\E\exp\left({\beta\Psi_0(\Phizero(t))}\right) \Big)\\
&=N I_1^\varepsilon(t)(1+ I_2^\varepsilon(t)+ I_3^\varepsilon(t)).
\end{align*}
We note that $\rx^\varepsilon\sim\nuzero$ satisfies the hypothesis of \cref{thm:epsilon->0:solution->limit-solution:finite-time}, thanks to~\eqref{lim:epsilon->0:invariant-measures:H^1_epsilon-bound} and~\eqref{lim:epsilon->0:invariant-measures:energy-exponential-bound}. In view of~\cref{thm:epsilon->0:solution->limit-solution:finite-time}, we readily obtain 
$$ I_1^\varepsilon(t)\le C e^{cT}\varepsilon^{1/6},\quad\text{as}\quad\varepsilon\to 0,$$
for all $t\leq T$. To estimate $I_2^\varepsilon(t)$, we employ~\eqref{ineq:exponential-bound:H^0_epsilon} to infer that for all sufficiently small $\beta>0$ sufficiently small, independent of $\varepsilon$, we have
\begin{align*}
I_2^\varepsilon(t)= \E\exp\left({\beta\Psi_0(\Phieps(t))}\right)&\le e^{-ct} \E\exp\left({\beta\Psi_0(\rx^\varepsilon)}\right)+C\le e^{-ct} \int_{\Hzeroeps}\close \exp\left({\beta\|U\|^2_{\Hzeroeps}}\right)\nuzero(\d U)+C\le C,
\end{align*}
for all $t\leq T$ and some positive constants $c,C$ independent of $\varepsilon$ and $t$ by virtue of the $\varepsilon$-uniform exponential moment~\eqref{lim:epsilon->0:invariant-measures:energy-exponential-bound} with respect to $\nuzero$.

Similarly, let $\Psitilde_0(U)= \frac{1}{2}\kappa_1\|u\|^2_H+\frac{1}{2}\kappa_2\|\eta\|^2_{\Mzeroeps}$ be as in~\eqref{form:Psitilde_0}. Thanks to \eqref{ineq:u^0,eta^0:exponential-estimate} from~\cref{lem:epsilon->0:invariant-measure:enegery-exponential-bound}, there exist $\kappa_1, \kappa_2>0$, sufficiently small, such that
\begin{align*}
\E \exp\left({ \Psitilde_0(\Phizero(t))}\right)\le e^{-ct}\E \exp\left({ \Psitilde_0(\rx^\varepsilon)}\right)+C,
\end{align*}
for all $t\leq T$. Thus, upon choosing $\beta>0$ small enough such that
$$\beta\Psi_0(U)=\frac{1}{2}\beta\big(\|u\|^2_H+(1-\kappa)\|\eta\|^2_{\Mzeroeps}\big)\le \frac{1}{2}\kappa_1\|u\|^2_H+\frac{1}{2}\kappa_2\|\eta\|^2_{\Mzeroeps}=\Psitilde_0(U),$$
we obtain
\begin{align*}
I_3^\varepsilon(t)\le \E \exp\left( \Psitilde_0(\Phizero(t))\right)&\le e^{-ct}\E \exp\left({ \Psitilde_0(\rx^\varepsilon)}\right)+C\\
&= e^{-ct} \int_{\Hzeroeps} \exp\left({\Psitilde_0(U)}\right)\nuzero(\d U)+C\le C,
\end{align*}
for all $t\leq T$. Note that in the final relation above, we invoked the uniform exponential moment~\eqref{ineq:exponential-moment:nu^(0-epsilon):Psitilde_0} with respect to $\nuzero$.

Lastly, we collect the above estimates to arrive at
\begin{align*}
&\big{|}\E\, d^{\mueps}_{N,\beta} \big(\Phieps(t),\Phizero(t)\big)\big{|}^2\le Ce^{cT}\varepsilon^{1/6}, \quad \text{as} \quad\varepsilon\to0,
\end{align*}
for all $0\leq t\leq T$. This establishes~\eqref{lim:epsilon->0:alpha_0(solutions-limit-solution)->0}, as desired.
\end{proof}

Reasoning as in the proofs of~\cref{thm:epsilon->0:solution->limit-solution:finite-time} and~\cref{lem:epsilon->0:alpha_0(solutions-limit-solution)->0}, we also obtain the following result.

\begin{corollary}\label{cor:ineq:d_epsilon(u^epsilon_x,u^0_x)} Under the same hypothesis as \cref{lem:epsilon->0:alpha_0(solutions-limit-solution)->0}, one has that for all $N$ sufficiently large and $\beta$ sufficiently small, independently of $\varepsilon$, the following holds:
\begin{equation} \label{ineq:d_epsilon(u^epsilon_x,u^0_x)}
\sup_{0\le t\le T}\E\,d^{\mueps}_{N,\beta}\big(\Phieps(t;\x^\varepsilon),\Phizero(t;\x^\varepsilon)\big)\le Ce^{cT}\varepsilon^{1/12},
\end{equation}
for all $T$, $R$ and deterministic initial conditions $\{\x^\varepsilon\}_{\varepsilon\in (0,1)}$ such that $\|\x^\varepsilon\|_{\Honeeps}\le R$, where the positive constants $c$ and $C$ are independent of $T$ and $\varepsilon$.
\end{corollary}

\subsubsection{Proof of~\cref{thm:epsilon->0:sup_t.E[f(u^epsilon(t))-f(u^0(t))]}}\label{sec:proof-of-E[u^epsilon-u^0]}
To prove this theorem, we will first prove that the \eqref{lim:epsilon->0:W_0(P^epsilon_t,P^0_t)} holds by making use of the results on geometric ergodicity involving $\nu^\varepsilon$ and $\nu^0$, as well as the convergence of the first marginal, $\pi_1^{-1}\nu^\varepsilon$, towards $\nu^0$. Second, we will employ the estimate~\eqref{ineq:W_0(nu_1,nu_2):dual} to show that for any $d_{N,\beta}$-Lipschitz function, $f$, the difference in terms of $f$, as in~\eqref{lim:epsilon->0:sup_t.E[f(u^epsilon(t))-f(u^0(t))]}, can be bounded by the corresponding difference in $\W_{d_{N,\beta}}$.

 We begin by fixing $T_1>0$, to be chosen later. We invoke~\eqref{form:W_0(nu_1,nu_2)} and~\eqref{ineq:d_0(x1,x2)<d_epsilon(x1,x2)} so that
\begin{align}
\sup_{0\le t\le T_1}\W_{d_{N,\beta}}(\pi_1^{-1}(\Peps_t)^*\delta_{\x^\varepsilon},(P_t^0)^*\delta_{\pi_1 \x^\varepsilon})
&\le \sup_{0\le t\le T_1}\E\, d_{N,\beta}(\ueps(t;\x^\varepsilon),u^0(t;\pi_1\x^\varepsilon)) \notag\\
&\le\sup_{0\le t\le T_1}\E \, d^{\mueps}_{N,\beta}(\Phieps(t;\x^\varepsilon),\Phizero(t;\x^\varepsilon)).\label{ineq:W_0(P^epsilon_T*,P^0_T*)}
\end{align}
Now we recall that $\Phizero(t;\x^\varepsilon)$ is the solution of~\eqref{eqn:react-diff:eta^0_epsilon} with initial condition $\x^\varepsilon$. On the other hand, by the generalized triangle inequality~\eqref{ineq:W_0:triangle-ineq}, for each $t\ge 0$, it holds that
\begin{equation} \label{ineq:W_0(P^epsilon_t,P^0_t):triangle}
\begin{aligned}
&\W_{d_{N,\beta}}(\pi_1^{-1}(\Peps_t)^*\delta_{\x^\varepsilon},(P_t^0)^*\delta_{\pi_1 \x^\varepsilon})\\&\le C\Big(\W_{d_{N,2\beta}}(\pi_1^{-1}(\Peps_t)^*\delta_{\x^\varepsilon},\pi_1^{-1}\nueps)+\W_{d_{N,4\beta}}(\pi_1^{-1}\nueps,\nu^0)+\W_{d_{N,4\beta}}(\nu^0,(P_t^0)^*\delta_{\pi_1 \x^\varepsilon})\Big).
\end{aligned}
\end{equation}
In view of~\eqref{ineq:react-diff:geometric-ergodicity:beta/2}, we readily see that
\begin{align*}
\W_{d_{N,4\beta}}(\nu^0,(P_t^0)^*\delta_{\pi_1 \x^\varepsilon})&=\W_{d_{N,4\beta}}((P^0_t)^*\nu^0,(P_t^0)^*\delta_{\pi_1 \x^\varepsilon}) \\
&\le Ce^{-c(t/2)}\W_{d_{N,2\beta}}((P^0_{t/2})^*\nu^0,(P^0_{t/2})^*\delta_{\pi_1\x^\varepsilon})\leq Ce^{-ct}\W_{d_{N,\beta}}(\nu^0,\delta_{\pi_1\x^\varepsilon}),
\end{align*}
for $t\geq\tilde{T}$, where $\tilde{T}$ is as in \eqref{ineq:react-diff:geometric-ergodicity:beta/2}.

For an arbitrary random variable $X\sim\nu^0$, we recall \eqref{form:d_N,beta} and use~\eqref{form:W_0(nu_1,nu_2)}, Holder's inequality, the fact that $\|\x^\varepsilon\|_{\Honeeps}\le R$, and the exponential moment bound~\eqref{ineq:react-diff:exponential-bound:nu^0} for sufficiently small $\beta$ to obtain
\begin{align*}
&\big{|}\W_{d_{N,\beta}}(\nu^0,\delta_{\pi_1\x^\varepsilon})\big{|}^2\\
&\le  N\left(\E\|X-\pi_1\x^\varepsilon\|_H\right)\left[1+\exp\left({\frac{\beta}{2}\|\pi_1\x^\varepsilon\|^2_H}\right)+\E\exp\left(\frac{\beta}{2}\|X\|^2_H\right)\right] \\
&\le N \Big( \int_H\|u\|_H\nu^0(\d u)+\|\pi_1\x^\varepsilon\|_H  \Big)\left[1+\exp\left({\frac{\beta}{2}\|\pi_1\x^\varepsilon\|^2_H}\right)+\int_H \exp\left({\frac{\beta}{2}\|u\|^2_H}\right)\nu^0(\d u)\right]\notag\\
&\le c(R,N,\beta).
\end{align*}
It follows that for $t\geq \tilde{T}$ with $\tilde{T}$ given by \eqref{ineq:react-diff:geometric-ergodicity:beta/2} we have
\begin{equation} \label{ineq:W_0(P^epsilon_t,P^0_t):triangle:1}
\W_{d_{N,2\beta}}(\nu^0,(P_t^0)^*\delta_{\pi_1 \x^\varepsilon})\le Ce^{-ct},\quad t\ge \tilde{T},
\end{equation}
where $C=C(R,N,\beta)$ and $c=c(R,N,\beta)$ do not depend on $\varepsilon$ and time $t$.

Concerning $\W_{d_{N,2\beta}}(\pi_1^{-1}(\Peps_t)^*\delta_{\x^\varepsilon},\pi_1^{-1}\nueps)$, we invoke~\eqref{ineq:W_0(nu_1,nu_2)<W_epsilon(nu_1,nu_2)}  together with~\eqref{ineq:react-diff:epsilon:geometric-ergodicity:beta/2} to deduce 
\begin{align*}
\W_{d_{N,2\beta}}(\pi_1^{-1}(\Peps_t)^*\delta_{\x^\varepsilon},\pi_1^{-1}\nueps)&\le \W_{d^{\mueps}_{N,2\beta}}((\Peps_t)^*\delta_{\x^\varepsilon},\nueps)\\
&=\W_{d_{N,2\beta}^{\mueps}}((\Peps_t)^*\delta_{\x^\varepsilon},(\Peps_t)^*\nueps)\\
&\le Ce^{-ct}\W_{d_{N,\beta}^{\mueps}}(\delta_{\x^\varepsilon},\nueps),\quad 
\end{align*}
for all $t\ge T^*$, where $T^*$ is the time asserted by \cref{thm:react-diff:epsilon:geometric-ergodicity}. Now, we may argue as in the above estimate for $\W_{d_{N,2\beta}}(\nu^0,\delta_{\pi_1\x^\varepsilon})$ and employ~\eqref{form:d_N,beta:mu}, \eqref{form:W_d:mu}, and Holder's inequality to see that for an arbitrary random variable $X\sim\nueps$ we have
\begin{align*}
&\big{|}\W_{d^{\mueps}_{N,\beta}}(\delta_{\x^\varepsilon},\nueps)\big{|}^2\\
&\le  N\left(\E\|X-\x^\varepsilon\|_{\Hzeroeps}\right)\left[1+\exp\left({\frac{\beta}{2}\|\x^\varepsilon\|^2_{\Hzeroeps}}\right)+\E\exp\left({\frac{\beta}{2}\|X\|^2_{\Hzeroeps}}\right)\right]   \\
&\le N \left( \int_{\Hzeroeps}\close\|\x\|_H\nueps(\d \x)+\|\x^\varepsilon\|_{\Hzeroeps} \right)\left[1+\exp\left({\frac{\beta}{2}\|\x^\varepsilon\|^2_{\Hzeroeps}}\right)+\int_{\Hzeroeps}\close \exp\left({\frac{\beta}{2}\|\x\|^2_{\Hzeroeps}}\right)\nueps(\d \x)\right]\notag\\
&\le c(R,N,\beta).
\end{align*}
In the above right-hand side, we employed the $\varepsilon$-uniform exponential moment bound of $\nueps$ (see \eqref{ineq:exponential-bound:H^1_epsilon:nu^mu}), for sufficiently small $\beta$ that is independent of $\varepsilon$. It follows that
\begin{equation} \label{ineq:W_0(P^epsilon_t,P^0_t):triangle:2}
\W_{d_{N,2\beta}}(\pi_1^{-1}(\Peps_t)^*\delta_{\x^\varepsilon},\pi_1^{-1}\nueps)\le Ce^{-ct}\W_{d_{N,\beta}^{\mueps}}(\delta_{\x^\varepsilon},\nueps) \le Ce^{-ct},
\end{equation}
for all $t\ge T^*$, for some constants $C$ and $c$ are independent of $\varepsilon$ and $t$. 

Suppose that $T_1\ge \max\{\tilde{T},T^*\}$. Upon collecting~\eqref{ineq:W_0(P^epsilon_t,P^0_t):triangle}, \eqref{ineq:W_0(P^epsilon_t,P^0_t):triangle:1}, \eqref{ineq:W_0(P^epsilon_t,P^0_t):triangle:2}, we infer 
\begin{align*}
\W_{d_{N,\beta}}\big(\pi_1^{-1}(\Peps_t)^*\delta_{\x^\varepsilon},(P_t^0)^*\delta_{\pi_1 \x^\varepsilon}\big)&\le C\Big(e^{-cT_1}+\W_{d_{N,2\beta}}\big(\pi_1^{-1}\nueps,\nu^0\big)\Big),
\end{align*}
for all $t\ge T_1$. Together with~\eqref{ineq:W_0(P^epsilon_T*,P^0_T*)}, we arrive at the bound
\begin{align*}
&\sup_{t\ge 0}\W_{d_{N,\beta}}\big(\pi_1^{-1}(\Peps_t)^*\delta_{\x^\varepsilon},(P_t^0)^*\delta_{\pi_1 \x^\varepsilon}\big)\\
&\le C\Big(e^{-cT_1}+\W_{d_{N,2\beta}}\big(\pi_1^{-1}\nueps,\nu^0\big)\Big)+\sup_{0\le t\le T_1}\E \, d^{\mueps}_{N,\beta}(\Phieps(t;\x^\varepsilon),\Phizero(t;\x^\varepsilon)).
\end{align*}
In light of~\eqref{ineq:W_0(nu^epsilon,nu^0)<d_epsilon} and~\eqref{ineq:d_epsilon(u^epsilon_x,u^0_x)}, for all $T_1$ sufficiently large, it holds that
\begin{align*}
\sup_{t\ge 0}\W_{d_{N,\beta}}\big(\pi_1^{-1}(\Peps_t)^*\delta_{\x^\varepsilon},(P_t^0)^*\delta_{\pi_1 \x^\varepsilon}\big)\le C\big(e^{-cT_1}+e^{\tilde{c}T_1}\varepsilon^{1/12}\big).
\end{align*}
We emphasize that in the above, $C$, $c$ and $\tilde{c}$ are positive constants independent of $T_1$ and $\varepsilon$.

To conclude~\eqref{lim:epsilon->0:W_0(P^epsilon_t,P^0_t)}, fixing any $\lambda_1\in (0,1/12)$ and setting $T_1=\tilde{c}^{-1}\log (\varepsilon^{-\lambda_1})$, we deduce
\begin{align*}
\sup_{t\ge 0}\W_{d_{N,\beta}}\big(\pi_1^{-1}(\Peps_t)^*\delta_{\x^\varepsilon},(P_t^0)^*\delta_{\pi_1 \x^\varepsilon}\big)\le C\big(\varepsilon^{\frac{c}{\tilde{c}}\lambda_1}+\varepsilon^{1/12-\lambda_1}\big).
\end{align*}
This establishes the limit~\eqref{lim:epsilon->0:W_0(P^epsilon_t,P^0_t)} by taking $\lambda_1=\frac{1/12}{1+c/\tilde{c}}$ and $\lambda_*=1/12-\lambda_1$. 

Turning to~\eqref{lim:epsilon->0:sup_t.E[f(u^epsilon(t))-f(u^0(t))]}, fix a function $f\in C_b(H)$ such that $L_f=[f]_{\text{Lip}_{d_{N,\beta}}}<\infty$. By virtue of~\eqref{ineq:W_0(nu_1,nu_2):dual}, we have
\begin{align*}
\big{|}\E f(\ueps(t;\x^\varepsilon))-\E f(u^0(t;\pi_1\x^\varepsilon))\big{|}&=\big{|}\int_{H}f(u)\pi_1^{-1}(P^{\mueps}_t)^*\delta_{\x^\varepsilon}(\d u)-\int_H f(v)(P^0_t)^*\delta_{\pi_1 \x^\varepsilon}(\d v) \big{|}\\
&\le L_f\W_{d_{N,\beta}}(\pi_1^{-1}(\Peps_t)^*\delta_{\x^\varepsilon},(P_t^0)^*\delta_{\pi_1 \x^\varepsilon}).
\end{align*}
In light of~\eqref{lim:epsilon->0:W_0(P^epsilon_t,P^0_t)}, we immediately obtain limit~\eqref{lim:epsilon->0:sup_t.E[f(u^epsilon(t))-f(u^0(t))]}, thus completing the proof.
\qed



\section{Estimates for stochastically forced memoryless system}\label{sec:0-equation} 
In this section, we collect important results concerning~\eqref{eqn:react-diff} that were invoked in~\cref{sec:epsilon->0:inv-measure} to obtain the convergence of $\nueps$ towards $\nu^0$. 

We begin by establishing~\cref{thm:react-diff:geometric-ergodicity}, where we assert that $\nu^0$ is the unique invariant probability measure in $H$ for~\eqref{eqn:react-diff}. We point out that the proof of~\cref{thm:react-diff:geometric-ergodicity} is similar to that of \cref{thm:react-diff:epsilon:geometric-ergodicity} and makes use of the framework of the weak Harris theorem (see~\cite{cerrai2020convergence,hairer2008spectral,hairer2011theory}). 
For the sake of completeness, we provide a sketch of the argument after stating~\cref{thm:react-diff:geometric-ergodicity} below.

\begin{theorem} \label{thm:react-diff:geometric-ergodicity}
Assume the hypotheses of~\cref{thm:react-diff:epsilon:geometric-ergodicity}. Then for all $N$ sufficiently large and $\beta$ sufficiently small, there exists a positive constant $\tilde{T}=\tilde{T}(N,\beta)$ large enough such that the following holds for all $\nu_1,\nu_2\in \Pcal r(H)$,
\begin{equation} \label{ineq:react-diff:geometric-ergodicity:beta/2}
\W_{d_{N,\beta}} \big((P^0_t)^*\nu_1,(P^0_t)^*\nu_2\big)\le Ce^{-c t}\W_{d_{N,\beta/2}}(\nu_1,\nu_2),\quad t\ge \tilde{T},
\end{equation}
for some positive constants $c=c(N,\beta), \,C=C(N,\beta)$ independent of $\nu_1,\nu_2$ and $t$. Here, $\W_{d_{N,\beta}}$ is the Wasserstein distance as in~\eqref{form:W_0(nu_1,nu_2)}. Consequently
\begin{equation} \label{ineq:react-diff:geometric-ergodicity}
\W_{d_{N,\beta}} \big((P^0_t)^*\nu_1,(P^0_t)^*\nu_2\big)\le Ce^{-c t}\W_{d_{N,\beta}}(\nu_1,\nu_2),\quad t\ge \tilde{T}.
\end{equation}
\end{theorem}

\begin{proof}[Sketch of the proof of~\cref{thm:react-diff:geometric-ergodicity}] We note that~\eqref{ineq:react-diff:geometric-ergodicity} is a consequence of~\eqref{ineq:react-diff:geometric-ergodicity:beta/2} owing to the fact that $d_{N,\beta}(u,v)\ge d_{N,\beta/2}(u,v)$. 

In order to prove~\eqref{ineq:react-diff:geometric-ergodicity}, similarly to the proof of~\cref{thm:react-diff:epsilon:geometric-ergodicity}, it is sufficient to gather the following three ingredients: Lyapunov structure, $d_N$-contractivity, and $d_N$-smalleness.

\noindent\textbf{Lyapunov structure.} For all $\beta$ sufficiently small, the followings hold for all initial conditions $u_0\in H$
\begin{equation} \label{ineq:Lyapunov-bound:0-eq}
\E\, e^{\frac{\beta}{2}\|u^0(t)\|^2_H}\le  e^{-c t}e^{\frac{\beta}{2}\|u_0\|^2_H}+C,
\end{equation}
and
\begin{align}\label{ineq:Lyapunov-bound:0-eq:b}
\E\, e^{\frac{\beta}{2}\|u^0(t)\|^2_H}\le  Ce^{\frac{\beta}{2}\|u_0\|^2_He^{-ct}},
\end{align}
where $c,C$ do not depend on $t$ and $u_0$. A slightly more generalized statement of~\eqref{ineq:Lyapunov-bound:0-eq} with $u_0$ being random is given below in \cref{lem:react-diff:exponential-bound}.

\noindent\textbf{$d_N$-contractivity.} The distance-like function $d_N$ as in~\eqref{form:d_0(x,y)} is contracting for $P^0_t$, that is, for all $N$ and $t$ sufficiently large, we have
\begin{equation} \label{ineq:contracting:W_d<1/2.d:u^0}
\W_{d_N}\big(P_t^0(u,\cdot),P_t^0(v,\cdot)\big)\le \frac{1}{2} d_{N}(u,v),
\end{equation}
for all $u,\,v$ such that $d_N(u,v)<1$.

\noindent\textbf{$d_N$-smallness.} For all $R>0$ and $N>0$, there exists a positive constant $t_0=t_0(R,N)$ sufficiently large and $\gamma_0=\gamma_0(R)\in (0,1)$ such that
\begin{equation} \label{ineq:error-in-law:u^0}
\sup_{u,v\in B_R}\W_{d_N}\big(P_t^0(u,\cdot),P_t^0(v,\cdot)\big) \le 1-\gamma_0, \quad t\ge t_0,
\end{equation}
where $B_R=\{u\in H:\|u\|\le R\}$.

The relation~\eqref{ineq:Lyapunov-bound:0-eq} can be established by employing an entirely similar to the one made for proving \eqref{ineq:exponential-bound:H^0_epsilon} in \cref{lem:moment-bound:H^0_epsilon}. On the other hand, \eqref{ineq:Lyapunov-bound:0-eq:b} follows along the lines of the proof of \cref{lem:exponential-bound:H^0_epsilon:beta/2} and exploits the exponential martingale inequality. 

To verify~\eqref{ineq:contracting:W_d<1/2.d:u^0} and \eqref{ineq:error-in-law:u^0}, we can employ the same strategy that was used to prove~\cref{prop:contracting-d-small} by making use of condition~\nameref{cond:Q:ergodicity}. 

Lastly, we may then mimic the proof of \cref{thm:react-diff:epsilon:geometric-ergodicity} (see \cite[Theorem 4.8]{hairer2011asymptotic}) by combining the above properties in order to conclude~\eqref{ineq:react-diff:geometric-ergodicity}, where $d_{N,\beta}$ is given by~\eqref{form:d_N,beta}.
\end{proof}

We now turn our attention to the solution $u^0$ of~\eqref{eqn:react-diff} with respect to random initial conditions. In the series of lemmas below, we collect some useful moment bounds on $u^0$. We have employed these results in~\cref{sec:epsilon->0:inv-measure} to establish the convergence of $\numu$ toward $\nu^0$.

\begin{lemma} \label{lem:react-diff:energy-bound}
Assume the hypotheses of~\cref{prop:well-posed} and let $\ru\in L^2(\Omega;H)$. Let $u^0(t;u_0)$ denote the solution of~\eqref{eqn:react-diff} corresponding to the random initial condition $\ru$. Then
\begin{equation} \label{ineq:react-diff:energy-bound:random-initial}
\begin{aligned}
\E\|u^0 (t)\|^2_H+2\int_0^t\E\|A^{1/2}u^0 (s)\|^2_H+a_2\E\|u^0 (s)\|^{p_0+1}_{L^{p_0+1}}\emph{d}s
\le\E\|\ru\|^2_H+2\Big(a_3|\domain|+\emph{Tr}(QQ^*)\Big) t,
\end{aligned}
\end{equation}
for all $t\ge 0$, and
\begin{equation}\label{ineq:react-diff:u^0_t<e^(-ct)}
\E\|u^0 (t)\|^2_H\le e^{-2\alpha_1 t}\E\|\ru\|^2_H+\frac{a_3|\domain|+\emph{Tr}(QQ^*)}{\alpha_1},
\end{equation}
where $\alpha_1$ is the first eigenvalue of $A$, $p_0$ is as in \nameref{cond:phi:1}, and $a_2,a_3$ are as in~\nameref{cond:phi:2}. Furthermore, for all even integers $n\ge 2$, there exist positive constants $c_n^0, C_n^0$, independent of $t$ and $u_0$, such that
\begin{align} \label{ineq:react-diff:moment-bound:H:n}
\E\|u^0 (t)\|^n_H\le e^{-c^0_n t}\E\|\ru\|^n_H+C^0_n.
\end{align}
\end{lemma}

\begin{proof}
By It\^o's formula
\begin{align*}
\frac{1}{2}\d\|u^0(t)\|^2_H+\|A^{1/2}u^0(t)\|^2_H\d t&=\la\f(u^0(t)),u^0(t)\ra_H\d t+\la u^0(t),Q\d w(t)\ra_H+\Tr(QQ^*)\d t.
\end{align*}
Then taking expectation yields
\begin{align*}
\frac{1}{2}\frac{\d}{\d t}\E\|u^0 (t)\|^2_H+\E\|A^{1/2}u^0 (t)\|^2_H&=\E\la\f(u^0 (t)),u^0 (t)\ra_H+\Tr(QQ^*).
\end{align*}
In view of~\nameref{cond:phi:2}, we see that
\begin{align*}
\la\f(u),u\ra_H\le -a_2\|u\|^{p_0+1}_{L^{p_0+1}(\domain)}+a_3|\domain|.
\end{align*}
Thus
\begin{align}
\frac{1}{2}\frac{\d}{\d t}\E\|u^0 (t)\|^2_H+\E\|A^{1/2}u^0 (t)\|^2_H\le -a_2\|u^0 \|^{p+1}_{L^{p+1}(\domain)}+a_3|\domain| +\Tr(QQ^*).\notag
\end{align}
Integrating over the interval $[0,t]$ furnishes~\eqref{ineq:react-diff:energy-bound:random-initial}. On the other hand, by Poincare's inequality, $\|A^{1/2}u\|^2_H\ge \alpha_1\|u\|^2_H$. This, in turn, produces~\eqref{ineq:react-diff:u^0_t<e^(-ct)} by virtue of Gronwall's inequality.

Turning to~\eqref{ineq:react-diff:moment-bound:H:n}, the case $n=2$ was already proved in~\eqref{ineq:react-diff:u^0_t<e^(-ct)}. Consider any even integer $n\ge 4$, we apply It\^o's formula to $\|u^0(t)\|^n_H$ and obtain
\begin{align}
\d \|u^0(t)\|^n_H&= 2n\|u^0(t)\|^{n-2}_H\big(-\|A^{1/2}u^0(t)\|^2\d t+\la\f(u^0(t)),u^0(t)\ra_H\d t+\Tr(QQ^*)\d t\big)\nt\\
&\qquad+2n(n-2)\|u^0(t)\|^{n-4}_H\|Qu^0(t)\|^2_Hdt+2n\|u^0(t)\|^{n-2}_H\la u^0(t),Q\d w(t)\ra_H. \label{ineq:d.|u^0|^n_H}
\end{align}
Observe that 
\begin{align}\label{est:double:star}
\|Qu^0(t)\|^2_H \le \Tr(QQ^*)\|u^0(t)\|^2_H.
\end{align}
Thus, upon invoking \eqref{est:double:star}, taking expectations on both sides of~\eqref{ineq:d.|u^0|^n_H}, then applying Young's inequality, we obtain
\begin{align*}
\frac{\d}{\d t}\E \|u^0(t)\|^n_H \le -c \E\|u^0(t)\|^n_H+C,
\end{align*}
for some positive constants $c,C$ independent of $t$ and $u_0$. An application of Gronwall's inequality yields~\eqref{ineq:react-diff:moment-bound:H:n}, as desired.
\end{proof}

We also obtain the following result.

\begin{lemma}\label{lem:react-diff:exponential-bound}
Assume the hypotheses of~\cref{prop:well-posed} and let $\ru\in L^2(\Omega; H)$. Then there exists $\beta_0>0$ such that for all $0<\beta\leq\beta_0$, there exist positive constants $c$ and $C$, independent of $\ru$ and $t$, such that
\begin{equation} \label{ineq:react-diff:exponential-bound}
\E\, \exp\left({\beta\|u^0 (t)\|^2_H}\right)\le e^{-ct} \E\, \exp\left({\beta\|\ru\|^2_H}\right)+C.
\end{equation}
Consequently
\begin{align} \label{ineq:react-diff:exponential-bound:nu^0}
\int_H \exp\left({\beta\|u\|^2_H}\right)\nu^0(\emph{d} u)<\infty,
\end{align}
where $\nu^0$ is the unique invariant probability measure of~\eqref{eqn:react-diff}.
\end{lemma}

We omit the proof of \cref{lem:react-diff:exponential-bound} since it makes use of arguments similar to ones made for the proofs of~\cref{lem:moment-bound:H^0_epsilon}. We next turn to estimates of $u^0$ in $H^1$. In \cref{lem:react-diff:|u|^2_(H^1)-bound.and.int_0^t|Au(s)|^2ds-bound:random-initial-cond}, we assert standard energy bounds of $u^0$ in $H^1$, whose argument is similar to that of~\cref{lem:moment-boud:H^1_epsilon}.

\begin{lemma} \label{lem:react-diff:|u|^2_(H^1)-bound.and.int_0^t|Au(s)|^2ds-bound:random-initial-cond} 
Assume the hypotheses of~\cref{prop:well-posed} and let $\ru\in L^2(\Omega;H^1)$. Then there exists a positive constant $c$, independent of $\ru$, such that
\begin{equation}\label{ineq:react-diff:|Phi|^2_(H^1_epsilon)-bound:random-initial-cond}
\E\|u^0 (t)\|^2_{H^1}+\int_0^t\E\|u^0 (s)\|^2_{H^2}\emph{ d} s\le c\big(\E\|\ru\|^2_{H^1}+t),
\end{equation}
and 
\begin{equation} \label{ineq:react-diff:|u^0|_H1<c}
\E\|u^0 (t)\|^2_{H^1}\le c\big(\E\|\ru\|^2_{H^1}+1),
\end{equation}
hold for all $t\geq0$.
\end{lemma}

\begin{proof}
By It\^o's formula
\begin{align*}
&\frac{1}{2}\d\|u^0(t)\|^2_{H^1}+\|Au^0(t)\|^2_H\d t\\
&=\la\f'(u^0(t))\grad u^0(t),\grad u^0(t)\ra_H\d t+\la u(t),Q\d w(t)\ra_{H^1}+\Tr(QAQ^*)\d t.
\end{align*}
Upon recalling \nameref{cond:Q}, it follows that
\begin{align*}
\frac{1}{2}\frac{\d}{\d t}\E\|u^0(t)\|^2_{H^1}+\E\|Au^0(t)\|^2_H=\E\la\f'(u^0(t))\grad u^0(t),\grad u^0(t)\ra_H+\Tr(QAQ^*).
\end{align*}
By making use of~\nameref{cond:phi:3} and the Cauchy-Schwarz inequality, we estimate 
\begin{align*}
\la\f'(u)\grad u,\grad u\ra_H\le a_\f\|A^{1/2}u\|^2_H=a_\f\la u,Au\ra_H\le \frac{1}{2} a_\f^2\|u\|^2_H+\frac{1}{2}\|Au\|^2_H.
\end{align*}
Hence
\begin{align*}
\frac{\d}{\d t}\E\|u^0(t)\|^2_{H^1}+\E\|Au^0(t)\|^2_H=a_\f^2\E\|u^0(t)\|^2_H+2\Tr(QAQ^*).
\end{align*}
We now invoke~\eqref{ineq:react-diff:u^0_t<e^(-ct)} to infer the existence of a positive constant $c$ such that
\begin{align*}
\frac{\d}{\d t}\E\|u^0(t)\|^2_{H^1}+\E\|Au^0(t)\|^2_H\le c\Big( e^{-2\alpha_1 t}\E\|\ru\|^2_H+1\Big).
\end{align*}
Integrating the above equation with respect to time $t$ yields~\eqref{ineq:react-diff:|Phi|^2_(H^1_epsilon)-bound:random-initial-cond}. We conclude the proof with an application of Gronwall's inequality together with \cref{lem:tech:1} to produce~\eqref{ineq:react-diff:|u^0|_H1<c}.
\end{proof}

Next, we assert higher moment bounds in $H^1$ for $\nu^0$. This property was employed in establishing the quantiative convergence result asserted in~\cref{thm:epsilon->0:invariant-measures->limit-measure}.

\begin{lemma} \label{lem:react-diff:|u|^n_H1-bound:nu^0}
Assume the hypotheses of~\cref{thm:react-diff:epsilon:geometric-ergodicity} and let $\ru\in L^2(\Omega;H^1)$. Then, for all even integer $n\ge 2$, there exists a positive constant $c=c(n)$ independent of $\ru$ and $t$ such that
\begin{equation}\label{ineq:react-diff:|Phi|^n_(H^1)|Phi|^2_(H^2)-bound:random-initial-cond}
\E\|u^0 (t)\|^n_{H^1}+\int_0^t\E\|u^0 (s)\|^{n-2}_{H^1}\|u^0 (s)\|^2_{H^2}\emph{ d} s\le c\big(\E\|\ru\|^n_{H^1}+\E \|\ru\|^{2n}_H+1)(t+1).
\end{equation}
As a consequence, for $p>0$, it holds that
\begin{align} \label{ineq:react-diff:|u|^n_H1-bound:nu^0}
\int_H \|u\|^{p}_{H^1}\nu^0(\emph{d} u)<\infty.
\end{align}
\end{lemma}

The argument is similar to the proof of~\cref{lem:moment-boud:H^1_epsilon} except tailored to system~\eqref{eqn:react-diff}.

\begin{proof}[Proof of \cref{lem:react-diff:|u|^n_H1-bound:nu^0}]

To establish~\eqref{ineq:react-diff:|Phi|^n_(H^1)|Phi|^2_(H^2)-bound:random-initial-cond}, we will proceed by induction on $n=2,4,\dots$ For the base case $n=2$,~\eqref{ineq:react-diff:|Phi|^n_(H^1)|Phi|^2_(H^2)-bound:random-initial-cond} was actually proven in~\cref{lem:react-diff:|u|^2_(H^1)-bound.and.int_0^t|Au(s)|^2ds-bound:random-initial-cond}.

Suppose \eqref{ineq:react-diff:|Phi|^n_(H^1)|Phi|^2_(H^2)-bound:random-initial-cond} holds for all $2,...,n-2$, we proceed to show it also holds for $n\ge 4$. We first apply It\^o's formula to $\|u^0(t)\|^n_{H^1}$ to obtain the identity
\begin{align*}
\d\|u^0(t)\|^{n}_{H^1}&=-2n\|u^0(t)\|^{n-2}_{H^1}\|u^0(t)\|^2_{H^2}\d t+2n\|u^0(t)\|^{n-2}_{H^1}\la \f'(u^0(t))\grad u^0(t),\grad u^0(t)\ra_{H}\d t\\
&\qquad+n\|u^0(t)\|^{n-2}_{H^1}\Tr(QAQ^*)\d t+2n(n-2)\|u^0(t)\|^{n-4}_{H^1}\|Qu^0(t)\|^2_{H^1}\d t\\
&\qquad +2n\|u^0(t)\|^{n-2}_{H^1}\la u^0(t),Q\d w(t)\ra_{H^1}.
\end{align*}
In view of~\nameref{cond:phi:3}, we estimate the term involving $\f$ using Young's and Holder's inequalities as follows:
\begin{align*}
\|u^0(t)\|^{n-2}_{H^1}\la \f'(u^0(t))\grad u^0(t),\grad u^0(t)\ra_{H}
&\le a_\f\|u^0(t)\|^{n-2}_{H^1}\la A^{1/2}u^0(t),A^{1/2}u^0(t)\ra_{H}\\
&=a_\f\|u^0(t)\|^{n-2}_{H^1}\la u^0(t),Au^0(t)\ra_{H}\\
&\le a_\f^2 \|u^0(t)\|^{n-2}_{H^1}\|u^0(t)\|^{2}_{H}+\frac{1}{4}\|u^0(t)\|^{n-2}_{H^1}\|u^0(t)\|^{2}_{H^2}\\
&\le \frac{\alpha_1}{4}\|u^0(t)\|^{n}_{H^1}+C\|u^0(t)\|^{2n}_{H}+\frac{1}{4}\|u^0(t)\|^{n-2}_{H^1}\|u^0(t)\|^{2}_{H^2}\\
&\le C\|u^0(t)\|^{2n}_{H}+\frac{1}{2}\|u^0(t)\|^{n-2}_{H^1}\|u^0(t)\|^{2}_{H^2}.
\end{align*}
Concerning the quadratic variation term, we see that
\begin{align*}
\|u^0(t)\|^{n-4}_{H^1}\|Qu^0(t)\|^2_{H^1}\le \Tr(QQ^*)\|u^0(t)\|^{n-2}_{H^1}.
\end{align*}
Altogether, we infer the existence of positive constants $c$ and $C$ such that
\begin{align*}
\E\|u^0(t)\|^{n}_{H^1}+c\int_0^t \E\|u^0(s)\|^{n-2}_{H^1}\|u^0(s)\|^{2}_{H^2}\d s\le \E\|\ru\|^{n}_{H^1}+C\int_0^t\E \|u^0(s)\|^{2n}_{H}\d s+Ct.
\end{align*}
Now we invoke~\eqref{ineq:react-diff:moment-bound:H:n} to control the time-integral on the right-hand side by
\begin{align*}
\int_0^t\E\|u^0(s)\|^{2n}_H\d s\le \E\|\ru\|^{2n}_H\int_0^t e^{-cs}\d s+Ct\le C\big(\E \|\ru\|^{2n}_H+t\big).
\end{align*}
We therefore arrive at the bound
\begin{align*}
\E\|u^0(t)\|^{n}_{H^1}+c\int_0^t \E\|u^0(s)\|^{n-2}_{H^1}\|u^0(s)\|^{2}_{H^2}\d s
&\le  C\big(\E\|\ru\|^n_{H^1}+\E \|\ru\|^{2n}_H+1\big)(t+1),
\end{align*}
which establishes~\eqref{ineq:react-diff:|Phi|^n_(H^1)|Phi|^2_(H^2)-bound:random-initial-cond}, as desired. 

Regarding~\eqref{ineq:react-diff:|u|^n_H1-bound:nu^0}, we first prove that $\nu^0(H^1)=1$. To see this, for $R>0$, we employ~\eqref{ineq:react-diff:energy-bound:random-initial} to estimate
\begin{align}\label{est:star2}
\int_H P_t\|u\|^2_H  \nu^0 (\d u)+\int_H\int_0^t P_s(\|A^{1/2}u\|^2_H\mi R)\d s\nu(\d u)\le  \int_H \|u\|^2_H  \nu^0 (\d u)+Ct.
\end{align} 
Since $\nu^0$ possesses exponential moment bounds in $H$ (see \eqref{ineq:react-diff:exponential-bound:nu^0}) by invariance we have
\begin{align*}
\int_H P_t\|u\|^2_H  \nu^0 (\d u)=\int_H \|u\|^2_H  \nu^0 (\d u).
\end{align*}
In light of \eqref{est:star2}, this implies
\begin{align*}
\int_H\int_0^t P_s(\|A^{1/2}u\|^2_H\mi R)\d s\nu(\d u)\le  Ct.
\end{align*}
By invariance of $P_t$ once again, we have
\begin{align*}
\int_H\int_0^t P_t(\|A^{1/2}u\|^2_H\mi R)\nu(\d u) = t \int_H(\|A^{1/2}u\|^2_H\mi R)\nu(\d u).
\end{align*}
It follows that
\begin{align*}
\int_H(\|A^{1/2}u\|^2_H\mi R)\nu(\d u)\le C.
\end{align*}
Since the above inequality holds for $C$ uniformly with respect to $R$, the Monotone Convergence Theorem implies
\begin{align*}
\int_H \|A^{1/2}u\|^2_H\nu(\d u)\le C.
\end{align*}
Hence, $\nu^0(H^1)=1$, as desired.

Now we prove \eqref{ineq:react-diff:|u|^n_H1-bound:nu^0}. For this, let us fix $R>0$. Then by Birkhoff's Ergodic Theorem we have that
\begin{align*}
\lim_{t\to\infty}\frac{1}{t}\int_0^t P_s(\|u_0\|^n_{H^1}\mi R)\d s= \int_{H_1}\big(\|u\|_{H^1}^n\mi R\big)\nu^0(\d u)\quad \text{holds}\ \nu^0\text{-a.s.}
\end{align*}
In view of~\eqref{ineq:react-diff:|Phi|^n_(H^1)|Phi|^2_(H^2)-bound:random-initial-cond} we have
\begin{align*}
\frac{1}{t}\int_0^t P_s(\|u_0\|^n_{H^1}\mi R)\d s \le \frac{C}{t}(\|u_0\|^n_{H^1}+\|u_0\|^{2n}_H+1)(t+1) \le C(\|u_0\|^n_{H^1}+\|u_0\|^{2n}_H+1).
\end{align*}
It follows that  
\begin{align*}
\int_{H_1}\big(\|u\|^n_{H^1}\mi R\big)\nu^0(\d u)\le C(\|u_0\|^n_{H^1}+\|\ru\|^{2n}_H+1),
\end{align*}
holds for all $R>0$. Hence, by applying the Monotone Convergence Theorem once again, we deduce~\eqref{ineq:react-diff:|u|^n_H1-bound:nu^0}, which completes the proof.
\end{proof}

In the following lemma concerning $u^0$, we assert that the difference $u^0(t)-u^0(s)$ is controlled by the length $|t-s|$. This property was crucially applied in the proof of~\cref{thm:epsilon->0:solution->limit-solution:finite-time}.

\begin{lemma} \label{lem:react-diff:phi.Lipschit:|u(t)-u(s)|-bound} Assume the hypotheses of~\cref{prop:well-posed}. Let $\ru\in L^2(\Omega;H^1)$ and $T>0$. Then

\begin{enumerate}[noitemsep,topsep=0pt,wide=0pt,label=\arabic*.,ref=\theassumption.\arabic*]

\item When $d\ge 4$, suppose that $\f(x)$ satisfies~\eqref{cond:phi:O(|x|)}. Then there exists a constant $c=c(\f)$ independent of $T$ such that
\begin{equation} \label{ineq:react-diff:phi.Lipschit:|u(t)-u(s)|-bound}
\E\|u^0 (t)-u^0 (s)\|^2_H\le c(t-s)(T+1)\big(\E\|\ru\|^2_{H^1}+1\big),
\end{equation}
for all $0\le s\le t\le T$.

\item When $d=1,2$, there exists a constant $c=c(\f)$ independent of $T$ such that 
\begin{equation} \label{ineq:react-diff:phi.Lipschit:|u(t)-u(s)|-bound:agmon:d=1-2}
\E\|u^0 (t)-u^0 (s)\|^2_H\le c(t-s)(T^2+1)\big(\E\|\ru\|^{2\lceil p_0 \rceil}_{H^1}+\E \|\ru\|^{4\lceil p_0 \rceil}_H+1\big),
\end{equation}
for all $0\le s\le t\le T$. In the above, $p_0$ is the constant as in \nameref{cond:phi:1}-\nameref{cond:phi:2}.

\item When $d\le 3$, suppose that $\f$ satisfies \eqref{cond:phi:agmon}. Then there exists a constant $c=c(\f)$ independent of $T$ such that 
\begin{equation} \label{ineq:react-diff:phi.Lipschit:|u(t)-u(s)|-bound:agmon}
\E\|u^0 (t)-u^0 (s)\|^2_H\le c(t-s)(T^2+1)\big(\E\|\ru\|^{10}_{H^1}+\E \|\ru\|^{20}_H+1\big),
\end{equation}
for all $0\le s\le t\le T$.
\end{enumerate}

\end{lemma}
 
\begin{proof}
First note that $\ru\in H^1$ ensures that $u^0(\cdot) \in L^2([0,T];\H^2)$ $\P-$a.s., by virtue of~\eqref{ineq:react-diff:|Phi|^2_(H^1_epsilon)-bound:random-initial-cond}. Furthermore, since $u^0$ is a mild solution of~\eqref{eqn:react-diff}, it is also a weak solution \cite{da2014stochastic}. Thus
\begin{align*}
\|u^0 (t)-u^0 (s)\|^2_H&=\int_s^t \la Au^0 (r),u^0 (t)-u^0 (s)\ra_H\d r+\int_s^t\la \f(u^0 (r)),u^0 (t)-u^0 (s)\ra_H\d r\\
&\qquad+\int_s^t \la u^0 (t)-u^0 (s),Q\d w(r)\ra_H,
\end{align*}
for $0\le s\le t\le T$. By Young's inequality, we have
\begin{align*}
\E\int_s^t \la Au^0 (r),u^0 (t)-u^0 (s)\ra_H\d r&\le \frac{1}{4}\E\|u^0 (t)-u^0 (s)\|^2_H+\E\Big(\int_s^t\|Au^0 (r)\|\d r\Big)^2\\
&\le \frac{1}{4}\E\|u^0 (t)-u^0 (s)\|^2_H+(t-s)\int_s^t\E\|Au^0 (r)\|_H^2\d r.
\end{align*}
Similarly, by It\^o's isometry, we have
\begin{align*}
\E\int_s^t \la u^0 (t)-u^0 (s),Q\d w(r)\ra_H&\le \frac{1}{4}\E\|u^0 (t)-u^0 (s)\|^2_H+\E\Big\|\int_s^t Q\d w(r)\Big\|^2_H\\
&= \frac{1}{4}\E\|u^0 (t)-u^0 (s)\|^2_H+(t-s)\text{tr}(QQ^*).
\end{align*}
Concerning $\f$, we also use Young's inequality to obtain
\begin{align*}
\E\int_s^t\la \f(u^0 (r)),u^0 (t)-u^0 (s)\ra_H\d r&\le \frac{1}{4}\E\|u^0 (t)-u^0 (s)\|^2_H+\E\Big(\int_s^t \|\f(u^0 (r))\|_H\d r\Big)^2\\
&\le \frac{1}{4}\E\|u^0 (t)-u^0 (s)\|^2_H+\int_s^t\E \|\f(u^0 (r))\|^2_H\d r.
\end{align*}
Altogether, upon recalling $\Tr(QQ^*)<\infty$ as in~\nameref{cond:Q}, we arrive at the estimate 
\begin{align}
\frac{1}{4}\E\|u^0 (t)-u^0 (s)\|^2_H&\le (t-s)\Big(\int_0^T \big(\E\|Au^0 \|^2_H+\E\|\f(u^0 (r))\|^2_H\big)\d r +\Tr(QQ^*)\Big)\nonumber\\
&\le c(t-s)\Big(\E\|\ru\|^2_{H^1}+T+\int_0^T \E\|\f(u^0 (r))\|^2_H\d r +\Tr(QQ^*)\Big)\label{ineq:react-diff:u_0(t)-u_0(s)},
\end{align}
where in the last implication above, we employed \eqref{ineq:react-diff:|Phi|^2_(H^1_epsilon)-bound:random-initial-cond} from~\cref{lem:react-diff:|u|^2_(H^1)-bound.and.int_0^t|Au(s)|^2ds-bound:random-initial-cond}. It therefore remains to find a suitable bound on the integral involving $\f$. To this end, we delineate between three cases: $d\geq 4$, $d=1,2$ and $d=3$.

\noindent\textbf{Case 1}. We assume $d\geq 4$ and $\f(x)=O(|x|)$ as $|x|\to\infty$. Then there exists a constant $c(\f)>0$ such that for all $x\in\rbb$, $\f(x)^2\le c(\f)(x^2+1)$. As a consequence, by virtue of~\eqref{ineq:react-diff:u^0_t<e^(-ct)}, we obtain
\begin{align*}
\int_0^T \E\|\f(u^0 (y))\|^2_H\d y\le c(\f)\int_0^T\close\big(\E\|u^0 (r)\|^2_H+1\big)\d r\le c(\f)\big(\E\|\ru\|^2_H+T\big).
\end{align*}
We then combine the above estimate with~\eqref{ineq:react-diff:u_0(t)-u_0(s)} to establish~\eqref{ineq:react-diff:phi.Lipschit:|u(t)-u(s)|-bound}, namely, there exists a positive constant $c=c(\f)$ such that
\begin{align*}
\E\|u^0 (t)-u^0 (s)\|^2_H\le c(t-s)(T+1)\big(\E\|\ru\|^2_{H^1}+1\big),
\end{align*}
which completes part 1.

\noindent\textbf{Case 2}. Considering dimension $d=1, 2$, on the one hand, when $d=1$, we employ \eqref{cond:phi:1} and  $H^1\subset L^\infty$ to deduce
\begin{align*}
\|\f(u)\|^2_H\le c(1+\|u\|^{2p_0}_{H^1})\le c(1+\|u\|^{2 \lceil p_0\rceil-2}_{H^1}\|u\|^2_{H^2}).
\end{align*}
On the other hand, when $d=2$, we recall Agmon's inequality that
\begin{align*}
\|u\|_{L^\infty}\le c\|u\|^\theta_{H^{s_1}}\|u\|^{1-\theta}_{H^{s_2}},
\end{align*}
holds for all $\theta\in[0,1]$, $s_1<1<s_2$ such that
\begin{align*}
1=\theta s_1 +(1-\theta)s_2.
\end{align*}
In particular, choosing
\begin{align*}
\theta=\frac{2p_0-1}{2p_0},\quad s_1=\frac{2p_0-2}{2p_0-1},\quad s_2=2,
\end{align*}
implies
\begin{align*}
\|\f(u)\|^2_H\le c(1+\|u\|^{2p_0}_{L^\infty}) \le c(1+\|u\|^{2p_0-1}_{H^{\frac{2p_0-2}{2p_0-1}}}\|u\|_{H^2}) \le c(1+\|u\|^{2\lceil p_0 \rceil-2}_{H^1}\|u\|^2_{H^2}).
\end{align*}
An application of~\cref{lem:react-diff:|u|^n_H1-bound:nu^0} (with $n=2\lceil p_0\rceil$) gives
\begin{align*}
\int_0^T\E\|\f(u^0 (r))\|^2_H\d r\le c\Big(T+\big(\E\|\ru\|^{2\lceil p_0 \rceil}_{H^1}+\E \|\ru\|^{4\lceil p_0 \rceil}_H+1)\int_0^T(r+1)\d r\Big),
\end{align*} 
which together with~\eqref{ineq:react-diff:u_0(t)-u_0(s)} produces the estimate~\eqref{ineq:react-diff:phi.Lipschit:|u(t)-u(s)|-bound:agmon:d=1-2}.

\noindent\textbf{Case 3}. In dimension $d\le3$, $\f$ satisfies~\eqref{cond:phi:agmon}, i.e.,  $\f(x)=O(|x|^5)$ as $|x|\to\infty$. Similar to Case 2, we employ Agmon's inequality to see that
\begin{align*}
\|\f(u)\|^2_H\le c(1+\|u\|^6_{L^6}\|u\|^4_{L^\infty})\le c(1+\|A^{1/2}u\|^8_H\|Au\|^2_H).
\end{align*}
In light of~\cref{lem:react-diff:|u|^n_H1-bound:nu^0} (with $n=10$), we obtain
\begin{align*}
\int_0^T\E\|\f(u^0 (r))\|^2_H\d r\le c\Big(T+\big(\E\|\ru\|^{10}_{H^1}+\E \|\ru\|^{20}_H+1)\int_0^T(r+1)\d r\Big),\end{align*}
which together with~\eqref{ineq:react-diff:u_0(t)-u_0(s)} produces the estimate~\eqref{ineq:react-diff:phi.Lipschit:|u(t)-u(s)|-bound:agmon}. The proof is thus complete.
\end{proof}



\section{Estimates on the auxiliary system~\eqref{eqn:react-diff:eta^0_epsilon}} \label{sec:nu^(0-epsilon)}

In this section, we collect useful estimates on the auxiliary system~\eqref{eqn:react-diff:eta^0_epsilon}. We will first state and prove a lemma concerning the norm of $\etazero(t)$ in $M^0_{\mu_\varepsilon}$ and $M^1_{\mu_\varepsilon}$ subject to random initial conditions; an estimate in $\Mzeroeps$ was employed to prove~\cref{thm:epsilon->0:solution->limit-solution:finite-time}, whereas a bound in $\Moneeps$ was used to prove the existence of an invariant probability measure $\nuzero$ for~\eqref{eqn:react-diff:eta^0_epsilon}. For this purpose, let us recall the notation $\rx=(\pi_1\rx,\pi_2\rx)$, where $\pi_1,\pi_2$ denote the coordinate-projections over $\Honeeps$ onto the first and second components of elements, respectively, where $\mu_\varepsilon$ is given as in~\eqref{form:mu_epsilon}.

\begin{lemma} \label{lem:uniform-energy-bounds}
Assume the hypotheses of \cref{prop:well-posed} and let $\rx\in L^2(\Omega;\Honeeps)$. Let $(u^0(t),\etazero(t))$ denote the solution of~\eqref{eqn:react-diff:eta^0_epsilon} corresponding to initial condition $\rx$. Then there exists a positive constant $c$, independent of $\rx,\,t$, and $\varepsilon$, such that
\begin{equation} \label{ineq:int_0^t.E|eta|_(M^1_epsilon)ds<epsilon(|x|+t)}
\int_0^t \E\|\etazero(r)\|^2_{\Mieps}\emph{d} r\le \varepsilon\,c(\E\|\rx\|^2_{\Hieps}+t),\qquad i=0,1,
\end{equation}
and
\begin{equation} \label{ineq:react-diff:eta^0-bound:random-initial-cond} 
\E\|\etazero(t)\|^2_{\Mzeroeps}\le e^{-\frac{\delta}{2\varepsilon}t}\E\|\pi_2\rx\|^2_{\Mzeroeps}+c(\E\|\pi_1\rx\|^2_{H^1}+1)\varepsilon.
\end{equation}
\end{lemma}

\begin{proof}
For each $i=0,1$, it follows from~\eqref{eqn:react-diff:eta^0_epsilon} that
\begin{align*}
\frac{1}{2}\frac{\d}{\d t}\|\etazero(t)\|^2_{\Mieps}&=\la \Tcal_{\mueps}\etazero(t),\etazero(t)\ra_{\Mieps}+\la \etazero(t),u^0(t)\ra_{\Mieps}.
\end{align*}
We then invoke~\eqref{ineq:<T_epsilon.eta,eta>_(M^n)} to see that
\begin{align*}
\la \Tcal_{\mueps}\etazero(t),\etazero(t)\ra_{\Mieps}\le -\frac{\delta}{2\varepsilon}\|\etazero(t)\|^2_{\Mieps}.
\end{align*}
An applications of Young's inequality yields
\begin{align*} 
\la \etazero(t),u^0(t)\ra_{\Mieps}&\le \frac{\delta}{4\varepsilon}\|\etazero(t)\|^2_{\Mieps}+\frac{\varepsilon}{\delta}\int_0^\infty\mu_\varepsilon(s)\d s\|u^0 (t)\|^2_{H^{i+1}}\\
&=  \frac{\delta}{4\varepsilon}\|\etazero (t)\|^2_{\Mieps}+\frac{\|\mu\|_{L^1(\rbb^+)}}{\delta}\|u^0 (t)\|^2_{H^{i+1}}.
\end{align*}
Hence
\begin{align} \label{ineq:d/dt.eta^0<-eta^0+u}
\frac{1}{2}\frac{\d}{\d t}\|\etazero (t)\|^2_{\Mieps}
&\le  -\frac{\delta}{4\varepsilon}\|\etazero (t)\|^2_{\Mieps}+\frac{\|\mu\|_{L^1(\rbb^+)}}{\delta}\|u^0 (t)\|^2_{H^{i+1}}.
\end{align}
Integrating the above inequality over the interval $[0,t$] yields
\begin{align*}
\|\etazero (t)\|^2_{\Mieps} +\frac{\delta}{2\varepsilon}\int_0^t\|\etazero (r)\|^2_{\Mieps}\d r\le \|\pi_2\rx\|^2_{\Mieps} +2\frac{\|\mu\|_{L^1(\rbb^+)}}{\delta}\int_0^t\|u^0 (r)\|^2_{H^{i+1}}\d r.
\end{align*}
In particular, we have
\begin{align*}
\int_0^t\E\|\etazero (r)\|^2_{\Mieps}\d r\le \frac{2\varepsilon}{\delta}\Big(\E\|\pi_2\rx\|^2_{\Mieps} +2\frac{\|\mu\|_{L^1(\rbb_+)}}{\delta}\int_0^t\E\|u^0 (r)\|^2_{H^{i+1}}\d r\Big).
\end{align*}
We now invoke either~\eqref{ineq:react-diff:energy-bound:random-initial}  or~\eqref{ineq:react-diff:|Phi|^2_(H^1_epsilon)-bound:random-initial-cond}, depending on either $i=0$ or $i=1$, respectively, together with the above estimate to produce the desired bound~\eqref{ineq:int_0^t.E|eta|_(M^1_epsilon)ds<epsilon(|x|+t)}.

Concerning \eqref{ineq:react-diff:eta^0-bound:random-initial-cond}, since $\rx\in\Honeeps$, we combine~\eqref{ineq:d/dt.eta^0<-eta^0+u} for $i=0$ with~\eqref{ineq:react-diff:|u^0|_H1<c}  to see that
\begin{align*}
\frac{1}{2}\frac{\d}{\d t}\E\|\etazero (t)\|^2_{\Mzeroeps}\le -\frac{\delta}{4\varepsilon}\E\|\etazero (t)\|^2_{\Mzeroeps}+c(\E\|\pi_1\rx\|^2_{H^1}+1).
\end{align*}
This implies
\begin{align*}
\E\|\etazero (t)\|^2_{\Mzeroeps}\le e^{-\frac{\delta}{2\varepsilon}t}\E\|\pi_2\rx\|^2_{\Mzeroeps}+c(\E\|\pi_1\rx\|^2_{H^1}+1)\int_0^te^{-\frac{\delta}{2\varepsilon}(t-r)}\d r,
\end{align*}
which is~\eqref{ineq:react-diff:eta^0-bound:random-initial-cond}, and we are done.
\end{proof}

Through~\cref{thm:existence-inv-measure:nu^(0-epsilon)} below, we establish that there exists at least one invariant probability measure for~\eqref{eqn:react-diff:eta^0_epsilon}. 

\begin{theorem} \label{thm:existence-inv-measure:nu^(0-epsilon)}
Assume the hypotheses of~\cref{prop:well-posed}. Then there exists at least one invariant probability measure, $\nuzero$, for the Markov semigroup correpsonding to~\eqref{eqn:react-diff:eta^0_epsilon}.
\end{theorem}

The argument of \cref{thm:existence-inv-measure:nu^(0-epsilon)} is similar to the proof of \cref{thm:exponential-bound:H^1_epsilon:nu^mu}. See also the discussion on the existence of $\nueps$ in \cite{glatt2024paperII}. For the reader's convenience, we briefly summarize the argument from \cite{glatt2024paperII} without providing additional details: following the classical Krylov-Bogoliubov argument (see, for instance, \cite{da2014stochastic}), it suffices to show that the family of measures $\{\nuzero_T\}_{T>0}$ given by
\begin{equation} \label{form:nu.epsilon.time-average}
\nuzero_T(\cdot)=\frac{1}{T}\int_0^T\close \Pzero_t(0,\cdot)\d t,
\end{equation}
is tight in $\Pcal r(\Hzeroeps)$. Here $\Pzero_t(0,\cdot)$ is the Markov semigroup associated with the solution $\big(u^0_0(t),\etazero_0(t)\big)$ of~\eqref{eqn:react-diff:eta^0_epsilon} with the zero initial condition in $\Hzeroeps$. We recall that the embedding $\Z^1_{\mueps}=H^1\times \Ecal^1_{\mueps}\subset\Hzeroeps$ is compact (see, for instance, \cite{gatti2004exponential,joseph1989heat,joseph1990heat}). Hence, the tightness of $\{\nuzero_T\}$ is a consequence of the moment bounds~\eqref{form:space:L:norm} in the $\Ecal^1_{\mueps}$-norm.
In particular, we may then adapt the proof of~\cref{lem:T:bound} 
to infer 
\begin{align*}
\sup_{t\geq0}\E\| \Tcal_{\mueps}\etazero_0(t)\|^2_{\Mzeroeps}\le c\qquad\text{and}\qquad \sup_{t\geq0}\left(\E\sup_{r\ge 1}r\T^{\mueps}_{\etazero_0(t)}(r)\right)\le c,
\end{align*}
for some positive constant $c$. Combined with the estimate~\eqref{ineq:int_0^t.E|eta|_(M^1_epsilon)ds<epsilon(|x|+t)} in $\Moneeps$, we employ the same argument as in the proof of \cref{thm:exponential-bound:H^1_epsilon:nu^mu} 
to deduce that $\{\nuzero_T\}$ is tight. Consequently, $\nuzero_T$ converges weakly to a probability measure $\nuzero$, which is invariant for~\eqref{eqn:react-diff:eta^0_epsilon}.

\begin{remark}
Under the same hypothesis of~\cref{thm:react-diff:epsilon:geometric-ergodicity}, it can be shown that $\nuzero$ is indeed the only invariant probability for $\Pzero_t$. However, since we do not need the uniqueness of $\nuzero$ to establish the convergence results as $\varepsilon\to 0$, we omit the formal statement and proof of this fact.
\end{remark}

In what follows, we collect useful properties of $\nuzero$. The first of which asserts that $\pi_1^{-1}\nuzero\sim\nu^0$.

\begin{lemma} \label{lem:pi_1nu^(0-epsilon)=nu^0} Assume the hypotheses of~\cref{thm:react-diff:epsilon:geometric-ergodicity}. Let $\nuzero$ be the unique invariant probability measure associated to~\eqref{eqn:react-diff:eta^0_epsilon} guaranteed by~\cref{thm:existence-inv-measure:nu^(0-epsilon)} and let $\nu^0$ be the unique invariant probability measure for~\eqref{eqn:react-diff}. Then, $\pi_1^{-1}\nuzero\sim\nu^0$.
\end{lemma}
\begin{proof}
By the uniqueness of $\nu^0$, it suffices to prove that $\pi_1^{-1}\nuzero$ is invariant for~\eqref{eqn:react-diff}. To see this, by the invariant property of $\nuzero$ for~\eqref{eqn:react-diff:eta^0_epsilon} and the fact that $u^0(t)$ does not depend on $\etazero(t)$, we have
\begin{align*}
\int_{H}f(u)\pi_1^{-1}\nuzero(\d u)&=\int_{\Hzeroeps}f(u)\nuzero(\d u,\d\eta)\\
&=\int_{\Hzeroeps}\E_{(u,\eta)} f(u^0(t))\nuzero(\d u,\d\eta)\\
&=\int_{H\times\Mzeroeps}\close\close\close\E_{u} f(u^0(t))\nuzero(\d u,\d\eta)=\int_H \E_u f(u^0(t))\pi_1^{-1}\nuzero(\d u),
\end{align*}
for all $t\ge 0$. As a consequence, $\pi_1^{-1}\nuzero$ is invariant probability for the system~\eqref{eqn:react-diff}.
\end{proof}
Next, we establish a moment bound in $\Honeeps$ with respect to $\nuzero$.

\begin{lemma} \label{lem:epsilon->0:invariant-measures:H^1_epsilon-bound}
Let $\nuzero$ be the invariant probability for~\eqref{eqn:react-diff:eta^0_epsilon} as in~\cref{thm:existence-inv-measure:nu^(0-epsilon)}. Then
\begin{equation} \label{lim:epsilon->0:invariant-measures:H^1_epsilon-bound}
\limsup_{\varepsilon\to 0}\left(\int_{\Hzeroeps}\close \|U\|^2_{\Honeeps}+\|\pi_1U\|^p_{H^1}\right)\nuzero(\emph{d} U)<\infty,
\end{equation}
for any $p\ge 1$. Consequently, there exists $\varepsilon_0>0$ such that $\nuzero(\Honeeps)=1$, for all $\varepsilon\in(0,\varepsilon_0)$.
\end{lemma}

\begin{proof} The fact that $\nuzero$ is full in $\Honeeps$ follows from limit~\eqref{lim:epsilon->0:invariant-measures:H^1_epsilon-bound}. It thus suffices to prove~\eqref{lim:epsilon->0:invariant-measures:H^1_epsilon-bound}.

By virtue of \cref{lem:pi_1nu^(0-epsilon)=nu^0}, $\pi_1^{-1}\nuzero\sim \nu^0$ where $\nu^0$ is the unique invariant probability measure for~\eqref{eqn:react-diff} guaranteed by \cref{thm:react-diff:geometric-ergodicity}. Thanks 
to~\eqref{ineq:react-diff:|u|^n_H1-bound:nu^0}, we readily have 
\begin{align*}
\limsup_{\varepsilon\to 0}\int_{\Hzeroeps}\close\|\pi_1U\|^p_{H^1}\nuzero(\d U) =\limsup_{\varepsilon\to 0}\int_{H}\|u\|^p_{H^1}\nu^0(\d u) <\infty,
\end{align*}
for any $p\ge 1$. Upon recalling the time average measures, $\nuzero_t$, as in~\eqref{form:nu.epsilon.time-average}, we apply~\eqref{ineq:react-diff:|u^0|_H1<c} and~\eqref{ineq:int_0^t.E|eta|_(M^1_epsilon)ds<epsilon(|x|+t)} with  $\ru=0$ and $\rx=0$, respectively, to infer that
\begin{align*}
\int_{\Hzeroeps}\close\|U\|^2_{\Honeeps}\nu^\varepsilon_t(\d U)&=\frac{1}{t}\int_0^t\left(\E\|u^0_0(r)\|^2_{H^1}+\E\|\etazero_0(r)\|^2_{\Moneeps}\right)\d r<\frac{1}{t} c\,t=c,
\end{align*}
holds for some positive constant $c$ that is independent of $\varepsilon$ and $t$. 

On the other hand, given any $R>0$, since $\nuzero_t$ converges weakly to $\nuzero$, one has
\begin{align*}
\lim_{t\to\infty}\int_{\Hzeroeps}\Big(\|U\|^2_{\Honeeps}\mi R\Big)\nuzero_t(\d U)\int_{\Hzeroeps}\big(\|U\|^2_{\Honeeps}\mi R\Big)\nuzero(\d U).
\end{align*}
It then follows that
\begin{align*}
\int_{\Hzeroeps}\left(\|U\|^2_{\Honeeps}\mi R\right)\nuzero(\d U)\le  c,
\end{align*}
for all sufficiently small $\varepsilon$. Upon passing to the limit $R\to\infty$, the Monotone Convergence Theorem then implies
\begin{align*}
\int_{\Hzeroeps}\|U\|^2_{\Honeeps}\nuzero(\d U)\le c,
\end{align*}
for all sufficiently small $\varepsilon$, which yields~\eqref{lim:epsilon->0:invariant-measures:H^1_epsilon-bound}, as desired.
\end{proof}

Lastly, we turn to exponential estimates for $\nuzero$. To state our result, for $\kappa_1,\,\kappa_2>0$, we introduce the following $\Psitilde_0:\Hzeroeps\mapsto\rbb$ given by
\begin{equation} \label{form:Psitilde_0}
\Psitilde_0(u,\eta)=\frac{1}{2}\kappa_1\|u\|^2_H+\frac{1}{2}\kappa_2\|\eta\|^2_{\Mzeroeps}.
\end{equation}
We note that $\Psi_0$ given in~\eqref{form:Psi_0} is a special case of \eqref{form:Psitilde_0}. In~\cref{lem:epsilon->0:invariant-measure:enegery-exponential-bound} below we establish exponential estimates for $\Psitilde_0$ evaluated along the solution pair $(u^0,\etazero)$ of~
\eqref{eqn:react-diff:eta^0_epsilon}. \cref{lem:epsilon->0:invariant-measure:enegery-exponential-bound} was applied in the proof of~\cref{lem:epsilon->0:alpha_0(solutions-limit-solution)->0}. 

\begin{lemma} \label{lem:epsilon->0:invariant-measure:enegery-exponential-bound}
 Under the same hypothesis of \cref{prop:well-posed}, let $\rx$ be a random variable in $L^2(\Omega,\Hzeroeps)$. Then, for all $\kappa_1$ sufficiently small, there exist positive constants $\kappa_2$, $c$ and $C$ independent of $\rx$ and $t$ such that
\begin{equation} \label{ineq:u^0,eta^0:exponential-estimate}
\E \exp\left(\Psitilde_0(u^0 (t),\etazero (t))\right)\le e^{-ct}\E \exp\left({\Psitilde_0(\rx)}\right)+C.
\end{equation}
Moreover, there exists $\beta_0>0$ such that
\begin{equation} \label{lim:epsilon->0:invariant-measures:energy-exponential-bound}
\sup_{\varepsilon\in(0,1)}\int_{\Hzeroeps}\exp\big\{\beta\|U\|^2_{\Hzeroeps}\big\}\nuzero(\emph{d}U)<\infty,
\end{equation}
for all $\beta\in(0,\beta_0)$, where $\nuzero$ is the invariant probability measure associated to~\eqref{eqn:react-diff:eta^0_epsilon} guaranteed by in~\cref{thm:existence-inv-measure:nu^(0-epsilon)}.
\end{lemma}

\begin{proof} Let $\L^{0,\varepsilon}$ denote the generator associated with~\eqref{eqn:react-diff:eta^0_epsilon}. By It\^o's formula applied to $\exp\left({\Psitilde_0(u,\eta)}\right)$ we obtain
\begin{align*}
\L^{0,\varepsilon} \exp\left({\Psitilde_0(u,\eta)}\right)&= \exp\left({\Psitilde_0(u,\eta)}\right)\Big(-\kappa_1\|A^{1/2}u\|^2_H+\kappa_1\la \f(u),u\ra_H+\frac{1}{2}\kappa_1\Tr(QQ^*)+\frac{1}{2}\kappa_1^2\|Qu\|_{H}^2\\
&\qquad\qquad\qquad\qquad\quad+\kappa_2\la \Tcal_{\mueps}\eta,\eta\ra_{\Mzeroeps}+\kappa_2\la \eta,u\ra_{\Mzeroeps}\Big).
\end{align*}
By~\nameref{cond:phi:2}, we readily have
\begin{align*}
\la \f(u),u\ra_H\le a_3|\domain|.
\end{align*}
Suppose that $\kappa_1>0$ satisfies
\begin{align*}
\kappa_1<\frac{\alpha_1}{2\Tr(QQ^*)}.
\end{align*}
Then by Poincare's inequality
\begin{align*}
\frac{1}{2}\kappa_1^2\|Qu\|_{H}^2\le \frac{1}{2}\kappa_1^2\Tr(QQ^*)\|u\|_{H}^2\le \frac{1}{4}\|A^{1/2}u\|^2_H.
\end{align*}
Next, we invoke Young's inequality and~\eqref{ineq:<T_epsilon.eta,eta>} to estimate
\begin{align*}
\la \Tcal_{\mueps}\eta,\eta\ra_{\Mzeroeps}+\la \eta,u\ra_{\Mzeroeps}&\le-\frac{\delta}{2\varepsilon}\|\eta\|^2_{\Mzeroeps}+\int_0^\infty\mu_\varepsilon(s)\la A^{1/2}\eta(s),A^{1/2}u\ra_H\d s\\
&\le -\frac{\delta}{4\varepsilon}\|\eta\|^2_{\Mzeroeps}+\frac{\varepsilon}{\delta}\int_0^\infty\close\mu_\varepsilon(s)\d s\,\|A^{1/2}u\|^2_H\\
&= -\frac{\delta}{4\varepsilon}\|\eta\|^2_{\Mzeroeps}+\frac{\|\mu\|_{L^1(\rbb^+)}}{\delta}\|A^{1/2}u\|^2_H.
\end{align*}
So that, provided $\kappa_2$ satisfies
\begin{align*}
\kappa_2<\frac{\kappa_1\delta}{4\|\mu\|_{L^1(\rbb^+)}},
\end{align*}
we have
\begin{align*}
\kappa_2\la \Tcal_{\mueps}\eta,\eta\ra_{\Mzeroeps}+\kappa_2\la \eta,u\ra_{\Mzeroeps}\le -\kappa_2\frac{\delta}{4\varepsilon}\|\eta\|^2_{\Mzeroeps}+\frac{1}{4}\kappa_1\|A^{1/2}u\|^2_H.
\end{align*}
We now collect everything to infer the existence of positive constants $C$ and $c$ independent of $\varepsilon$, $t$ and $\rx$ such that
\begin{align*}
\frac{\d}{\d t}\E \exp\left({\Psitilde_0(u^0(t),\etazero(t))}\right)\le -c\E \exp\left({\Psitilde_0(u^0(t),\etazero(t))}\right)\left(\Psitilde_0(u^0(t),\etazero(t))-C\right).
\end{align*}
Then~\eqref{ineq:u^0,eta^0:exponential-estimate} can be established using the same argument as in the proof of~\eqref{ineq:exponential-bound:H^0_epsilon} of \cref{lem:moment-bound:H^0_epsilon}.

Turning to~\eqref{lim:epsilon->0:invariant-measures:energy-exponential-bound}, in view of~\eqref{ineq:u^0,eta^0:exponential-estimate} corresponding to zero initial condition, observe that
\begin{align*}
\int_{\Hzeroeps}\exp\left({\Psitilde_0(\x)}\right)\nuzero_t(\d\x)=\frac{1}{t}\int_0^t \E \exp\left({\Psitilde_0(u^0_0(r),\etazero_0(r))}\right)\d r\le \frac{C}{t}\int_0^t e^{-ct}\d r +C\le C.
\end{align*}
Following the same argument as in the proof of ~\cref{lem:epsilon->0:invariant-measures:H^1_epsilon-bound} for the function $\exp\left({\Psitilde_0(u,\eta)}\right)$, we obtain
\begin{align} \label{ineq:exponential-moment:nu^(0-epsilon):Psitilde_0}
\sup_{\varepsilon\in (0,1)} \int_{\Hzeroeps}\exp\left({\Psitilde_0(U)}\right)\nuzero(\d U)<\infty.
\end{align}
Recalling $\Psitilde_0$ as in~\eqref{form:Psitilde_0}, this produces~\eqref{lim:epsilon->0:invariant-measures:energy-exponential-bound}, as claimed.
\end{proof}


\subsection*{Acknowledgments}
The authors would like to thank anonymous referees for their providing a thorough review of this work. We appreciate their careful reading and insightful comments, which have improved the manuscript. The work of N.E. Glatt-Holtz was partially supported under the
National Science Foundation grants DMS-1313272, DMS-1816551, DMS-2108790,
and under a Simons Foundation travel support award 515990. The work of V.R. Martinez was partially supported by NSF-DMS 2213363 and NSF-DMS 2206491, PSC-CUNY Award 65187-00 53, which is jointly funded by The Professional Staff Congress and The City University of New York, and the Dolciani Halloran Foundation.

\appendix

\section{Estimates on Wasserstein distances} \label{sec:Wasserstein}

In this section, we collect some useful estimates on Wasserstein distances $\W_{d^{\mueps}_{N,\beta}}$ and $\W_{d_{N,\beta}}$ as in~\eqref{form:W_d:mu} and~\eqref{form:W_0(nu_1,nu_2)}, respectively. We have employed the results below to study the singular limit as $\varepsilon\to 0$ and prove \cref{thm:epsilon->0:invariant-measures->limit-measure} in \cref{sec:epsilon->0:inv-measure}.

In \cref{lem:W_0(nu_1-nu_2)<W_epsilon(nu_1-nu_2)} below, we assert that $\W_{d^{\mueps}_{N,\beta}}$ dominates $\W_{d_{N,\beta}}$.

\begin{lemma} \label{lem:W_0(nu_1-nu_2)<W_epsilon(nu_1-nu_2)} Given $\varepsilon\in(0,1]$, it holds that
\begin{align} \label{ineq:d_0(x1,x2)<d_epsilon(x1,x2)}
d_{N,\beta}(\pi_1U_1,\pi_1U_2)\le d^{\mueps}_{N,\beta}(U_1,U_2),
\end{align}
for all $U_1,U_2\in\Hzeroeps$, where $d^{\mueps}_{N,\beta}$ and $d_{N,\beta}$ are as in~\eqref{form:d_N,beta:mu} and~\eqref{form:d_N,beta}, respectively. As a consequence, let $\nu_1,\nu_2$ be two probability measures in $\Pcal r(\Hzeroeps)$, then
\begin{equation}\label{ineq:W_0(nu_1,nu_2)<W_epsilon(nu_1,nu_2)}
\W_{d_{N,\beta}}(\pi_1^{-1}\nu_1,\pi_1^{-1}\nu_2)\le \W_{d^{\mueps}_{N,\beta}}(\nu_1,\nu_2).
\end{equation}
\end{lemma}

\begin{proof} Given two elements $U_1,U_2\in\Hzeroeps$, from~\eqref{form:d_N,beta} and~\eqref{form:Psi_0}, we see that
\begin{align*}
d_{N,\beta}(\pi_1U_1,\pi_1U_2)&=\sqrt{(1\mi N\|\pi_1U_1-\pi_1U_2\|_H)\big[1+e^{\frac{\beta}{2}\|\pi_1U_1\|^2_H } +e^{\frac{\beta}{2}\|\pi_1U_2\|^2_H } \big]}\\
&\le \sqrt{(1\mi N\|U_1-U_2\|_{\Hzeroeps})\big[1+e^{\beta\Psi_0(U_1) } +e^{\beta\Psi_0(U_2) } \big]}=d^{\mueps}_{N,\beta}(\pi_1U_1,\pi_1U_2).
\end{align*}
This proves~\eqref{ineq:d_0(x1,x2)<d_epsilon(x1,x2)}. By the expressions~\eqref{form:W_d:mu} and~\eqref{form:W_0(nu_1,nu_2)}, we immediately obtain~\eqref{ineq:W_0(nu_1,nu_2)<W_epsilon(nu_1,nu_2)}.
\end{proof}

Next, we establish a generalized triangle inequality for $\W_{d^{\mueps}_{N,\beta}}$.

\begin{lemma} \label{lem:W_epsilon:triangle-ineq}
For all $N,\beta,\varepsilon>0$, there exists a positive constant $C$, 
 such that 
\begin{equation} \label{ineq:W_epsilon:triangle-ineq}
\W_{d^{\mueps}_{N,\beta/2}}(\nu_1,\nu_3)\le C\Big(\W_{d^{\mueps}_{N,\beta}}(\nu_1,\nu_2)+\W_{d^{\mueps}_{N,\beta}}(\nu_2,\nu_3)\Big),
\end{equation}
for all $\nu_1,\nu_2,\nu_3\in \Pcal r(\Hzeroeps)$.
\end{lemma}

\begin{proof}
By~\eqref{form:W_0(nu_1,nu_2)}, it suffices to prove that there exists a $C>0$, independent of $U_1,U_2,U_3$ and $\varepsilon$, such that
\begin{equation} \label{ineq:d_epsilon:triangle-ineq}
d^{\mueps}_{N,\beta/2}(U_1,U_3)\le C\Big(d^{\mueps}_{N,\beta}(U_1,U_2)+d^{\mueps}_{N,\beta}(U_2,U_3)\Big),\quad U_1,U_2,U_3\in\Hzeroeps.
\end{equation}
To see this reduction, we first recall $\Psi_0$ from~\eqref{form:Psi_0} and use the Cauchy-Schwarz inequality to see that
\begin{align}
\Psi_0(u,\eta)&=\frac{1}{2}\|u\|^2_H+\frac{1}{2}(1-\kappa)\|\eta\|^2_{\Mzeroeps}\notag\\
&\ge \frac{1}2\|u-\utilde\|_H^2+\la u-\utilde,\utilde\ra_H+\frac{1}{2}\|\utilde\|^2_H+(1-\kappa)\left(\frac{1}2\|\eta\|_{\Mzeroeps}^2+\la \eta-\etatilde,\etatilde\ra_{\Mzeroeps}+\|\etatilde\ra_{M^0_{\mueps}}^2\right)\notag\\
&\ge -\|u-\utilde\|^2_H-(1-\kappa)\|\eta-\etatilde\|^2_{\Mzeroeps}+\frac{1}{2}\Psi_0(\utilde,\etatilde)\notag\\
&\ge -\|(u,\eta)-(\utilde,\etatilde)\|^2_{\Hzeroeps}+\frac{1}{2}\Psi_0(\utilde,\etatilde).\label{ineq:Psi_0(x)>-|x-xtilde|+Psi_0(xtilde)}
\end{align}
We now turn to the proof of~\eqref{ineq:d_epsilon:triangle-ineq}, recall
\begin{align*}
d^{\mueps}_{N,\beta/2}(U_1,U_3)&=\sqrt{d^{\mueps}_{N}(U_1,U_3)\left[1+\exp\left({\frac{\beta}{2}\Psi_0(U_1)}\right)+\exp\left({\frac{\beta}{2}\Psi_0(U_3)}\right)\right]}\\
&= \sqrt{\left(1\mi \left(N\|U_1-U_3\|_{\Hzeroeps}\right)\right)\left[1+\exp\left({\frac{\beta}{2}\Psi_0(U_1)}\right)+\exp\left({\frac{\beta}{2}\Psi_0(U_3)}\right)\right]}.
\end{align*}
Without loss of generality, assume that $e^{\frac{\beta}{2}\Psi_0(U_1)}\le e^{\frac{\beta}{2}\Psi_0(U_3)}$. Since $d^{\mueps}_{N}$ is actually a metric in $\Hzeroeps$, it holds that
\begin{align*}
\big{|}d^{\mueps}_{N,\beta}(U_1,U_3)\big{|}^2&= d^{\mueps}_N(U_1,U_3)\left[1+\exp\left({\frac{\beta}{2}\Psi_0(U_1)}\right)+\exp\left({\frac{\beta}{2}\Psi_0(U_3)}\right)\right]\\
&\le \left(d^{\mueps}_N(U_1,U_2)+d^{\mueps}_N(U_2,U_3)\right)\left[1+2\exp\left({\frac{\beta}{2}\Psi_0(U_3)}\right)\right].
\end{align*}
Now, there are two cases depending on $N\|U_2-U_3\|_{\Hzeroeps}$.

\noindent\textbf{Case 1:} $N\|U_2-U_3\|_{\Hzeroeps}<1$. In this case, we invoke~\eqref{ineq:Psi_0(x)>-|x-xtilde|+Psi_0(xtilde)} to see that
\begin{align*}
&\big{|}d^{\mueps}_{N,\beta}(U_1,U_2)\big{|}^2+\big{|}d^{\mueps}_{N,\beta}(U_2,U_3)\big{|}^2\\
&=d^{\mueps}_N(U_1,U_3)\left[1+\exp\left({\beta\Psi_0(U_1)}\right)+\exp\left({\beta\Psi_0(U_2)}\right)\right]+d^{\mueps}_N(U_2,U_3)\left[1+\exp\left({\beta\Psi_0(U_2)}\right)+\exp\left({\beta\Psi_0(U_3)}\right)\right]\\
&\ge  d^{\mueps}_N(U_1,U_3)\left[1+\exp\left(-\frac{\beta}{N^2}\right)\exp\left({\frac{\beta}{2}\Psi_0(U_3)}\right)\right]+d^{\mueps}_N(U_2,U_3)\left[1+\exp\left({\frac{\beta}{2}\Psi_0(U_3)}\right)\right]\\
&\ge \exp\left(-\frac{\beta}{N^2}\right)\big(d^{\mueps}_N(U_1,U_2)+d^{\mueps}_N(U_2,U_3)\big)\left[1+2\exp\left({\frac{\beta}{2}\Psi_0(U_3)}\right)\right].
\end{align*}
It follows that

\begin{align*}
\big{|}d^{\mueps}_{N,\beta/2}(U_1,U_3)\big{|}^2\le \exp\left(\frac{\beta}{N^2}\right)\Big(\big{|}d^{\mueps}_{N,\beta}(U_1,U_2)\big{|}^2+\big{|}d^{\mueps}_{N,\beta}(U_2,U_3)\big{|}^2\Big).
\end{align*}

\noindent\textbf{Case 2:} $N\|U_2-U_3\|_{\Hzeroeps}\ge 1$. In this case we have
$d^{\mueps}_N(U_2,U_3)  =1$.
Hence
\begin{align*}
&\big{|}d^{\mueps}_{N,\beta}(U_1,U_2)\big{|}^2+\big{|}d^{\mueps}_{N,\beta}(U_2,U_3)\big{|}^2\\
&=d^{\mueps}_N(U_1,U_3)\left[1+\exp\left({\beta\Psi_0(U_1)}\right)+\exp\left({\beta\Psi_0(U_2)}\right)\right]+d^{\mueps}_N(U_2,U_3)\left[1+\exp\left({\beta\Psi_0(U_2)}\right)+\exp\left({\beta\Psi_0(U_3)}\right)\right]\\
&\ge  1+\exp\exp\left({\frac{\beta}{2}\Psi_0(U_3)}\right).
\end{align*}
Also, since $d^{\mueps}_N$ is bounded by 1, 
\begin{align*}
\big{|}d^{\mueps}_{N,\beta}(U_1,U_3)\big{|}^2 \le 2\left[1+2\exp\left({\frac{\beta}{2}\Psi_0(U_3)}\right)\right].
\end{align*}
As a consequence, we obtain
\begin{align*}
\big{|}d^{\mueps}_{N,\beta}(U_1,U_3)\big{|}^2 \le 4\Big(\big{|}d^{\mueps}_{N,\beta}(U_1,U_2)\big{|}^2+\big{|}d^{\mueps}_{N,\beta}(U_2,U_3)\big{|}^2\Big).
\end{align*}

Lastly, we combine two cases above to produce~\eqref{ineq:d_epsilon:triangle-ineq}. The proof is thus complete.
\end{proof}

Lastly, in \cref{lem:W_0:triangle-ineq} below, we establish a similar generalized triangle inequality for $\W_{d_{N,\beta}}$. Since the proof of \cref{lem:W_0:triangle-ineq} is similarly to that of \cref{lem:W_epsilon:triangle-ineq}, it is omitted.

\begin{lemma} \label{lem:W_0:triangle-ineq}
For all $N,\beta>0$, $\nu_1,\nu_2,\nu_3\in \Pcal r(H)$, the following holds
\begin{equation} \label{ineq:W_0:triangle-ineq}
\W_{d_{N,\beta/2}}(\nu_1,\nu_3)\le C\Big(\W_{d_{N,\beta}}(\nu_1,\nu_2)+\W_{d_{N,\beta}}(\nu_2,\nu_3)\Big),
\end{equation}
for some positive constant $C=C(N,\beta)$ independent of $\nu_1,\nu_2,\nu_3$. Here, $\W_{d_{N,\beta}}$ is the Wasserstein distance associated with $d_{N,\beta}$ defined in~\eqref{form:d_N,beta}.
\end{lemma}

\section{Auxiliary results}\label{sec:auxiliary-result}

\begin{lemma}\label{lem:tech:1}
Given $c,c'>0$, the following holds 
\begin{equation}\notag
{e^{-ct}\int_0^t e^{-(c'-c)r}\emph{d} r\le Ce^{-ct},}
\end{equation}
{for all $t\geq0$.}
\end{lemma}
\begin{proof}
There are three cases depending on the sign of $c'-c$. If $c'-c>0$, then
\begin{align*}
e^{-ct}\int_0^t e^{-(c'-c)r}\d r\le \frac{e^{-ct}}{c'-c}.
\end{align*}
Otherwise, if $c'-c< 0$
\begin{align*}
e^{-ct}\int_0^t e^{-(c'-c)r}\d r= e^{-c't}\int_0^t e^{-(c-c')(t-r)}\d r\le \frac{e^{-c't}}{c-c'}.
\end{align*}
Now if $c'-c=0$,
\begin{align*}
e^{-ct}\int_0^t e^{-(c'-c)r}\d r= e^{-ct}t\le \frac{2e^{-ct/2}}{c}.
\end{align*}
Altogether, we observe that
\begin{equation}\notag
e^{-ct}\int_0^t e^{-(c'-c)r}\d r\le Ce^{-ct}.
\end{equation}

\end{proof}

\begin{lemma} \label{lem:|x^epsilon|<R}
Assume that \nameref{cond:mu} holds. Let $u_0(\cdot)\in C((-\infty,0];H^1)$ satisfy
\begin{align}\label{lem:hypo}
\sup_{r\geq r_0}e^{-\delta_0r}\|A^{1/2}u_0(-r)\|_H^2\le c_0e^{\delta_0 r},
\end{align}
for some $c_0,\delta_0,\,r_0>0$. Let
\begin{align*}
\eta_0(s):=\int_0^s u_0(-r)\d r,
\end{align*}
for $s\geq0$. Then 
\begin{align*}
\limsup_{\varepsilon\to 0}\|\eta_0\|^2_{\Moneeps}<\infty.
\end{align*}
\end{lemma}

\begin{proof}
By the Cauchy-Schwarz inequality and the \eqref{lem:hypo}, we have
\begin{align*}
\|A\eta_0(s)\|^2_H&= \int_0^s \int_0^s\la Au_0(-r),Au_0(-\tilde{r})\ra_H\d r\d \tilde{r}\\
&\le \int_0^s\int_0^s\|Au_0(-r)\|_H\|Au_0(-\tilde{r})\|_H\d r\d \tilde{r}\\
&\le s\int_0^s \|Au_0(-r)\|^2_H\d r\\
&\le s^2\sup_{r\in[0,r_0]}\|Au_0(-r)\|^2_H+ \frac{c_0}{\delta_0} s\,e^{\delta_0 s}\boldsymbol{1}\{s\ge r_0\}.
\end{align*}
By~\eqref{form:mu_epsilon} and \nameref{cond:mu}, it follows that
\begin{align*}
\|\eta_0\|^2_{\Honeeps}&=\int_0^\infty\close  \mueps(s)\|A\eta_0(s)\|^2_H\d s\\
&\le \frac{\mu(0)}{\varepsilon^2}\int_0^\infty\close e^{-\frac{\delta}{\varepsilon}s}  \Big[s^2\sup_{r\in[0,r_0]}\|Au_0(-r)\|^2_H+ \frac{c_0}{\delta_0} s\,e^{\delta_0 s}\boldsymbol{1}\{s\ge r_0\}\Big]\d s\\
&= \varepsilon\mu(0) \sup_{r\in[0,r_0]}\|Au_0(-r)\|^2_H\int_0^\infty\close e^{-\delta s} s^2\d s +\frac{\mu(0)c_0}{\delta_0}\int_{r_0/\varepsilon}^{\infty} \close e^{-(\delta-\varepsilon\delta_0)s}s\, \d s.
\end{align*}
In the last implication above, we made a change of variable. It follows that 
\begin{align*}
\limsup_{\varepsilon\to 0}\|\eta_0\|^2_{\Moneeps}<\infty,
\end{align*}
as claimed.
\end{proof}

\bibliographystyle{abbrv}
{\footnotesize\bibliography{react-diff}}

\end{document}